\documentclass[dvips,bernouilli]{imsart}
\RequirePackage[OT1]{fontenc}
\RequirePackage{amsthm,amsmath,natbib}
\RequirePackage[colorlinks]{hyperref}
\usepackage{amsfonts}
\usepackage{amsbsy}
\usepackage{subfigure}
\usepackage{wrapfig}
\usepackage{amssymb}
\usepackage{latexsym}
\usepackage{euscript}
\usepackage{fancyhdr}
\usepackage{epsfig}
\usepackage{floatflt}
\usepackage{graphicx}
\usepackage{subfigure}
\usepackage{fancybox}
\usepackage{varioref}
\usepackage{enumerate}
\usepackage{verbatim}
\usepackage{makeidx}
\usepackage[english]{babel}
\usepackage[applemac]{inputenc}

\newcommand{\R}{\mathbb{R}}
\newcommand{\E}{\mathbb{E}}
\newcommand{\N}{\mathbb{N}}
\newcommand{\1}{{\sf \hspace*{0.9ex}}\rule{0.15ex}{1.6ex}\hspace*{-1ex} 1}

\DeclareMathOperator{\Arginf}{Arginf}

\DeclareMathOperator{\sign}{sign}

\DeclareMathOperator{\X}{\mathcal{X}}

\arxiv{math.ST/0806.0729}

\startlocaldefs
\theoremstyle{plain}
\newtheorem{Lemme}{Lemma}[section]
\newtheorem{proposition}{Proposition}[section]
 \newtheorem{remarque}{Remark}[section]
 \newtheorem{theoreme}{Theorem}[section]
 
 \newtheorem{exemple}{Example}[section]
 \newtheorem{corollaire}{Corollary}[section]
 \newtheorem{question}{Problem}
\endlocaldefs

\begin{document}

\begin{frontmatter}
\title{High dimensional gaussian classification\protect}
\runtitle{High dimensional gaussian classification}

\begin{aug}
\author{\fnms{Robin} \snm{Girard}\thanksref{t1}\ead[label=e1]{Robin.Girard@imag.fr}}

\address{LJK, Grenoble, France\\
}

\runauthor{R. Girard}
\affiliation{Universit\'e Joseph Fourier}
\end{aug}
\begin{abstract} 
  High dimensional data analysis is known to be as a challenging problem (see \cite{Donoho:2000yx}). In this article, we give a theoretical analysis of high dimensional classification of Gaussian data which relies on a geometrical analysis of the error measure. It links a problem of classification with a problem of nonparametric regression. We give an algorithm designed for high dimensional data which appears straightforward in the light of our theoretical work, together with the thresholding estimation theory. We finally attempt to give a general treatment of the problem that can be extended to frameworks other than gaussian.
\end{abstract}



\tableofcontents
\end{frontmatter}

\section{Introduction}

Let $\X$ be a vector space, typically $\X=\R^p$ but $\X$ can also be an infinite dimensional polish space (i.e: separable complete metric space). In Section \ref{proofth2} $\X$ is a separable Banach space. In the binary classification problem, the aim is to recover the unknown class $y\in \{0,1\}$ associated with an observation $x\in X$. In other words, we seek a classification rule (also called classifier), i.e a measurable $g:\X \rightarrow \{0,1\}$. This rule gives an incorrect classification for the observation $x$ if $g(x)\neq y$. The underlying probabilistic model, that makes a performance measure of $g$ possible, is set by distributions $P_k$ ($k=0,1$) on $\X$. For $k=0,1$, the distribution $P_k$ is the distribution of the data having label equal to $k$. In this framework, the weighted sum of the probabilities of misclassification is defined by 
\begin{equation}\label{intro}
\mathcal{C}(\pi,g)=\pi P_1(g(X)\neq 1)+(1-\pi)P_0(g(X)\neq 0).
\end{equation}
\indent In a bayesian framework, the weight $\pi$ reflects the marginal distribution of the label $Y$. In our approach, we do not want this marginal distribution to set the importance of the different errors. In the many applications we have in mind, such as tumour detection from an MRI signal, the class that appears most frequently is not necessarily the one for which a classification error has the most important medical consequences. This is the reason why we search a procedure $g$ that minimise $\mathcal{C}(\pi,g)$ and not its bayesian counterpart : $P(g(X)=Y)$.\\
\indent Here, we do not want to study the influence of the weight $\pi$ in the problem. The main reason is that our results, to be given later, are simpler to formulate and to understand when $\pi=1/2$, and that the problem we are interested in is  the problem that rise from the high dimension of the space $\X$, and not the problem related to the use of $\pi$. Therefore, in the rest of the present paper we will make the assumption that $\pi=1/2$. In the sequel, we will set $\mathcal{C}(g)=\mathcal{C}(1/2,g)$. This is a usual assumption (see for example Bickel and Levina \cite{bickel:2004fk})\\

\indent In the case where $\pi=1/2$ it is known that, if $P_0$ and $P_1$ are equivalent, then the rule that minimises $\mathcal{C}(g)$ is given by 
\begin{equation}\label{optimal}
g^*(x)=\1_{V}, \;\;\; V=\{x\in \X\;:\; \mathcal{L}_{10}(x)\geq 0\} \;\; \text{ where }\mathcal{L}_{10}=\log\left (\frac{dP_1}{dP_0}\right )
\end{equation}
is the logarithm of the likekihood ratio between $P_1$ and $P_0$ (i.e the Radon-Nikodym derivative).\\

 In real life problems, $\mathcal{L}_{10}$ is unknown, and the only thing we have is a substitute $\widehat{\mathcal{L}}_{10}$ of it. Also, it is natural to plug it in (\ref{optimal}) and to use the classifier 
\[g(x)=\1_{\hat{V}}(x)\text{ and }\hat{V}=\left \{x\in \X\;:\; \widehat{\mathcal{L}}_{10}\geq 0  \right \}. \]
The natural question that we will investigate in this article is the following: 
\begin{question}\label{Pb1}
Is there a simple way to relate the excess risk $\mathcal{C}(g)-\mathcal{C}(g^*)$ to a measure of the log-likelihood "perturbation": $\widehat{\mathcal{L}}_{10}-\mathcal{L}_{10}$.
\end{question}
In other words we seek an upper bound and a lower bound of $\mathcal{C}(g)-\mathcal{C}(g^*)$ by a simple-to-study real valued function of $\widehat{\mathcal{L}}_{10}-\mathcal{L}_{10}$.
In this article we focus on the gaussian case, and unless the contrary is explicitly stated, $P_1$ and $P_0$ will be gaussian equivalent probabilities on $\X$. We investigate Problem \ref{Pb1} and the answer we obtain in the general case leads to the bound 
\[\mathcal{C}(g)-\mathcal{C}(g^*)\leq c(r)\|\widehat{\mathcal{L}}_{10}-\mathcal{L}_{10}\|^{1/6}_{L_2(\gamma)}\]
while $\|\mathcal{L}_{10}\|_{L_2(\gamma)}\geq r>0$ for a gaussian measure $\gamma$, where $c(r)$ is a constant only depending on $r$. In some particular cases (when $\widehat{\mathcal{L}}_{10}-\mathcal{L}_{10}$ and $\mathcal{L}_{10}$ are affine) we are able to give an explicit constant $c(\mathcal{L}_{10})$ and an exponent higher than $1/6$ (exponent $1$).\\
\indent  If we suppose that $P_0$ and $P_1$ have equal covariance, then it is known that $\mathcal{L}_{10}$ is affine and it is  natural to take an affine $\widehat{\mathcal{L}}_{10}$. The corresponding procedure is usually called Linear Discriminant Analysis (LDA) (even if the underlying procedure is affine). If we suppose that $P_0$ and $P_1$ have different covariance, then $\mathcal{L}_{10}$ is quadratic and it is natural to take a quadratic $\widehat{\mathcal{L}}_{10}$. The corresponding classification procedure will be called Quadratic Discriminant Analysis (QDA).\\
 \indent The corresponding procedures are also known as plug-in procedures: $\widehat{\mathcal{L}}_{10}$ is plugged into (\ref{optimal}) in order to obtain $g$. Plug-in procedure have been studied in a different context (see for example \cite{Audibert:2006fk} and the references therein), but our approach differs from those. \\
 
 The interest of Problem \ref{Pb1} in the gaussian setting, is understood by addressing the problem of finding a good substitute $\widehat{\mathcal{L}}_{10}$ for $\mathcal{L}_{10}$. For example, in many applications, we are given a learning set consisting of $n$ random variables drawn independently from $P_1$ and $n'$ drawn from $P_0$. The problem of finding a good substitute $\widehat{\mathcal{L}}_{10}$ of $\mathcal{L}_{10}$ then becomes an estimation problem whose error measure is given in the answer to Problem \ref{Pb1}. Also, our answer to Problem 1 given below gives rise to a natural way to estimate $\mathcal{L}_{10}$ in high dimension, which is the answer to what we call Problem \ref{Pb2}:
 \begin{question}\label{Pb2}
Given a learning set, construct $\widehat{\mathcal{L}}_{10}$ in order to get a satisfactory classification procedure in high dimension: a procedure that can be justified theoretically and with numerical experiment.
 \end{question}
 \indent Classical methods of classification break down when the dimensionality is extremely large. For example. Bickel and Levina \cite{bickel:2004fk} have studied the poor performances of Fisher discriminant  analysis. Although, the number of parameters to learn in order to build a classification rule seems to be responsible for the poor performance. In the sequel we shall give theoretical non-asymptotic results that emphasise this poor performances. To overcome the poor performance Bickel and Levina \cite{bickel:2004fk} propose to use a rule which relies on feature independence, Fan and Fan \cite{Fans:2007vn} propose to select the interesting features with a multiple testing procedure. Bickel and Levina give a theoretical study of a particular LDA procedure (i.e a LDA procedure based on a particular estimator $\widehat{\mathcal{L}}_{10}$), they do not study the QDA procedure. \\
\indent The selection of interesting features constitutes a reduction of the dimension of the space on which the classification rule acts. Feature selection is widely used in high dimensional classification, the procedures used for selection of interesting features are often motivated by theoretical results (see \cite{Fans:2007vn}). Unfortunately, these theoretical results are based on the following two postulates. On the one hand, features can be a priori divided into two parts, an interesting one and a non interesting one. On the other hand, selecting the interesting features is necessary and sufficient to get a good classification rule. If we accept that these postulates reflect nothing but a relatively clear intuition, we would like to give an analysis of the classification risk in order to justify a feature selection method based on multiple hypothesis testing.\\
\indent  Thresholding techniques are widely used in the non-parametric regression framework (see \cite{Candes:2006lk} for an introduction to the thresholding techniques), and as we shall see, the techniques can be used to give an answer to Problem  \ref{Pb2}. Also we believe that our answer to Problem \ref{Pb1} will shed light on the simple link that exists between the nonparametric regression and the classification problem.\\
\indent Functional data analysis is the study of data that lives in an infinite dimensional functional space. Hence curve classification is one of the problems it deals with. Since \cite{Grenander:1950fk}, functional data analysis has undergone further developments and especially in the context of classification (see for example \cite{Berlinet:2005at} and the references therein). In the gaussian setting, it is rather natural to expect results that are dimensionless and that can be applied to any abstract polish space. Hence, our answer to problem $1$ will be given in terms of $L_2(\gamma)$ norms, with $\gamma$ a gaussian measure, and since the constant involved in our theoretical result does not depend on the dimension, the extension from $\X=\R^p$ to more abstract spaces is straightforward. \\

Let us introduce some notation. In the whole article, $\gamma_{C,\mu}$ is a gaussian measure on $\X$ with mean $\mu$ and covariance $C$, $\gamma_{C}$ is the zero mean gaussian measure with covariance $C$ and $\gamma_{p}$ is the gaussian measure on $\R^p$ with mean zero and covariance $Id_{\R^p}$; $\Phi(x)$ is the cumulative distribution function of a real gaussian random variable with mean zero and variance one. If $\gamma$ is a probability measure on $\R^p$, $\|\Pi_{x}^{\bot}e\|_{L_2(\gamma)}$ will be the norm of the orthogonal projection in $L_2(\gamma)$ of the vector $e\in L_2(\gamma)$ on the hyper-plan orthogonal to $x\in L_2(\gamma)$; if $F\in \R^p$ $\|F\|_{L_2(\gamma)}$ will be the norm of the linear application $x\in \R^p\rightarrow \langle F,x \rangle_{\R^p}$. We shall use both the fact that if $F\in \R^p$ and $\gamma$ is a gaussian measure with mean zero and covariance $C$, then $\|F\|_{L_2(\gamma)}=\|C^{1/2}F\|_{\R^p}$; and that $\|F\|_{L_2(\gamma)}$ is a natural measure that can be extended in an infinite dimensional framework. The symmetric difference between two subsets of $\X$ $A$ and $B$ is denoted by $A\Delta B$, it is the set of all elements that are in $A\setminus B$ or in $B\setminus A$. If $A$ is a matrix of $\R^p$ $\|A\|_{HS}$ will be the Hilbert-Schmidt norm of the matrix $A$, $trace(A)$ the trace of $A$, and $q_{A}(x)$ will be given by $\langle Ax,x \rangle_{\R^p}$ for all $x\in \R^p$.\\

This article is organized as follows.  We give the main theoretical results -leading to a solution to Problem \ref{Pb1}- for the LDA procedure in Section \ref{linea}, and for the QDA procedure in Section \ref{quadrat}. In section \ref{procetude} we give our algorithm for high dimensional data classification and the theoretical result related to it. This leads to our contribution to Problem \ref{Pb2} in the light of our solution to Problem \ref{Pb1}. In Section \ref{application} we apply this algorithm to curve classification. In Section \ref{errapp} we introduce a geometric measure of error and derive its link with the excess risk. Section \ref{proofth1} is devoted to the proof of results given in Section \ref{linea} and Section \ref{proofth2}, to the proof of results given in Section \ref{quadrat} and possible generalisations.  

\section{Affine perturbation of affine rules}\label{linea}
\subsection{An solution to Problem \ref{Pb1}}
\subsubsection{Main result}
In this section, $\X=\R^p$, $C$ is a symmetric definite positive matrix and  $P_1=\gamma_{\mu_1,C}$ $P_0=\gamma_{\mu_0,C}$.  Under these hypotheses $\mathcal{L}_{10}(x)=\mathcal{L}^A_{10}(x)$ is affine on $\R^p$:
\begin{equation}\label{eque:RN}
\mathcal{L}^A_{10}(x)=\langle F_{10},x-s_{10}\rangle_{\R^p}\text{ where }s_{10}=\frac{\mu_1+\mu_0}{2}, \;F_{10}=C^{-1}m_{10} 
\end{equation}
and $m_{10}=\mu_1-\mu_0$. In this section, we restrict ourselves to an affine substitute $\widehat{\mathcal{L}}^A_{10}(x)$, we note $\hat{F}_{10}$ and $\hat{s}_{10}$ the corresponding substitutes of $F_{10}$ and $s_{10}$. We then decide that $X$ comes from $P_1$ if it is in    
\begin{equation}\label{e:baye}
\hat{V}=\left \{x\in \R^p \text{ st } \widehat{\mathcal{L}}^A_{10}(x)\geq 0\right \}.
\end{equation}

One can define the angle $\alpha$ in $L_2(\gamma_{C})$ between $F_{10}$ and $\hat{F}_{10}$ by 
 \begin{equation}\label{alpha}
\alpha=\arctan \left (\frac{\|\Pi_{F_{10}^{\bot}}\hat{F}_{10}\|_{L_2(\gamma_C)}\|F_{10}\|_{L_2(\gamma_C)}}{\langle \hat{F}_{10},F_{10}\rangle_{L_2(\gamma_C)}}\right ). 
\end{equation}
This angle will play a very important role in the sequel. We obtained the following solution to Problem \ref{Pb1}. 
\begin{theoreme}\label{the:32}
Let $\hat{F}_{10}$ and $\hat{s}_{10}$ be two $\R^p$ vectors and $\widehat{\mathcal{L}}^A_{10}(x)$ defined by substituting $\hat{F}_{10}$ and $\hat{s}_{10}$ for $F_{10}$ and $s_{10}$ in (\ref{eque:RN}). Let $P_1$ and $P_0$ be two gaussian measures on $\X=\R^p$ with the same covariance $C$ with means respectively $\mu_1$ and $\mu_0$.\\
\indent If $\hat{V}$ is the $\R^p$ subset defined by  (\ref{e:baye}), we have:

\[
\mathcal{C}(\1_{\hat{V}})-\mathcal{C}(\1_{V})\leq \frac{\mathcal{E}}{\|F_{10}\|_{L_2(\gamma_C)}}
\]
where
\begin{equation}\label{eq:er}
\mathcal{E}=\left ( \frac{4\|F_{10}\|_{L_2(\gamma_C)}}{\sqrt{\pi}\|\hat{F}_{10}\|_{L_2(\gamma_C)}}|\langle \hat{F}_{10},\hat{s}_{10}-s_{10}\rangle_{\R^p}|+\|F_{10}-\hat{F}_{10}\|_{L_2(\gamma_{C})}\right ).
\end{equation} 
\\
\indent If $|\langle \hat{F}_{10},\hat{s}_{10}-s_{10}\rangle_{\R^p}|\leq \frac{1}{4}|\langle\hat{F}_{10},F_{10}\rangle_{L_2(\gamma_C)}|$ and $\alpha\leq \pi/4$  ($\alpha$ is defined by (\ref{alpha})), then 
\begin{equation}\label{eq:er2}
\mathcal{C}(\1_{\hat{V}})-\mathcal{C}(\1_{V})\leq e^{-\frac{\|F_{10}\|_{L_2(\gamma_C)}^2}{32}} \frac{\mathcal{E}}{\|F_{10}\|_{L_2(\gamma_C)}}.
\end{equation}
\end{theoreme}

The proof of this theorem is given in Section \ref{proofth1} at Sub-section \ref{dthe32}. It is a consequence of  Theorem \ref{the:1} obtained by simple geometric methods emphasizing the fact that $P_0(X\in V \setminus \hat{V})$ is the measure of an area between two hyperplans obtained by a rotation of angle $\alpha$. The proof also uses the inequality
\begin{equation}\label{errapprr}
\mathcal{C}(\1_{\hat{V}})-\mathcal{C}(\1_{V})\leq \frac{1}{2}\left (P_1(X\in V\setminus \hat{V})+P_0(X\in \hat{V} \setminus V)\right )=\mathcal{R}(\1_{\hat{V}}),
\end{equation}
which defines $\mathcal{R}(\1_{\hat{V}})$. We call $\mathcal{R}(\1_{\hat{V}})$ the learning error, it is the probability of making a a wrong classification with $g(x)=\1_{\hat{V}}(x)$ and a good classification with the optimal rule $g^*=\1_{V}$. We will use and motivate more deeply this measure of error in Section \ref{errapp}. Let us now give comments on Theorem \ref{the:32}.

\subsubsection{General comments} 
 If we note 
\begin{equation}\label{nottt}
\delta=\hat{F}_{10}-F_{10} \text{ and } d_0=\langle \hat{F}_{10},s_{10}-\hat{s}_{10} \rangle_{\R^p},
\end{equation}
we have 
\[\hat{\mathcal{L}}_{10}(x)=\mathcal{L}_{10}(x)+\langle \delta,x-s_{10}\rangle_{\R^p}+d_0.\]
Also, in the sequel we will talk about affine perturbation of the optimal rule. The preceding theorem results from the study of affine perturbations of affine rules. \\
The case where $d_0=0$ will be studied later but we can already note that in this case, Theorem \ref{the:32} yields 
\[\mathcal{C}(\1_{\hat{V}})-\mathcal{C}(\1_{V})\leq \frac{\|\mathcal{L}_{10}-\widehat{\mathcal{L}}_{10}\|_{L_2(\gamma_{C,s_{10}})}}{\|\mathcal{L}_{10}\|_{L_2(\gamma_{C,s_{10}})}},\]
which is a nice answer to Problem \ref{Pb1}. In the sequel (see Section \ref{proofth1} Theorem \ref{the:1}), we shall see that it is optimal whenever $\|\mathcal{L}_{10}\|_{L_2(\gamma_{C,s_{10}})}$ does not become to large. \\

The quantity $r=\|F_{10}\|_{L_2(\gamma_C)}$ measures the theoretical separation of the data. Indeed it is the $L_1$ distance between $P_1$ and $P_0$, defined by $d_1(P_1,P_0)=\int |dP_1-dP_0|$ that measures this separation: it is known that 
$d_1(P_1,P_0)=(1-2\mathcal{C}(\1_{V}))$, which implies
 \[d_1(P_1,P_0)=\Phi\left (-\frac{1}{2}r\right )-\Phi\left (\frac{1}{2}r\right ).\]
  Also, $d_1(P_1,P_0)\sim r  \text{ when }r\rightarrow 0$, and then the data cannot be distinguished by any rule. The data tends to be perfectly separated when $d_1(P_1,P_0) \rightarrow 1$. In this case, $r\rightarrow \infty$ and  
 \[d_1(P_1,P_0)\sim 1-\frac{2e^{-\frac{r^2}{8}}}{r\sqrt{2\pi}}.\]
 Also note that in the infinite dimensional setting two gaussian measures $P_0$ and $P_1$ are either orthogonal (there exists a Borelian set $A$ such that $P_1(A)=P_0(\X\setminus A)=0$ ) or equivalent (i.e mutually absolutely continuous) and the latter case appears if and only if $r$ is finite.\\
 
 Although, if $\mathcal{E}$ measures the estimation error, 
\begin{equation}\label{termehyp}
\frac{1}{\|F_{10}\|_{L_2(\gamma_{C})}}\;\;\;\text{ and }\;\;\;e^{-\frac{\|F_{10}\|_{L_2(\gamma_C)}^2}{32}}
\end{equation} 
in the upper bounds (\ref{eq:er}) and  (\ref{eq:er2}), are linked with the proximity of the measures $P_0$ and $P_1$. When $\|F_{10}\|_{L_2(\gamma_C)}^2$ is large, data are well separated and  the terms in (\ref{termehyp}) measure the impact of this separation on the excess risk. We believe that when $\|F_{10}\|_{L_2(\gamma_C)}^2$ tends to $0$, $\frac{1}{\|F_{10}\|_{L_2(\gamma_C)}}$ is linked to the error measure $\mathcal{R}(\1_{V})$ used in the proof (defined by (\ref{errapprr})). Indeed, it is not correct to think that the classification problem is harder (in the sense of the excess risk) when data are not well separated: straightforward computation leads to 
\[\forall \tilde{V}\subset \R^p\;\;\; \mathcal{C}(\1_{\tilde{V}})-\mathcal{C}(g^*)\leq \frac{1}{2}d_1(P_1,P_0).\]
As we shall see in the sequel (see Theorem \ref{theorem-recip}) $\mathcal{R}(\1_{V})$ behaves almost like the excess risk if and only if $d_1(P_0,P_1)$ does not tend to $0$. \\

The learning set has to be used to elaborate estimators $\hat{F}_{10}$ and $\hat{s}_{10}$ of $F_{10}$ and $s_{10}$. The preceding theorem allows us to quantify what intuition clearly indicates: a good estimation of the parameters $F_{10}$ and $s_{10}$ (or more indirectly $\mu_1,\mu_0$ and $C$) leads to a good classification rule. These estimators must lead to a small excess risk and by the preceding theorem
\begin{equation}
\E_{P^{\otimes n}}[\mathcal{C}(\1_{\hat{V}})-\mathcal{C}(\1_{V})]\leq \frac{\E_{P^{\otimes n}}[\mathcal{E}]}{\|F_{10}\|_{L_2(\gamma_C)}},
\end{equation}
where $P^{\otimes n}$ is the learning set distribution.\\
\indent It seems that little is known on theoretical behaviour of the LDA procedure (a plug-in procedure) with respect to the optimal rule (the Bayes rule). The result that is classically used (see for example Anderson and Bahadur \cite{Anderson:1962fk}) to show the consistency of a LDA rule using estimators $\hat{F}_{10}=\widehat{C^{-1}}\hat{m}_{10}=\widehat{C^{-1}}(\hat{\mu}_1-\hat{\mu}_0)$ and $\hat{s}_{10}=(\hat{\mu}_1+\hat{\mu}_0)/2$ is that the probability to observe $X\leadsto \gamma_{C,\mu_0}$ (in that case $X$ comes from class $0$) falling into $\hat{V}$ (and affect it to class $1$) is  
\begin{equation}\label{levina}
P\left (\langle\hat{F}_{10},C^{1/2}\xi\rangle_{\R^p}\geq \langle \hat{s}_{10}-\mu_0, \hat{F}_{10}\rangle_{\R^p} |\mathcal{A}\right )=1-\Phi\left (\frac{\langle \hat{ s}_{10}-\mu_0, \hat{F}_{10}\rangle_{\R^p}}{\|\hat{F}_{10}\|_{L_2(\R^p,\gamma_C)}} \right ),
\end{equation}
where $\mathcal{A}$ is the $\sigma$-field generated by the learning set, and $\xi$ is a centered gaussian random vector of $\R^p$ with covariance $Id_{\R^p}$. Note that the proof of (\ref{levina}) follows from a straightforward calculation. We believe that a direct analysis of this error term misses the geometrical aspect of the problem. In addition, this error has to be compared with the lowest possible error $\mathcal{C}(g^*)$. Note that for the LDA procedure in a high dimensional framework, an analysis of the worst case excess risk has been done with (\ref{levina}) by Bickel and Levina \cite{bickel:2004fk} for a particular choice of $\hat{F}_{10}$ and $\hat{s}_{10}$. Our Theorem, because it is intrinsic to the classification procedure, is singularly different from the type of result that they obtain. In particular, it will allow us to establish a revealing link between dimensionality reduction and thresholding estimation.

\subsubsection{The constant part of the perturbation}
The error due to the constant part of the perturbation ($d_0$ in equation (\ref{nottt})), is measured by 
\[\frac{4}{\sqrt{\pi}}\left |\left \langle \frac{\hat{F}_{10}}{\|\hat{F}_{10}\|_{L_2(\gamma)}},\hat{s}_{10}-s_{10}\right \rangle_{\R^p}\right |.\]
In order to give a first simple analysis of this term, we are going to suppose that $\hat{F}_{10}$ and $\hat{s}_{10}$ are independent. This independence can be obtained by keeping a part of the learning set for the estimation of  $F_{10}$ and a part for the estimation of $s_{10}$. In thisat case, if $n'$ observations of the learning set were used to construct $\hat{s}_{10}$, and if $\hat{s}_{10}=(\bar{\mu}_1+\bar{\mu}_0)/2$ ($\bar{\mu}_i$ is the empirical mean of the observations of group $i$), then, straightforward calculation leads to
\[\E_{P^{\otimes n}}\left [\frac{4}{\sqrt{\pi}\|\hat{F}_{10}\|_{L_2(\gamma)}}|\langle \hat{F}_{10},\hat{s}_{10}-s_{10} \rangle_{\R^p}|\right ]\leq \frac{8}{\sqrt{2n'}\pi}.\]
Ultimately, the difficulty of the problem does not come from the constant part of the perturbation, but from the linear part.\\

The conditions under which the second inequality (\ref{eq:er2}) of the theorem is given shall easily be satisfied. The second condition  is that $\alpha\leq \frac{\pi}{4}$. It is not difficult to satisfy if $\hat{F}_{10}$ and $F_{10}$ are close enough to each other. The first one is verified if the second is and if we have:
\[\left |\left \langle \frac{\hat{F}_{10}}{\|\hat{F}_{10}\|_{L_2(\gamma_C)}}, s_{10}-\hat{s}_{10}\right \rangle_{\R^p} \right |\leq \frac{\sqrt{2}}{8} \|F_{10}\|_{L_2(\gamma_C)}.\]
If for example $\hat{s}_{10}=(\bar{\mu}_1+\bar{\mu}_0)/2$ and the learning set is composed of $n'$ observations uniquely used for the estimation of $s_{10}$, then, given the rest of the learning set, $\langle \frac{\hat{F}_{10}}{\|\hat{F}_{10}\|_{L_2(\gamma_C)}}, s_{10}-\hat{s}_{10}\rangle_{\R^p} \leadsto\gamma_{\frac{1}{n'}}$ and the preceding condition is satisfied with probability 
\[\frac{1}{2}\Phi\left (\frac{\sqrt{2}}{8} \|F_{10}\|_{L_2(\gamma_C)}n'\right ).\]

\subsubsection{The linear part of the perturbation}
As we shall explain in the proof of Theorem \ref{the:32}, the angle $\alpha$ defined by (\ref{alpha}) measures quite well the error due to the linear part of the perturbation. Also, the upper bound given in the preceding theorem is not sharp everywhere. Indeed, if $\beta\in \R$, and $\hat{F}_{10}=\beta F_{10}$, the error $\mathcal{R}(\1_{V})$ is null and the bound (\ref{eq:er}) can be arbitrarily large. We believe that the study of methods designed to estimate direction (parameter on the sphere $\mathbb{S}^{p-1}$) in a high dimensional setting are required. We only want to give the link between the problem of estimating $F_{10}$ as a vector of $\R^p$ and the problem of estimating $F_{10}$ in order to get small $\mathcal{C}(\1_{\hat{V}})$. In addition, this invariance of the error under dilatation only exists in the direction $F_{10}$ which is unknown and is seems to be quite tricky to make a direct use of it.\\
\indent Let us give a simple example to illustrate the interest of the link between estimation and learning.
\begin{exemple}\label{ex1}
 Let $\sigma>0$, suppose  $X\leadsto \gamma_{\frac{1}{n} I_p,F_{10}}$, $C=I_p$ and that $s_{10}$ is known. In the estimation problem of $F_{10}$ for classification we wish to recover $F_{10}$ from the observation $X$ and the error is measured by
 \[
 \mathcal{R}(\1_{\hat{V}})\leq \frac{\|F_{10}-\hat{F}_{10}\|_{L_2(\gamma_{C})}}{\|F_{10}\|_{L_2(\gamma_{C})}}=\frac{\|\hat{F}_{10}-F_{10}\|_{\R^p}}{\|F_{10}\|_{\R^p}}.
\]
\end{exemple}

  In Example \ref{ex1} the problem is exactly the one we encounter in the regression framework, while estimating $F_{10}$ from $p$ noisy observations of $(F_{10}[i])_{i=1,\dots,p}$ with an error measured with a $l^2$ norm. Suppose now that we want to let $p$ grow to infinity. If the coefficients of $F_{10}$ decrease sufficiently fast, for example if $F_{10}\in l^q(R)$ with $q<2$, then (see for example \cite{Candes:2006lk}), it is possible to obtain a good statistical estimation of $F_{10}$ by setting to zero the coefficient that are are, in absolute value, under a threshold. It is a thresholding estimation and we shall use this type of procedure in Section \ref{procetude}. In the case where we observe $X$ from the distribution $\gamma_{C/n,m_{10}}$ (or equivalently $X^i$, $i=0,1,$ from the distribution $\gamma_{2C/n,\mu_i}$) and if $C\neq I_p$ is known, the problem can be reduced to the preceding particular case thanks to the transformation $x\rightarrow C^{-1/2}x$. When $C$ is unknown, the parallel with the estimation framework is more delicate because the error $\mathcal{E}$ depends on $C$.

 \begin{remarque}
Replacing coefficients by zero in the regression framework of Example \ref{ex1} is equivalent to reducing the dimension of the space on which the chosen classification rule acts. Selecting the significant coefficients of $F_{10}$ is equivalent to finding the direction $e_{i}\in \R^p$ for which $|\langle C^{-1/2}(\mu_1-\mu_0),e_i\rangle_{\R^p}|^2$ is large. This is almost equivalent to finding the direction in which a theoretical version of the ratio between inter-variance and intra-variance is big. This type of heuristic with empirical quantities has been used by Fisher \cite{Fisher:1936fk}, whose strategy is to maximize the Rayleigh quotient (see for example \citep{Friedman:2001wq}). The point is that the use of empirical quantities in high dimension can be catastrophic (see next subsection). 
\end{remarque}
\subsection{Procedures to avoid in high dimension}
We are going to give two results that will lead to the following precepts in the problem of estimating $\mathcal{L}_{10}$. While giving a solution to Problem \ref{Pb2}, 
\begin{enumerate}
\item one should not try to estimate the full covariance matrix $C$ from the data,  
\item one should restrict the possible values of $m_{10}$ to a (sufficiently small) subset of $\R^p$.
\end{enumerate}
These precepts have been known for some time, but we give precise non-asymptotic results emphasising them. The first fact is a consequence of Proposition \ref{incons} below while the second one results from Proposition \ref{properrlin}. \\

These two proposition arise from the use of a more geometric error measure, the learning error $\mathcal{R}$, which has already been defined by (\ref{errapprr}) and which shall be studied in more detail in Section \ref{errapp}. In fact it is an easy geometric exercise, for one who knows a little on gaussian measure, to obtain the following lower bound
\begin{equation}\label{lower1}
\mathcal{R}(\1_{\hat{V}})\geq \frac{|\alpha|}{2\pi}e^{-\frac{\|F_{10}\|^2_{L_2(\gamma_C)}}{8}},
\end{equation}
(which is the last point of Theorem \ref{the:1} in Section \ref{proofth1}) where $\alpha$, the angle in $L_2(\gamma_C)$ between $F_{10}$ and $\hat{F}_{10}$, is defined by (\ref{alpha}).
On the other hand, Theorem \ref{theorem-recip} from Section \ref{errapp} leads to 
\[
\mathcal{C}(g)-\mathcal{C}(g^*) \geq\min\left \{ \frac{\sqrt{2\pi}}{2*16^2}\|C^{-1/2}m_{10}\|_{\R^p}e^{\frac{\|C^{-1/2}m_{10}\|_{\R^p}^2}{8}}\mathcal{R}(g)^{2},\frac{\mathcal{R}(g)}{8}\right \},
\]
for all measurable $g:\X\rightarrow \{0,1\}$. Also, it suffices to get a lower bound on the Learning error $\mathcal{R}(\1_{\hat{V}})$ by the use of (\ref{lower1}) to get (a good) lower bound on the excess Risk when $d_1(P_0,P_1)$ cannot be as closed as desired from zero. This is what we shall do. For the case where the distributions $P_1$ and $P_0$ are almost undistinguishable ($d_1(P_1,P_0)\rightarrow 0$) we refer to the discussion in Section \ref{errapp}.
\subsubsection{One should not try to identify the correlation structure}
Let us recall that if $A$ is a definite positive matrix, one can define its generalised inverse, also called Moore-Penrose pseudo-inverse: $C^-$. This generalised inverse $C^-$ arises from the decomposition $\R^p=Ker(C) \oplus Ker(C)^{\bot}$. On $Ker(C)$, $C^-$ is null, and on $Ker(C)^{\bot}$, $C^-$ equals the inverse of $\tilde{C}=C_{|Ker(C)^{\bot}}$ (  i.e $\tilde{C}$ is the restriction of $C$ to $Ker(C)^{\bot}$). 
\begin{proposition}\label{incons}
Suppose we are given $X_1,\dots,X_n$ drawn independently from a gaussian Probability distribution $P$ with mean zero and covariance $C$ on $\R^p$. Let $\hat{C}$ be the empirical covariance and $\hat{C}^-$ its generalised inverse. If $\hat{F}_{10}=\hat{C}^-m_{10}$ and $\hat{s}_{10}=s_{10}$, the classification rule $\1_{\hat{V}}$ defined by (\ref{e:baye}) leads to   
  \[\E_{P^{\otimes n}}[\mathcal{R}(\1_{\hat{V}})]\geq \frac{\arccos \left (\sqrt{\frac{n}{p}}\right )}{2\pi}e^{-\frac{\|F_{10}\|^2_{L_2(\gamma_C)}}{8}}.\]  
  \end{proposition}
 Before we prove this proposition, let us comment it in few words.\\
 \indent {\bf Comment}. As a particular application of this proposition, we see that the Fisher rule performs badly when $p>>n$, which was already given in \cite{bickel:2004fk}, but in a different form (asymptotic and not in a direct comparison of the risk with the Bayes risk). Many alternatives to the estimation of the correlation structure can be used, based for example on approximation theory of covariance operators, together with model selection procedure or more sophisticated aggregation procedure. Much work has already been done in this direction, see for example \cite{Bickel:2007fk} and the references therein. The approximation procedure has to be linked with a statistical hypothesis, as it is in the case when stationarity assumptions are made that lead to a Toeplitz covariance matrix $C$ (i.e $C_{ij}=c(i-j)$ with $c:\mathbb{Z}\rightarrow \R$ a $p$-perioric sequence). These matrices are circular convolution operators and are diagonal in the discrete Fourier Basis $(g^m)_{0\leq m<p}$ where
\[(g^{m})_k=\frac{1}{\sqrt{p}}\exp\left (\frac{2i\pi mk}{p}\right ).\]   
 This is roughly the type of harmonic analysis that is used in Bickel and Levina \cite{bickel:2004fk} and combined with an approximation in \cite{bestbasis}. Under assumption such as commutation (or quasi-commutation) of the covariance with a given family of projections, the covariance matrix can be search in the set of operator given by a spectral density. This leads to a huge reduction of the parameters to estimate. Let us finally notice that the use of harmonic analysis of stationarity in curve classification becomes very interesting when one considers the larger class of group stationary-processes (see \cite{Yazici:2004mz}) or semi-group stationary processes (see \cite{Girardin:2003rt}). \\
\begin{proof}
 The proof is based on ideas from Bickel and Levina \cite{bickel:2004fk} used in their Theorem $1$:
 if $C$ is the identity their exist $\xi_1,\dots,\xi_p$, $p$ $\R^p$ valued random variables forming an orthonormal basis of $\R^p$,  a random vector $(\lambda_1,\dots,\lambda_n)$ of $\R^n$ whose property are the following. 
  \begin{enumerate}
  \item The $\lambda_i$ are independent of each other, independent of $(\xi_i)_{i=1,\dots,p}$, and $n\lambda_i$ follows a $\chi^2$ distribution with $n-1$ degrees of freedom.
  \item For every $i$, $\xi_i$ is drawn in an independent and uniform fashion on the intersection of the unitary sphere of $\R^p$ and the orthogonal to $\xi_{1},\dots,\xi_{i-1}$.
  \item The empirical estimator $\hat{C}$ of $C$ satisfies:  
\[
  \hat{C}=\sum_{i=1}^n\lambda_i\xi_i\otimes \xi_i, 
\]
where if $x,y\in \R^p$, $x\otimes y$ is the linear operator of $\R^p$ that associates to $z\in \R^p$ the vector $\langle x,z \rangle_{\R^p}y$.  
\end{enumerate}

When $C$ does not necessarily equal $I_p$, we get, $\gamma_C-$almost-surely:
\[
C^{-1/2}\hat{C}C^{-1/2}=\sum_{i=1}^n\lambda_i\xi_i\otimes \xi_i, \text{ et }C^{1/2}\hat{C}^{-}C^{1/2}=\sum_{i=1}^n\frac{1}{\lambda_i}\xi_i\otimes \xi_i.
\]
Then, if we define $\beta_i=\langle C^{-1/2}m_{10},\xi_i\rangle_{\R^p}^2$, we have the following equations
\begin{equation}\label{bick1}
\langle C^{-1}m_{10},\hat{C}^{-}m_{10}\rangle_{L_2(\gamma_C)}=\langle C^{-1/2}m_{10},C^{1/2}\hat{C}^{-}C^{1/2}C^{-1/2}m_{10}\rangle_{\R^p}=\sum_{i=1}^n \frac{\beta_i}{\lambda_i},
\end{equation}
\begin{equation}\label{bick2}
\|\hat{F}_{10}\|_{L_2(\gamma_C)}^2=\sum_{i=1}^n\frac{\beta_i}{\lambda_i^2} \text{ et }\|F_{10}\|_{L_2(\gamma_C)}^2=\sum_{i=1}^p\beta_i.
\end{equation}
For reasons of symmetry  (the $\xi_i$ are drawn uniformly on the sphere), we have for all subsets $I_n$ from $\{1,\dots,p\}$ of size $n$ :
\[
u_{I_n,p}=\E\left [\frac{\sum_{i\in I_n}\beta_i}{\sum_{i=1}^p\beta_i}\right ]=\E\left [\frac{\sum_{i=1}^n\beta_i}{\sum_{i=1}^p\beta_i}\right ],
\]
and we obtain 
\begin{equation}\label{bikc1}
u_{I_n,p}=\frac{n}{p}.
\end{equation}
 From equations (\ref{bick1}) and (\ref{bick2}), the expectation of the angle $\alpha$ between $\hat{F}_{10}$ and $F_{10}$ in $L_2(\gamma_C)$ (defined by \ref{alpha}) is
\begin{align*}
\E[|\alpha|]&=\E\left [\arccos\left ( \frac{\sum_{i=1}^n\frac{\beta_i}{\lambda_i}}{\sum_{i=1}^p\beta_i\sum_{i=1}^n\frac{\beta_i}{\lambda_i^2}}\right ) \right ]\text{ (definition of }\alpha)\\
&\geq \E\left [\arccos\left ( \frac{\sum_{i=1}^n\beta_i}{\sum_{i=1}^p\beta_i}\right ) \right ]\\
&\text{ ( Cauchy-Schwartz inequality and function arccos is decreasing)}\\
&\geq \arccos\left ( \E\left [\frac{\sum_{i=1}^n\beta_i}{\sum_{i=1}^p\beta_i}\right ]\right )\\
& \text{ ( Jensen inequality and concavity of arccos on }[0,1])\\
&\geq \arccos \left (\sqrt{\frac{n}{p}}\right ) \text{ (from (\ref{bikc1})).}
\end{align*}
This and inequality (\ref{lower1}) lead to the desired result. 
\end{proof}
\subsubsection{One should not use a simple linear estimate to get $\hat{F}_{10}$.}
 \begin{proposition}\label{properrlin}
 Suppose that $C$ is a positive definite matrix, and that we are given $X_1,\dots,X_n$ drawn independently from a gaussian Probability distribution $P$ with mean $m_{10}$ and covariance $C$ on $\R^p$. Let $\bar{m}_{10}$ be the associated empirical mean. Let us take $\hat{F}_{10}=C^{-1}\bar{m}_{10}$ and $\hat{s}_{10}=s_{10}$. Then, the classification rule $\1_{\hat{V}}$ defined by (\ref{e:baye}) leads to
  \[\E_{P^{\otimes n}}[\mathcal{R}(\1_{\hat{V}})]\geq \frac{\arccos \left (\frac{1}{\sqrt{p-3}}(\sqrt{n}\|F_{10}\|_{L_2(\gamma_C)}+1)\right )}{2\pi}e^{-\frac{\|F_{10}\|^2_{L_2(\gamma_C)}}{8}}.\] 
 \end{proposition}
 Before we give a proof, we comment this result briefly.\\
\indent {\bf Comment}. Suppose there exists $0<r<R$ such that $R>\|F_{10}\|^2_{L_2(\gamma_C)}\geq r$. From the preceding proposition, uniformly on all the possible values of $\mu_1$ and $\mu_0$, the learning error and the excess risk can converge to zero only if  $\frac{n}{p}$ tends to $0$. Recall that if no a priori assumption is done on $m_{10}$, $\bar{m}_{10}$ is the best estimator  (according to the mean square error) of $m_{10}$. Also, as in the estimation of a high dimensional vector problem (such as those described in (\cite{Candes:2006lk})), one should make a more restrictive hypothesis on $m_{10}$. We will suppose, in Section \ref{application}, that if $(a_k)_{k\geq 0}$ are the coefficients of $C^{-1/2}m_{10}$ in a well chosen basis, then  $\sum_{k\geq 0} a_k^q\leq R^q$  for $0<q<2$. 
\begin{proof}
As in the preceding proposition, we will use inequality (\ref{lower1}). Also it is sufficient to show the following 
\[\E\left [ |\alpha|\right ]\geq \arccos \left (\frac{1}{\sqrt{p-3}}(\sqrt{n}\|F_{10}\|_{L_2(\gamma_C)}+1)\right )\]
where $\alpha$ is defined by (\ref{alpha}). Because the function $\arccos$ is decreasing and concave on $[0,1]$, it suffices to obtain 
\begin{equation}\label{aprouver}
\E\left [\frac{|\langle F_{10},\hat{F}_{10}\rangle_{L_2(\gamma_C)}|}{\|F_{10}\|_{L_2(\gamma_C)}\|\hat{F}_{10}\|_{L_2(\gamma_C)}}\right ]\leq \frac{1}{\sqrt{p-3}}(\sqrt{n}\|F_{10}\|_{L_2(\gamma_C)}+1).
\end{equation} 
On the other hand, 
\[
\E\left [\frac{|\langle F_{10},\hat{F}_{10}\rangle_{L_2(\gamma_C)}|}{\|F_{10}\|_{L_2(\gamma_C)}\|\hat{F}_{10}\|_{L_2(\gamma_C)}}\right ] \leq \E\left [\frac{\|F_{10}\|_{L_2(\gamma_C)}}{\|\hat{F}_{10}\|_{L_2(\gamma_C)}}\right ]+\E\left [\frac{|\langle F_{10},\hat{F}_{10}-F_{10}\rangle_{L_2(\gamma_C)}|}{\|F_{10}\|_{L_2(\gamma_C)}\|\hat{F}_{10}\|_{L_2(\gamma_C)}}\right ]
\]
\[
\hspace{1cm}\leq \E\left [\frac{\|F_{10}\|_{L_2(\gamma_C)}^2}{\|\hat{F}_{10}\|_{L_2(\gamma_C)}^2}\right ]^{1/2}\left (1+\E\left [\frac{\langle F_{10},\hat{F}_{10}-F_{10}\rangle_{L_2(\gamma_C)}^2}{\|F_{10}\|_{L_2(\gamma_C)}^2}\right ]^{1/2}\right ),
\]
where this last inequality results from Cauchy-Schwartz.
Recall that  
\[\hat{F}_{10}=F_{10}+\frac{C^{-1/2}}{\sqrt{n}}\xi,\]
where $\xi$ is a standardised gaussian random vector of $\R^p$. Also, we easily obtain,
\[\E\left [\frac{\langle F_{10},\hat{F}_{10}-F_{10}\rangle_{L_2(\gamma_C)}^2}{\|F_{10}\|_{L_2(\gamma_C)}^2}\right ]^{1/2}=\frac{1}{\sqrt{n}},\]
 and 
\[\frac{\|F_{10}\|_{L_2(\gamma_C)}^2}{\|\hat{F}_{10}\|_{L_2(\gamma_C)}^2}=\frac{\|\sqrt{n}C^{1/2}F_{10}\|^2_{\R^p}}{\|\sqrt{n}C^{1/2}F_{10}+\xi\|^2_{\R^p}}.\]
The rest of the proof follows from the following simple fact which is a consequence of the Cochran Theorem and a classical calculation on $\chi^2$ random variables:\\
\indent Let $\sigma>0$, $\beta\in \R^p$, $X$ a gaussian random vector of $\R^p$ with mean $\beta$ and covariance $I_p$. Then 
\[\E\left [\frac{1}{\|X\|_{\R^p}^2}\right ]\leq \frac{1}{p-3}.\] 
\end{proof}
\subsection{Case where $\|F_{10}\|_{L_2(\gamma_C)}$ diverges: well separated data.}
We shall now rapidly consider the case when the data are well separated: the case where $\|F_{10}\|_{L_2(\gamma_C)}$ diverges. In the next theorem, we assume that $p$ tends to infinity.   
\begin{theoreme}\label{the:11}
Suppose that $0< \alpha<\pi/2$  ($\alpha$ is defined by (\ref{alpha})), and that\\  $\cos(\alpha)\|F_{10}\|_{L_2(\gamma_C)}\rightarrow \infty$ when $p$ tends to infinity. We then have 
\[
 \mathcal{R}\rightarrow \left \{\begin{array}{cc}0 & \text{ si }\liminf_{p\rightarrow \infty} \frac{2|d_0|}{|\langle F_{10} ,\hat{F}_{10} \rangle_{L_2(\gamma_C)}|}<1\\ b\geq \frac{1}{8}  & \text{ si }\limsup_{p\rightarrow \infty} \frac{2|d_0|}{|\langle F_{10} ,\hat{F}_{10} \rangle_{L_2(\gamma_C)}|}> 1\end{array}\right. \text{ when }p \rightarrow \infty.
\]

\end{theoreme}

This theorem is proved in Section \ref{proofth1}.
In the case of well separated data it is obvious that the optimal rule will perform perfectly. Theorem \ref{the:11} shows that for a given estimator $\hat{F}_{10}$ one should check that the probability to have $\liminf_{p\rightarrow \infty} \frac{2|d_0|}{|\langle F_{10} ,\hat{F}_{10} \rangle_{L_2(\gamma_C)}|}>1$ is small enough. 

\section{Quadratic perturbation of quadratic rule}\label{quadrat}
\subsection{Main results and remarks about the infinite dimensional setting}
In the case where $C_1\neq C_0$,  $\mathcal{L}_{10}(x)=\mathcal{L}^Q_{10}(x)$ is a polynomial function of degree two on $\R^p$:
\begin{equation}\label{eque:RN2}
\mathcal{L}^Q_{10}(x)=-\frac{1}{2}\langle A_{10}(x-s_{10}),x-s_{10}\rangle_{\R^p}+\langle G_{10},x-s_{10}\rangle_{\R^p}-c,
\end{equation}
where 
\begin{equation}\label{fact1}
A_{10}=C_1^{-1}-C_0^{-1},\;\; G_{10}=Sm_{10},\;\; 
\end{equation}
\[ S=\frac{C_0^{-1}+C_1^{-1}}{2},\;\;c=\frac{1}{8}\langle A m_{10},m_{10}\rangle_{\R^p}-\frac{1}{2}\log|\det(C_0^{-1}C_1)|,\]
 $m_{10}$ and $s_{10}$ are defined by (\ref{eque:RN}). 
\begin{remarque}
The equation (\ref{fact1}) giving $\mathcal{L}^Q_{10}(x)$ can be modified using the fact that 
\begin{equation}\label{eque:RN3}
 A_{10}=\frac{1}{2}\left ( C_1^{-1/2}W_{10}C_1^{-1/2}-C_0^{-1/2}W_{01}C_0^{-1/2}\right )\text{ where }W_{ij}=I-C_i^{1/2}C_j^{-1}C_i^{1/2}.
 \end{equation} 
 This modification has two advantages. It involves $W_{ij}$ which play an important role in the infinite dimensional framework (see remark \ref{infdimrem}).  On the other hand, it involves $W_{10}$ as much as $W_{01}$ which can lead in practice (while estimating $A_{10}$) to a symmetric procedure that does not give more importance to any group.  
 \end{remarque}
 In the classification problem, a polynomial of degree two $\widehat{\mathcal{L}}^Q_{10}(x)$ is used as a substitute for $\mathcal{L}_{10}$.  We decide that $X$ comes from class one if it belongs to 
\begin{equation}\label{hatfdev}
\hat{V}=\left \{x\in \R^p \text{ tq } \widehat{\mathcal{L}}^Q_{10}(x)\geq 0\right \},
\end{equation}
 The following theorem gives our solution to Problem \ref{Pb1}. 

\begin{theoreme}\label{the:2}
Let $\gamma$ be a gaussian measure on $\R^p$. Suppose that $\mathcal{L}_{10}^Q$ is a polynomial of degree two on $\R^p$ and that we have $\|\mathcal{L}_{10}^Q\|_{L_2(\gamma)}\geq r$ for $r>0$. Then, for all $q\in ]0,1[$, there exists $c_1(r,q)>0$ such that
\begin{equation}\label{rgiy}
\mathcal{R}(\1_{\hat{V}})\leq c_1(r,q)\|\mathcal{L}^Q_{10}-\widehat{\mathcal{L}}^Q_{10}\|^{q/3}_{L_2(\gamma)},
\end{equation}  
where $\hat{V}$ is given by (\ref{hatfdev}) and $\mathcal{R}$ by (\ref{errapprr}).
\end{theoreme} 

We emphasise the fact that $c_1(r,q)$ depends only $r$ and $q$. In particular it does not depend on the dimension $p$ of the problem. The proof of this Theorem is given in Section \ref{proofth2}. It is implicitly infinite dimensional, and the preceding theorem could have been stated in an infinite dimensional framework. We do not want to introduce this complicated framework and we refer to \cite{Bogachev:1998fk} for an introduction to the subject. The infinite dimensional framework highlights a particular aspect of the problem that is contained in the following remark.
\begin{remarque}\label{infdimrem}[infinite dimensional framework]
When $\X$ is a separable Hilbert space (it can also be a separable Banach space in the case of LDA) two gaussian measures $\gamma_{C_1,\mu_1}$ and $\gamma_{C_0,\mu_0}$ that are not equivalent are orthogonal.\\
\indent  If these measures are orthogonal then the observed data from the two classes are perfectly separated and $\mathcal{C}(g^*)=0$. In this case one can hope to obtain $\mathcal{C}(g)=0$ for a reasonable classification rule $g$ (Even if it is not trivial, see Theorem \ref{the:11} in the linear case).\\
\indent A necessary and sufficient condition for these measures to be equivalent is that
\begin{equation}\label{eque:condeq1}
m_{10}=\mu_1-\mu_0\in H(\gamma_{C_1,\mu_1})=H(\gamma_{C_0,\mu_0}),
\end{equation}
and
\begin{equation}\label{eque:condeq2}
W_{10}=I-C_1^{1/2}C_0^{-1}C_1^{1/2}\in HS(\X), \;\; 
\end{equation}
where $H(\gamma)$ is the reproducing Kernel Hilbert Space associated with a gaussian measure $\gamma$ and $HS(\X)$ is the space of Hilbert Shmidt operators with values in $\X$ (see corollaries p293 in \cite{Bogachev:1998fk}). In particular, the eigenvalues   of $W_{10}$ are in $l^2$. In the case where they are equivalent, one can define $\mathcal{L}_{10}$ as a limit (almost surely and $L_2$) of its finite dimensional counterpart. This can also be understand as measurable and squared integrable (with respect to $\gamma_{C_1,\mu_1}$) polynomials of degree two in $\X$ (see Chapter $5.10$ in \cite{Bogachev:1998fk}).  
\end{remarque}

\subsection{Comment and Corollary}. Suppose $\widehat{\mathcal{L}}^Q_{10}(x)$ is defined substituting $\hat{G}_{10}$, $\hat{s}_{10}$ $\hat{A}_{10}$ and $\hat{c}$ to $G_{10}$, $s_{10}$ $A_{10}$ and $c$ in (\ref{eque:RN2}). If we note 
\begin{equation}\label{delta0}
\delta_0=\hat{c}-c+\left \langle \hat{G}_{10}+(\hat{A}_{10}^*+\hat{A}_{10})(\hat{s}_{10}-s_{10}), \hat{s}_{10}-s_{10} \right \rangle_{\R^p},
\end{equation}
($A^*$ is the transpose of a matrix $A$)
\begin{equation}\label{deltaL}
\delta^L=\hat{G}_{10}-G_{10}+(\hat{A}_{10}^*+\hat{A}_{10})(\hat{s}_{10}-s_{10})
\end{equation}
and
\begin{equation}\label{deltaQ}
\delta^Q=\hat{A}_{10}-A_{10},
\end{equation} 
we then get, by straightforward calculation: 
\begin{equation}\label{perturQ}
\forall x\in \R^p\;\; \widehat{\mathcal{L}}^Q_{10}(x)=\mathcal{L}_{10}^Q(x)+\delta_0+\langle \delta^L,x-s_{10}\rangle_{\R^p} -\frac{1}{2}\langle \delta^Q(x-s_{10}),x-s_{10}\rangle_{\R^p}.
\end{equation}
Also, are result are about quadratic perturbations of quadratic rules.\\

The following corollary of Theorem \ref{the:2} is easier to use.
\begin{corollaire}\label{coroe:1}
Let $\X=\R^p$ and $C$ be a symmetric positive definite matrix on $\R^p$. Suppose that there exists $r>0$ such that $\|\mathcal{L}_{10}\|^2_{L_2(\gamma_{C,s_{10}})}>r$. Then, for $\1_{\hat{V}}$ given by (\ref{hatfdev}) and for all $0<q<1$ there exists $c_1(r,q)>0$ such that:
\[
 \mathcal{R}(\1_{\hat{V}})\leq c_1(r,q) \left (\frac{1}{2}\|C(A_{10}-\hat{A}_{10})\|^2_{HS(\R^p)}+\|C^{1/2}\delta^L\|_{\R^p}^2\right .\]
 \[ \hspace{2cm}\left .+2\delta_0^2+\frac{1}{2}trace^2(C(A_{10}-\hat{A}_{10})) \right )^{q/3},
 \]
where $\delta^L$ is given by (\ref{deltaL}) and $\delta_0$ by (\ref{delta0}).
\end{corollaire}
\begin{proof}
Let us recall that $\delta^Q$ is given by (\ref{deltaQ}). We have
\[
\|\mathcal{L}_{10}-\widehat{\mathcal{L}}_{10}\|^2_{L_2(\gamma_{C,s_{10}})}
\]
\begin{align*}
&=\|\frac{1}{2}(\delta^Q(x)-\E_{\gamma_C}[q_{\delta^Q}(X)])-\langle \delta^L,x \rangle_{\R^p}-(\delta_0-\frac{1}{2}\E_{\gamma_C}[q_{\delta^Q}(X)])\|^2_{L_2(\gamma_C)}\\
&\leq \frac{1}{4}Var(q_{C^{1/2}\delta^QC^{1/2}}(\xi))+Var(\langle C^{1/2}\delta^L,\xi\rangle_{\R^p})+2 \delta_0^2+2\E_{\gamma_C}^2[q_{C^{1/2}\delta^QC^{1/2}}(\xi)]\\
&(\xi\leadsto\gamma_{I_p,0}, \text{ note that there is equality here})\\
&=\frac{1}{2}\|C^{1/2}\delta^QC^{1/2}\|_{HS(\R^p)}^2+ \|C^{1/2}\delta^L\|_{\R^p}^2+2\delta_0^2+\frac{1}{2}trace^2(C^{1/2}\delta^QC^{1/2}).
\end{align*}
\end{proof}
\subsection{Comparison of this result with those obtained for LDA.}
The preceding theorem and its corollary are less powerful than those obtained for the LDA procedure and some conjectures might be made in a parallel with Theorem \ref{the:32}. In this theorem and in Theorem \ref{the:11}, both concerning linear rules, we explained and quantified how parameter estimation errors are less important when $\|F_{10}\|_{L_2(\gamma_C)}$ is large. This observation was based on the presence of a term exponentially decreasing with $\|F_{10}\|_{L_2(\gamma_C)}$ in the quantities which determine the upper bound to the learning error (and as a consequence the excess risk). In Theorem \ref{the:2} concerning QDA procedure, we did not obtain that type of term. Nevertheless, Remark \ref{infdimrem} (more precisely the relation this leads to equivalence of the measures) allow us to conjecture that such a term exists.  \\
\indent We also have to clarify the hypothesis under which the norm of $\mathcal{L}^Q_{10}$ is lower bounded. Let us recall that this hypothesis guaranties that the constant $c_1$ in equation  (\ref{rgiy}) is independent of the parameters of the problem. In a parallel with the results obtained for the procedure LDA the lower bound that is required for the norm of $\mathcal{L}^Q_{10}$ corresponds to the assumption that the two groups considered can always be distinguished. We believe that even if this hypothesis is natural, it is deeply linked with error measure that is used in our proof: the learning error. Hence, it is obvious that the excess risk is small when the data cannot be distinguished (see Section \ref{errapp} for a fuller discussion) but our result does not reflect this fact.
\begin{figure}
\center
    \includegraphics[width=10cm]{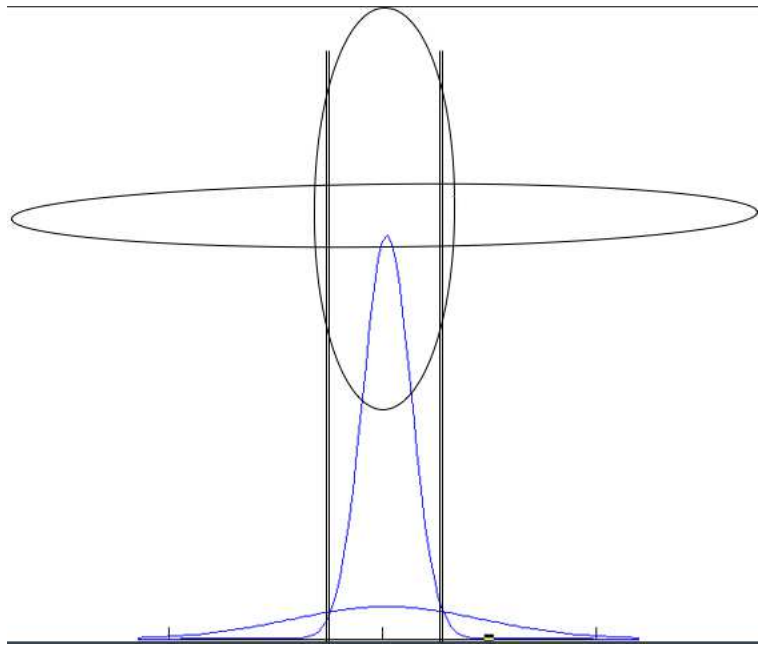}
   \caption{Separation of the data in a direction where the variances are different. The two groups can be identified with their ellipsoids of concentration: a horizontal ellipsoid and a vertical ellipsoid. the two groups have the same mean, but different covariance, which makes the data quite well separated. One can take advantage of this separation only if a quadratic rule is used.}
   \label{fig:graphvariancediff}
\end{figure}

We do not discuss the estimation of $G_{10}$ which leads to the same analysis as that for $F_{10}$ in the case of a linear rule. Let us now discuss the estimation of $W_{10}$ (and $W_{01}$). 
\subsection{Thresholding estimation of an operator and linearisation of a procedure.}
 Recall that $W_{10}$ is a symmetric matrix. Suppose we know an orthonormal base in which it is diagonal. Let $\lambda_{10}=(\lambda_{10i})_{i= 1,\dots,p}$ be the vector of its eigenvalues. To build the estimator $\hat{W}_{10}$ of $W_{10}$, we have to estimate its eigenvalues. It remains to measure the learning error and hence the estimation error of the eigenvalues vector in $l^2$ norm. Suppose that $p$ tends to infinity. We will recall later that if the measure of class $0$ and $1$ tend to equivalent gaussian measure in a separable Hilbert space, then $W_{10}$ tends to be Hilbert-Schmidt. This means that $\lambda_{10}$ stays in $l^2(\N)$. Once again, if $\lambda_{10}$ has coefficients decreasing sufficiently fast, the thresholding estimation should be used. This thresholding estimation is no longer a reduction of the dimension of the space in which the rules acts, but becomes a linearisation of the classification rules -It can be interpreted as a reduction of the dimension of the space in which the used rule lives- Indeed, let $\hat{W}_{10}=\sum_{i=1}^l\hat{\lambda}_{10i}e_i\otimes e_i$ for $l\leq  p$ and $(e_i)_{i=1,\dots,p}$ be  an orthonormal bases of $\R^p$, we have:
 \[\widehat{\mathcal{L}}^Q_{10}=\sum_{i=1}^l\hat{\lambda}_{10i}\langle e_i,x-\hat{s}_{10}\rangle^2_{\R^p}+g(x),\]
 where $g(x)$ is affine and defined on $\R^p$. In this case, the plug-in rule is affine in a subspace of dimension $p-l$ and quadratic in the subspace of dimension $l$ spanned by $(e_{i})_{i=1,\dots,l}$.\\
\indent Let us note that because $W_{10}=I-C_1^{-1/2}C_0C_1^{-1/2}$, setting the eigenvalues of $\hat{W}_{ij}$ to zero in a subspace of $\R^p$, is equivalent to choosing a subspace in which the covariance matrices $C_1$ and $C_0$ are "close enough". In this subspace, one can suppose that $C_1$ equals $C_0$. The classification rule, in this subspace, is linear. Figure \ref{fig:graphvariancediff} illustrates the case where the eigenvalues of $W_{10}$ are big enough and why a quadratic rule is better in that case.\\

\section{Classification procedure in high dimension: a way to solve Problem \ref{Pb2}}\label{procetude}
\subsection{Introduction.}
In this section, we give a practical method of classification for gaussian data in high dimension and hence present our contribution to Problem \ref{Pb2}.  Note that if we only treat the binary classification problem, it is easy to extend our procedure to the case of $K$ classes as we have done in \cite{Girard:2008wd}. Recall that we are given $n_1$ observations from $P_1$ and $n_0$ observations from $P_0$. We will note $n=n_1+n_0$. We suppose that each of the $n_k$ vectors of group $k$ is composed of the $p$ first wavelet coefficient (see \cite{wavetour}) of a random curve  from $\X=L^2[0,1]$ which is a realisation of a gaussian random variable $P_k=\gamma_{C_k,\mu_k}$ of unknown mean and covariance. \\
\indent Recall that a learning rule can be defined by a partition of $\R^p$. We construct this partition  $\hat{V},\R^p\setminus \hat{V}$ of $\R^p$ with the use of a frontier functions $\widehat{\mathcal{L}}_{10}$:
\begin{equation}\label{equ:parition1}
\hat{V}=\left \{x\in \R^p \;:\;\;\;\widehat{\mathcal{L}}_{10}(x)\geq 0\right \},
\end{equation}
which should be given in the sequel.\\
 
We divide here the presentation into two parts. In the first part, we give a theoretical result in the case where the covariance matrices are supposed to be known. In the second part, we give the method that is used when the covariances are unknown.  We keep the notation of the preceding sections. In the case of LDA procedure, $m_{10}=\mu_1-\mu_0$ $F_{10}=C^{-1}m_{10}$, $s_{10}=\frac{\mu_1+\mu_0}{2}$, and in the case of the QDA procedure, $G_{10}=\frac{1}{2}(C_1^{-1}+C_0^{-1})m_{10}$, $A_{10}=C_1^{-1}-C_0^{-1}$. 
\subsection{Case of known and equal covariance: procedure and theoretical result.}
\paragraph{Notation and assumptions.} Let $\bar{\mu}_{k}$ be the empirical mean of the learning data $(X_{ik})_{i=1,\dots,n_k}$ of class $k$. 
We suppose here that the covariance of group $0$ and $1$ equal $C$, and that $s_{10}$ is known. The separation frontier between the two groups is affine and $F_{10}$ is the only unknown parameter. We suppose that the learning set is made of $n_1=n_0=n(p)/2$ $p$-dimensional vectors. We give a method to construct an estimator of $F_{10}$ and give theoretical results when $n(p)$ tends then to infinity much more slowly than $p$.\\

For $q>0$, the ball $l_p^q(R)$ is composed of the vectors $\theta\in \R^p$ such that 
\[
\sum_{i=1}^{p}|\theta_i|^q\leq R^q.
\]

We will note
\begin{equation}
\Omega_p(\Theta(R),r)=\left \{(x,y,C)\in \R^p\times \R^p \times \mathcal{C}_p \; \text{ such that } \right .
\end{equation}
\[\hspace{2cm} \left . C^{-1/2}(x-y)\in \Theta(R) \text{ and }\|C^{-1/2}(x-y)\|_{\R^p}\geq r\right\}
\]
where $\mathcal{C}_p$ is the set of symmetric definite positive matrices in $\R^p$. If $(\mu_0,\mu_1,C)\in \Omega_p(\Theta(R),r)$, we will note 
\begin{equation}\label{errapp2}
\mathcal{D}(\hat{\mathcal{L}}_{10})=\mathcal{C}(\1_{\hat{V}})-\mathcal{C}(\1_{V}),
\end{equation}
where $\hat{V}$ is given by (\ref{equ:parition1}) and $V$ is given by (\ref{optimal}).

\paragraph{The Procedure.}
The plug-in rule affect the observation $X$ to class $1$ if it belongs to $\hat{V}$ defined by (\ref{equ:parition1})  where 
 \[\widehat{\mathcal{L}}_{10}=\langle\hat{F}_{10},X-s_{10}\rangle_{\R^p}.\]
  We estimate $F_{10}=C^{-1}m_{10}$ by  $\hat{F}_{10}=C^{-1}\hat{m}_{10}$, where the coefficients of $C^{-1/2}\hat{m}_{10}$ are given by 
\[\left ( y_{10l} 1_{|y_{10l}|>\lambda^{FDR}_{10}}\right)_{l=1,\dots,p}, \;\;\text{ where }\;  y_{10l}=\left (C^{-1/2}(\bar{\mu}_{1}-\bar{\mu}_{0})\right )_{l=1,\dots,p},\]
and $\lambda^{FDR}_{10}$ is chosen by the Benjamini and Hocheberg procedure \cite{BH95} for the control of the false discovery rate (FDR) of the following multiple hypotheses:
\begin{equation}\label{multipletest}
\forall l=1,\dots,p \;\;\; H_{0l}\;:\; E[y_{10l}]=0\;\;\;:\text{ Versus } H_{0l}\;:\; E[y_{10l}]\neq0
\end{equation}
 We recall that this procedure is the following. The $(|y_{10l}|)_l$ are ordered in decreasing order:
\[
|y_{10(1)}|\geq \dots \geq |y_{10(p)}|\text{ and }  \lambda_{10}^{FDR}=|y_{10(k_{10}^{FDR})}|
\]
\[
\text{ where } k_{10}^{FDR}=\max\left \{k \in \{1,\dots,p\}\; : \; |y_{10(k)}|\geq \sqrt{\frac{1}{n(p)}}z\left (\frac{b_pk}{2p}\right )\right \},
\]
  $z(\alpha)$ is the quantile of order $\alpha$ of a standardized gaussian random variable and $b_p\in [0,1/2[$ is lower bounded by $\frac{c_0}{\log p}$ where $c_0$ is a positive constant (which does not depend on $p$.\\

\paragraph{Theoretical result}
 
\begin{theoreme}\label{thbp}
Let $R>0$, and $q\in ]0,2[$. Let $\hat{V}$ be defined by (\ref{equ:parition1}) and $\eta_p=p^{-\frac{1}{q}}R\sqrt{n(p)}$. Suppose that $p$ tends to infinity. If $\eta_p^q\in [\frac{\log^5(p)}{p},p^{-\delta}]$ for $\delta>0$, then, for $r>0$, we have 
\[
\sup_{(\mu_0,\mu_1,C)\in\Omega_p(l^q(R),r)}\E_{P^{\otimes n}}\left [\mathcal{D}_p(\hat{\mathcal{L}}_{10})\right ]\leq \frac{1+o_p(1)}{r} \left ( \sqrt{2}\frac{\log^{1/2}\left (\frac{p}{R^qn(p)^{q/2}}\right )}{Rn^{1/2}(p)}\right )^{\frac{2-q}{2}},
\]

where $\mathcal{D}_p$ is the excess risk as defined by (\ref{errapp2}), and $P^{\otimes n}$ is the law of the learning set.
\end{theoreme}
\begin{proof} The covariance matrix of the vector  $C^{-1/2}(\bar{\mu}_1-\bar{\mu}_0)$ equals $I_p \frac{1}{n(p)}$. We then have to use successively Theorem \ref{the:32} (of this article),  Theorem $1.1$ of Abramovich et .al \cite{ABDJ05},  and Theorem $5$ point $3b.$ of Donoho and Johnstone \cite{Donoho:1994uq} to be able to write, $\forall r>0$: 
\[
\sup_{(\mu_0,\mu_1,C)\in\Omega_p(l^q(R),r)}\E_{P^{\otimes n}}\left [\mathcal{D}^2_p(\hat{\mathcal{L}}_{10})\right ]\leq \frac{1+o_p(1)}{r^2} \left ( \sqrt{2}\frac{\log^{1/2}\left (\frac{p}{R^qn(p)^{q/2}}\right )}{Rn^{1/2}(p)}\right )^{2-q}.\]
This inequality leads to the result by the use of the Jensen inequality:
\[\E_{P^{\otimes n}}\left [\mathcal{D}_p(\hat{\mathcal{L}}_{10})\right ]\leq \E_{P^{\otimes n}}\left [\mathcal{D}^2_p(\hat{\mathcal{L}}_{10})\right ]^{1/2}.\]
\end{proof}
\paragraph{Comments.}
Let us make a few remarks on this result. 
\begin{enumerate}
\item The rate of convergence is faster when $q$ is close to $0$, and slower when it is close to $2$. This leads us to consider the sparsity of $C^{-1/2}(\mu_0-\mu_1)$, and makes the use of the wavelet basis attractive.  On the one hand, it transforms a wide class of curves into sparse vectors and on the other hand, it almost diagonalises a wide class of covariance operators. 
\item  We could obtain the same speed with a universal threshold (i.e with the threshold $\lambda_U=\frac{1}{n(p)}\sqrt{2\log(p)}$). In this case, the  constant $\frac{1+o_p(1)}{r^2}$ would not be that good (cf \cite{ABDJ05}).
\item We are not aware of any results concerning the convergence of any classification procedure in this framework (the high dimensional gaussian framework with the set of possible parameter determined by $\Omega_p$). Indeed we do not make any strong assumption on $C$.  Bickel and Levina \cite{bickel:2004fk} as well as Fan and Fan \cite{Fans:2007vn} suppose in their work that the ratio between the highest and the lowest eigenvalue is lower and upper-bounded. Even if our Theorem doesnot treat the case where $C$ is unknown the hypotheses we use seems more natural. Let us recall  that  if $Y$ is a gaussian random variable with values in a Hilbert Space, then the covariance operator is necessarily nuclear. Also, the assumption used by the above mentioned authors does not allow us to consider gaussian measures with support in a Hilbert space. 
\item Finding the significant component of the normal vector $F_{10}$ defining the optimal separating hyperplan is equivalent with finding the significant contrast in a multivariate ANOVA. Hence, controlling the expected false discovery rate in this ANOVA is sufficient to get a good classification rule. 
\end{enumerate}
\subsection{The case of different unknown covariances}

 For the rest of this section, if $k\in \{0,1\}$, $\bar{\mu}_{k}$ will be the empirical mean of the Learning data of class $k$. We are going to use a diagonal estimator $\hat{C}_k$ of the covariance matrix $C_k$. The diagonal elements of $\hat{C}_k$ will be $(\hat{\sigma}^2_{kq})_{q=1,\dots,p}$. For $q\in \{1,\dots,p\}$, $k\in \{0,1\}$,  $\hat{\sigma}^2_{kq}$ will we the unbiased version of the empirical variance of feature $q$ of the observations $(X_{ikq})_{i=1,\dots,n_k}$ of class $k$.  We will note
 \[\hat{s}_{10}=(\bar{\mu}_1+\bar{\mu}_0)/2.\]

 The classification rule used chooses that $X\in \R^p$ comes from the class $k$ if $X$ belongs to $\hat{V}_k$ given by (\ref{equ:parition1}) and
\[\hat{\mathcal{L}}_{10}=-\frac{1}{2}\langle \hat{A}_{10}(x-\hat{s}_{10}),x-\hat{s}_{10}\rangle_{\R^p}+\langle \hat{G}_{10},x-\hat{s}_{10}\rangle_{\R^p}-\hat{c}_{10},\] 
where the quantities of this equation will be given in what follows. for all $(1,0)\in \{1,\dots,K\}^2$, $1\neq 0$, we now give $\hat{G}_{10}$ (equation (\ref{vdir})), $\hat{A}_{10}$ (equation \ref{matricxc}), and $\hat{c}_{10}$ (equation \ref{hatcij}).\\

   We estimate $G_{10}=\frac{1}{2}(C_1^{-1}+C_{0}^{-1})m_{10}$ by  
\begin{equation}\label{vdir}
\hat{G}_{10}=\left (\frac{1}{\sqrt{2}}\left (\frac{1}{\hat{\sigma}^2_{1q}}+\frac{1}{\hat{\sigma}^2_{0q}}\right )^{1/2}y_{10q} 1_{|y_{10q}|>\lambda^{FDR}_{10}}\right)_{q=1,\dots,p}
\end{equation}
\[ \text{ where } y_{10q}=\frac{1}{\sqrt{2}}\left (\frac{1}{\hat{\sigma}^2_{1q}}+\frac{1}{\hat{\sigma}^2_{0q}}\right )^{1/2}(\hat{\mu}_{1q}-\hat{\mu}_{0q}),\]
and $\lambda^{FDR}_{10}$ is chosen by the Benjamini and Hocheberg procedure. This procedure is the following. Let $Var_0(y_{ijq})$ be the variance of  $y_{10q}$ calculated under the hypothesis that $\mu_{1q}=\mu_{0q}$. The term
\[
 \frac{1+\hat{\sigma}^2_{1q}/\hat{\sigma}^2_{0q}}{2n_1}+\frac{1+\hat{\sigma}^2_{0q}/\hat{\sigma}^2_{1q}}{2n_0}
\]
  is an estimation of this variance when $\sigma^2_{kq}$ ($k=0,1$) are known and equal to $\hat{\sigma}^2_{kq}$. In practice, we substitute these terms for $Var_0(y_{10q})$. The real 
  \[(|y_{10q}|/\sqrt{Var_0(y_{10q})})_{q=1,\dots, p}\]
   are ordered by decreasing order:
\[
|y_{10(1)}|/\sqrt{Var_0(y_{10(1)})}\geq \dots \geq |y_{10(p)}/\sqrt{Var_0(y_{10(p)})}|\text{ and }  \lambda_{10}^{FDR}=|y_{10(k_{10}^{FDR})}|
\]
where
\[ k_{10}^{FDR}=\max\left \{k \; : \; |y_{10(k)}|\geq \sqrt{\frac{1+\hat{\sigma}^2_{1(k)}/\hat{\sigma}^2_{0(k)}}{2n_1}+\frac{1+\hat{\sigma}^2_{0(k)}/\hat{\sigma}^2_{1(k)}}{2n_0}}z\left (\frac{b_pk}{2p}\right )\right \},
\]
  $z(\alpha)$ is the quantile of order $\alpha$ of a standardized gaussian random variable and $b_p\in [0,1[$ is as in the preceding algorithm.\\
\indent   In practice, we choose $b_p=0.01$, but one could keep a part of the learning set to learn the best value of $b_p$. Note that in the application we have in mind, the learning set is too small to be divided. In addition, the choice of $b_p$, in view of Theorem \ref{thbp} does not determine the performances of the algorithm. In practice the difference of classification error between the choices $b_p=0.01$ and $b_p=0.05$ for example, is not important.\\
\indent This first part of the methods constitute a dimension reduction. Indeed, the only coordinates of $(\hat{G}_{10q})_{q=1,\dots,p}$ that are kept non null are those for which $|y_{10q}|\geq \lambda_{ij}^{FDR}$. The linear application associated with $(\hat{G}_{10q})_{q=1,\dots,p}$ only acts in $k_{10}^{FDR}$ directions. Let us also note that if we extend our procedure to a multiclass procedure, for two couples of classes $(i,j)\neq (l,m)$, the corresponding estimations $G_{ij}$ and $G_{lm}$ might be based on different dimension reduction.\\

\begin{remarque}
The testing procedure used can be analysed as a "vertical" ANOVA that reveals the interesting direction 
\begin{enumerate}
\item in which classification should be done (with thresholding estimation of $G_{10}$)
\item in which classification should be quadratic (with thresholding estimation of $A_{10}$).
\end{enumerate}
\end{remarque}
 
The matrix $A_{10}$ is estimated by a diagonal matrix with diagonal elements given by
\begin{equation}\label{matricxc}
\hat{a}_{10q}=\left ( \frac{1}{\hat{\sigma}_{1q}^2}-\frac{1}{\hat{\sigma}_{0q}^2}\right ) 1_{|w_{10q}|\geq \eta_{10}^{FDR}},\text{ where }\;\;w_{10q}=\hat{\sigma}^2_{1q}-\hat{\sigma}^2_{0q},\;\;\;\; q=1,\dots,p,
\end{equation}
and the threshold $\eta_{10}^{FDR}$ is chosen with the same type of procedure as the one used to find $\lambda_{10}^{FDR}$. Let $Var_0(w_{10q})$ be the variance of $w_{10q}$ under the hypothesis that $\sigma_{1q}=\sigma_{0q}$. The term $\frac{2\hat{\sigma}_{1q}^4}{n_1-1}+\frac{2\hat{\sigma}_{0q}^4}{n_0-1}$ is an estimation of it that we use in practice. The real numbers $(|w_{10q}/\sqrt{Var_0(w_{10q})}|)_q$ are ordered by decreasing order:
\[
|w_{10(1)}/\sqrt{Var_0(w_{10p})}|\geq \dots \geq |w_{10(p)}/\sqrt{Var_0(w_{10p})}|\text{ and }  \eta_{10}^{FDR}=|w_{10(k_{10}^{FDR})}|
\]
\[\text{ where } k_{10}^{FDR}=\max\left \{k \; : \; |w_{10(k)}|\geq \sqrt{\frac{2\hat{\sigma}_{1(k)}^4}{n_1-1}+\frac{2\hat{\sigma}_{0(k)}^4}{n_0-1}}z\left (\frac{b_pk}{2p}\right )\right \}.
\]
 This part of the method constitutes a linearisation of the rule. Indeed, the directions $q\in \{1,\dots,p\}$ in which $\hat{a}_{10q}$ is $0$ are the directions in which the classification rule between the groups $1$ and $0$ is linear. In the other directions, the rule is quadratic.\\
\indent The use of this methods is still motivated by Theorem \ref{thbp} and the theorems used in its proof, but it needs additional theoretical justification.\\
 
We will finally note: 
\begin{equation}\label{hatcij}
\hat{c}_{10}=\sum_{q=1}^p1_{|w_{10q}|\geq \eta_{10}^{FDR}}\left (\frac{1}{8} \hat{a}_{10q} (\bar{\mu}_{1q}-\bar{\mu}_{0q})^2+\frac{1}{2}\log|\det(\hat{\sigma}_{0q}^{-1}\hat{\sigma}_{1q})|\right ).
\end{equation}

\section{Application to medical data and the TIMIT database}\label{application}

We are going to study the performance of the given procedure. With that aim, we compare our method with the one given by Rossi and Villa \cite{Rossi:2006fk} on the  database TIMIT. We then use test our procedure on medical data. 
\subsection{Comparison of our method with the one of Rossi and Villa in the case of two class classification}

Rossi and Villa use a support vector machine (SVM) with different types of kernels. Recall that the SVM procedure is to construct an affine frontier function $f$ given by
\[f(x)=\langle w,x \rangle_{\R^p}+b,\]
where $w$ and $b$ are solutions of an optimization problem of the following type:
\[\min_{w,b,\xi} \|w\|_{\R^p}^2+C\sum_{i=1}^N\xi_i\]
\[ \text{ under } y_i\left (\langle w,x_i\rangle_{\R^n}+b\right )\geq 1-\xi_i,\;\;\;\xi_i\geq 0\;\;\;i=1,\dots,n\]

where $(x_i,y_i)_{i=1,\dots,n}$ are the couples  (observations, labels) of the learning set.  \\

	The TIMIT database has notably been studied by Hastie et al. \cite{Hastie:1995fk}. This database includes phonemes "aa"  and " ao "  pronounced by many different persons. The corresponding records are  curves observed at a fine enough sampling frequency. More precisely, one curve is a $p$-dimensional vector with $p=256$. The learning set is composed of $519$ "aa"  and $759$ " ao " and the test set is composed of $176$ "aa"  and $263$ "ao". Also, the curves  $(x_{i})_{i=1,\dots,519}$ are those which correspond to the pronunciation of phoneme "aa" and the label $y_i=0$ is associated to them. The label "1" is associated to the other curves which correspond to the pronunciation of phoneme " ao ". The method of Rossi and Villa gives almost the same results as ours: $20\%$ of classification mistakes. 
    
\subsection{Application to medical data}
The medical problem is the following. In Magnetic resonance imagery, one can obtain spectra characterizing tissues localized in some area of the brain. The spectra obtained can be used to characterize tumors. Unfortunately, even for a specialist, it is hard to define a good rule to associate the name of a tumor with a given spectra. Some spectra have been obtained on identified tumors. We have been given these spectra. In order to have enough spectra in our learning set, we retained five groups of spectra (some of them regrouping many tumors). The glioblastomes of the first type\footnote{The group of Glioblastomes has a too large variability, also, we chose to divide it into two groups: first type and second type. These two types correspond to the presence of certain chemical substances.}, the glioblastomes of the second type, the Meningiomes, the Metastases and the healthy tissues. The database provided by the specialists contains $21$ glioblastomes of first type, $9$ glioblastomes of second type, $16$ Méningiomes,  $18$ métastases and $9$ healthy tissues, that is, $75$ spectra sampled at $1024$ points. We give the plot of the spectra considered in Figure \ref{fig:contour}. In order to test our procedure, we used a strategy of type "leave on out". Figure \ref{fig:tum} leads us to an experimental confirmation that in the case of two class classification, the chosen dimension is a good one. 

\begin{figure}[htp]
  \centering
  \subfigure[$21$ glioblastomes A]{\label{fig:edge-a}\includegraphics[scale=0.41]{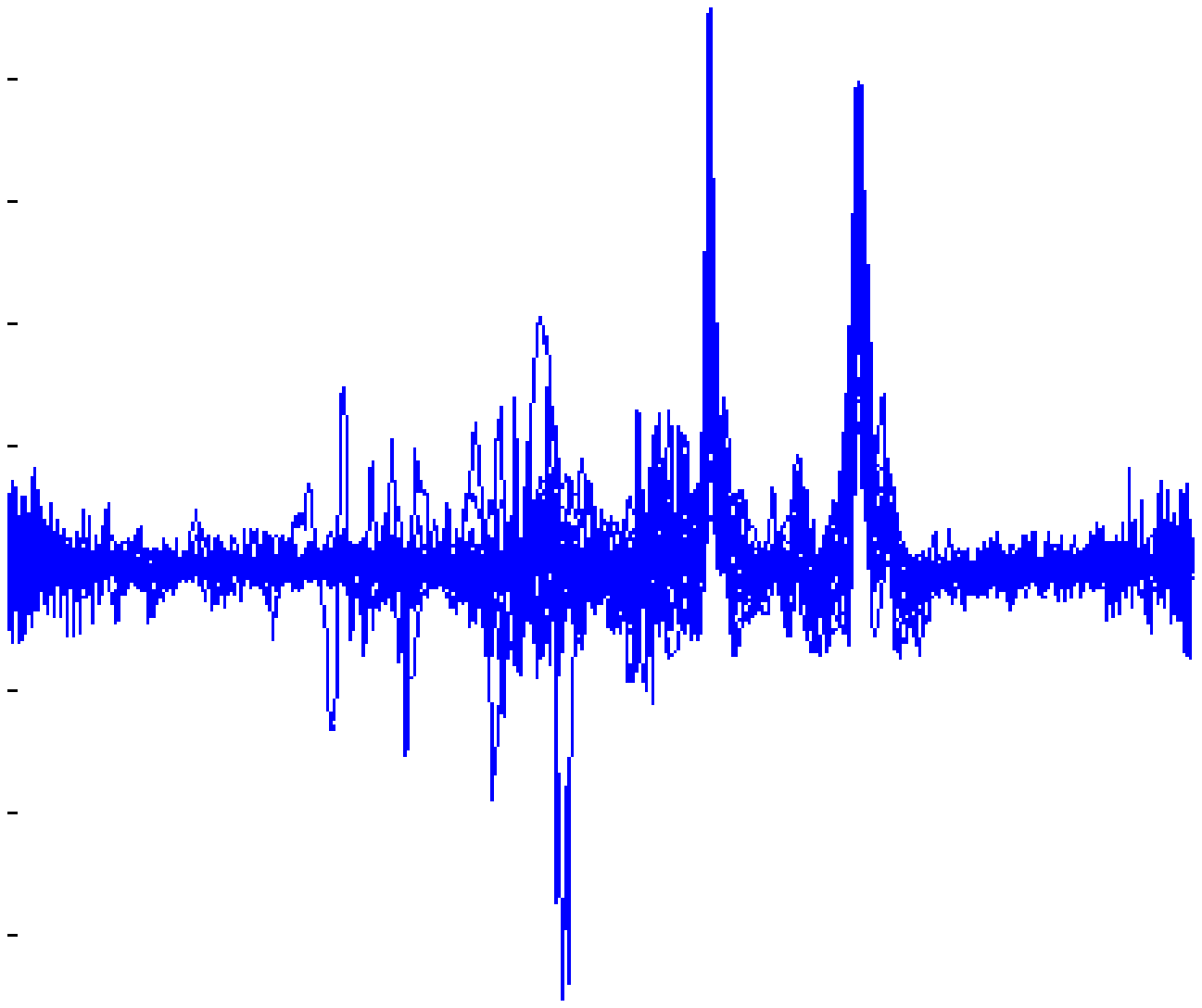}}                
  \subfigure[$9$ glioblastomes B]{\label{fig:contour-b}\includegraphics[scale=0.41]{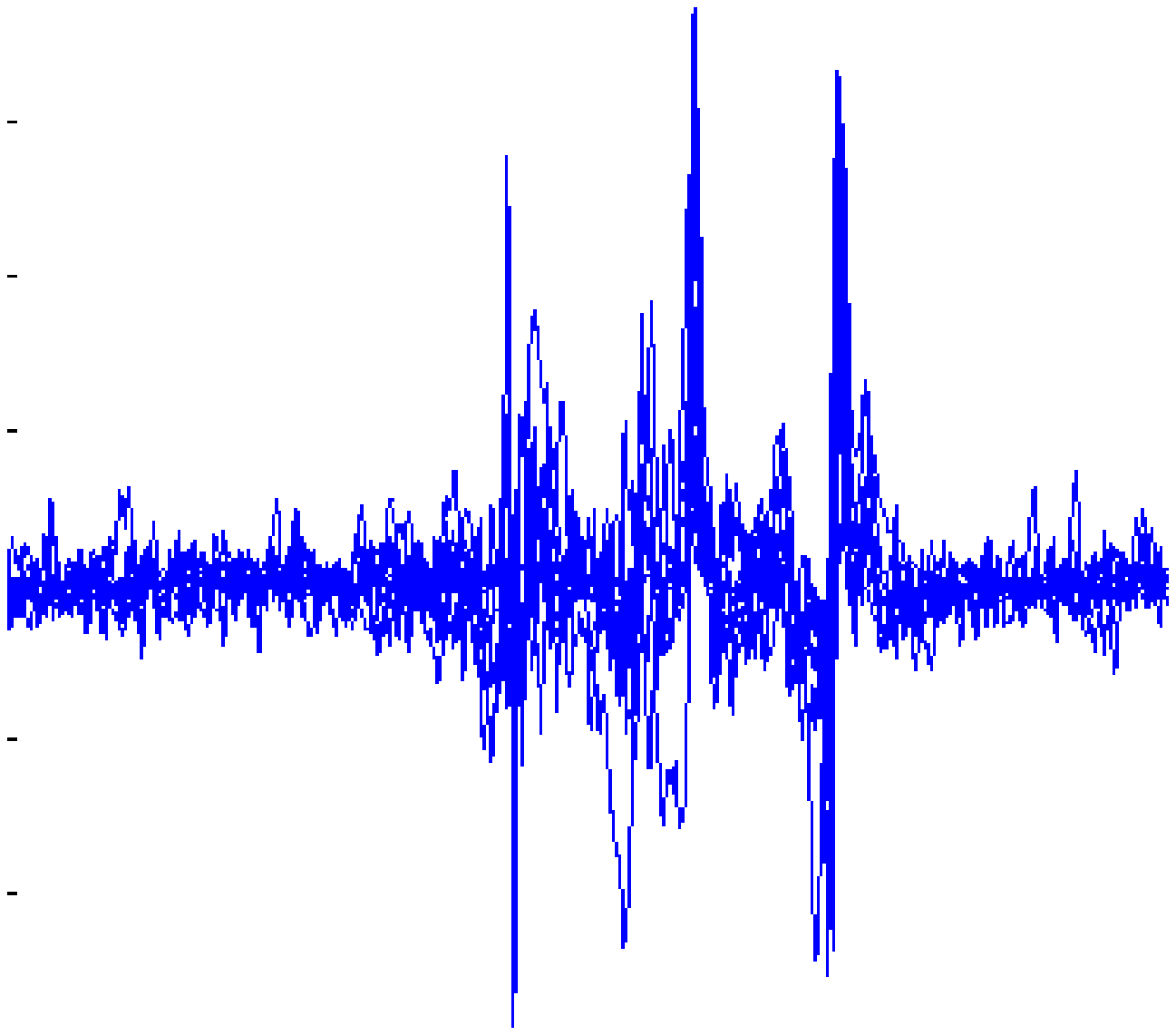}} 
  \subfigure[$16$ Meningiomes]{\label{fig:contour-c}\includegraphics[scale=0.41]{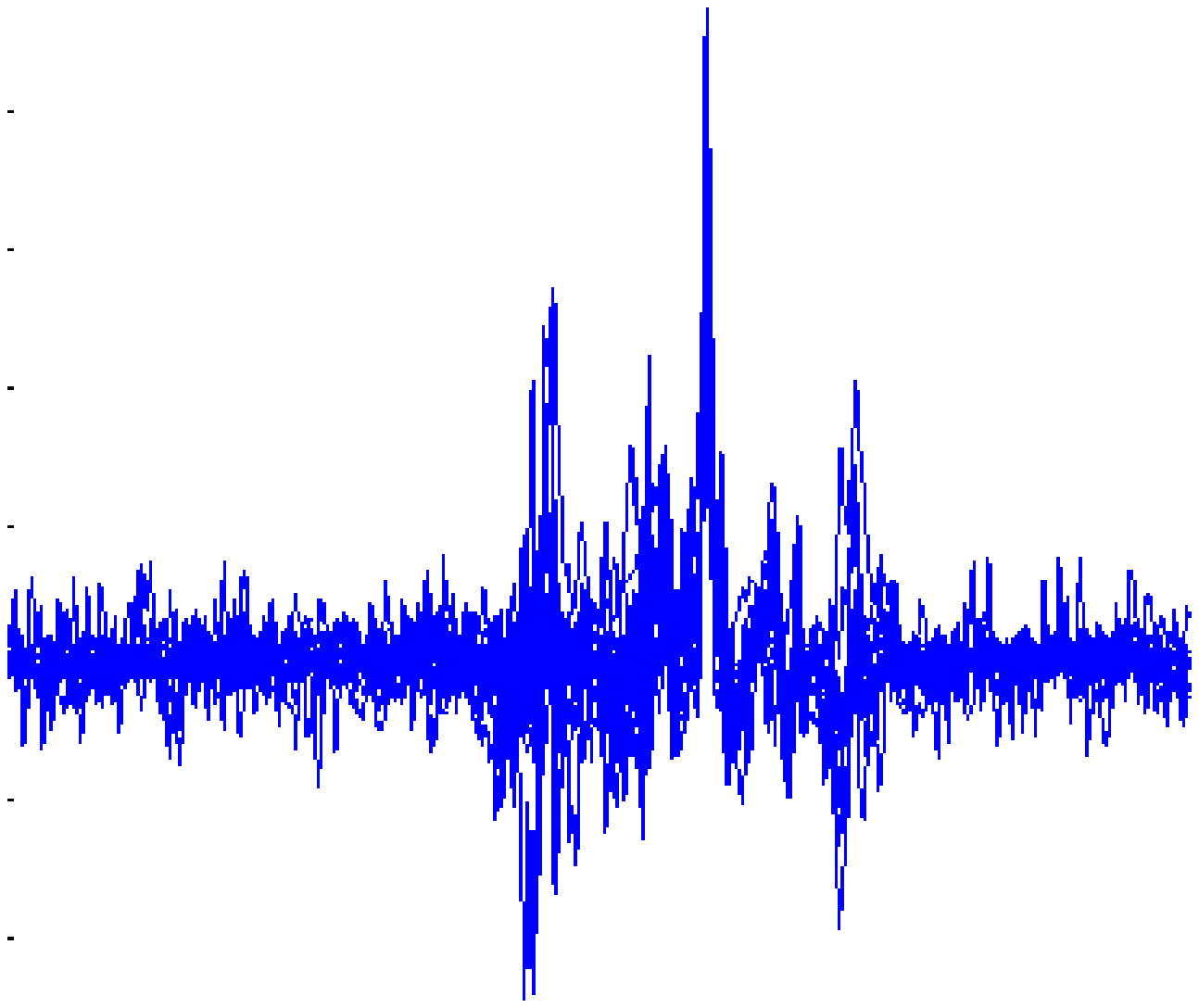}}
    \subfigure[$18$ metastases]{\label{fig:contour-d}\includegraphics[scale=0.41]{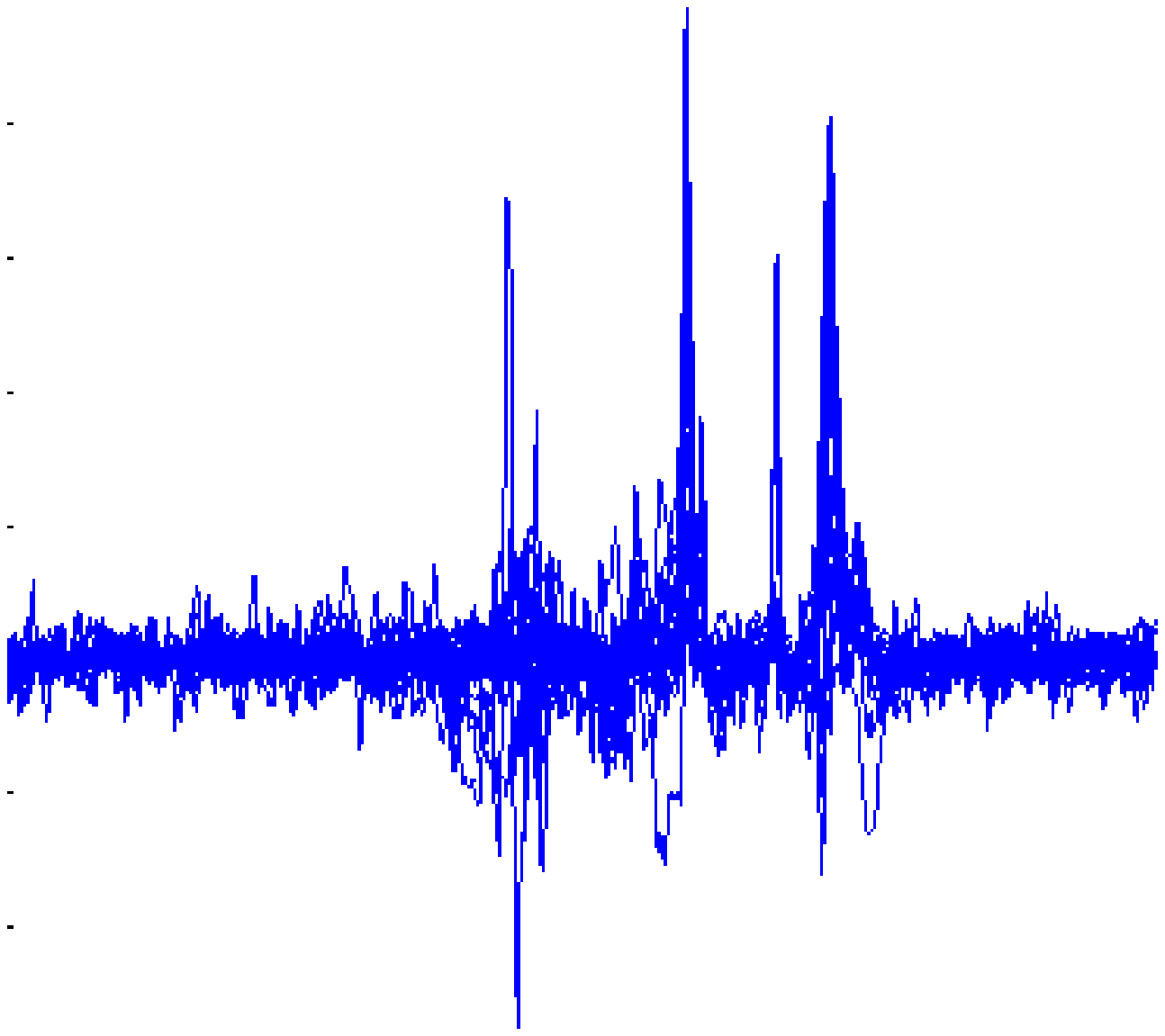}}
    \subfigure[$9$ healthy tissues]{\label{fig:contour-d}\includegraphics[scale=0.41]{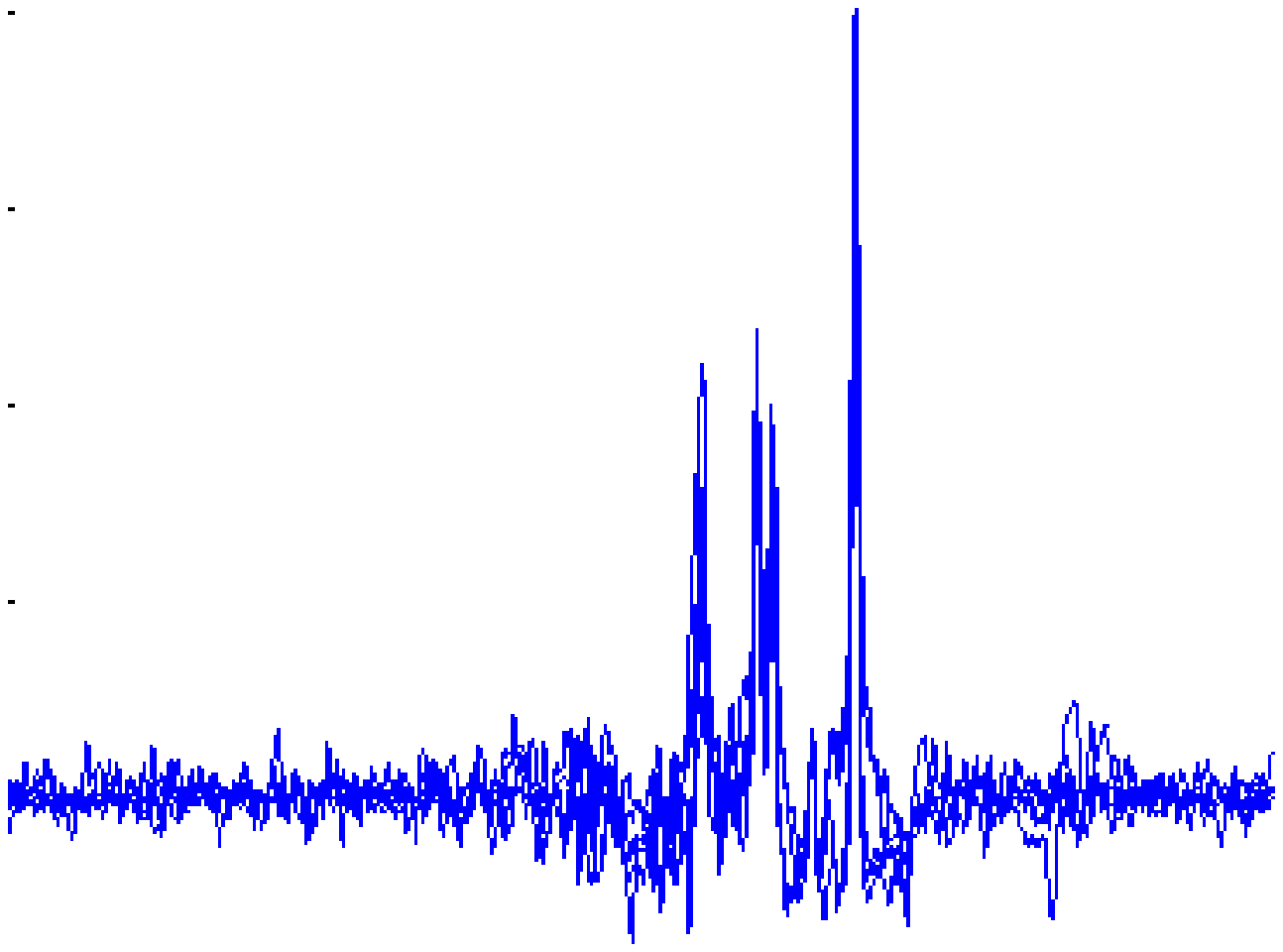}} 
    \caption{Spectra of the learning set}
  \label{fig:contour}
\end{figure}

\begin{figure}
\begin{center}
\begin{tabular}{|c|c|c|c|}
\hline
Groups considered & all & all except & Glioblastomes of first type \\
& &  Metastases & and Meningiomes\\
\hline
error rate & 43 \% & 30 \%  & 5\% \\
\hline
\end{tabular}
\end{center}
\caption{Considered groups and error rate in each case.}\label{tabl3}
\end{figure} 
\begin{figure}[htbp]
\begin{center}
\includegraphics[width=10cm]{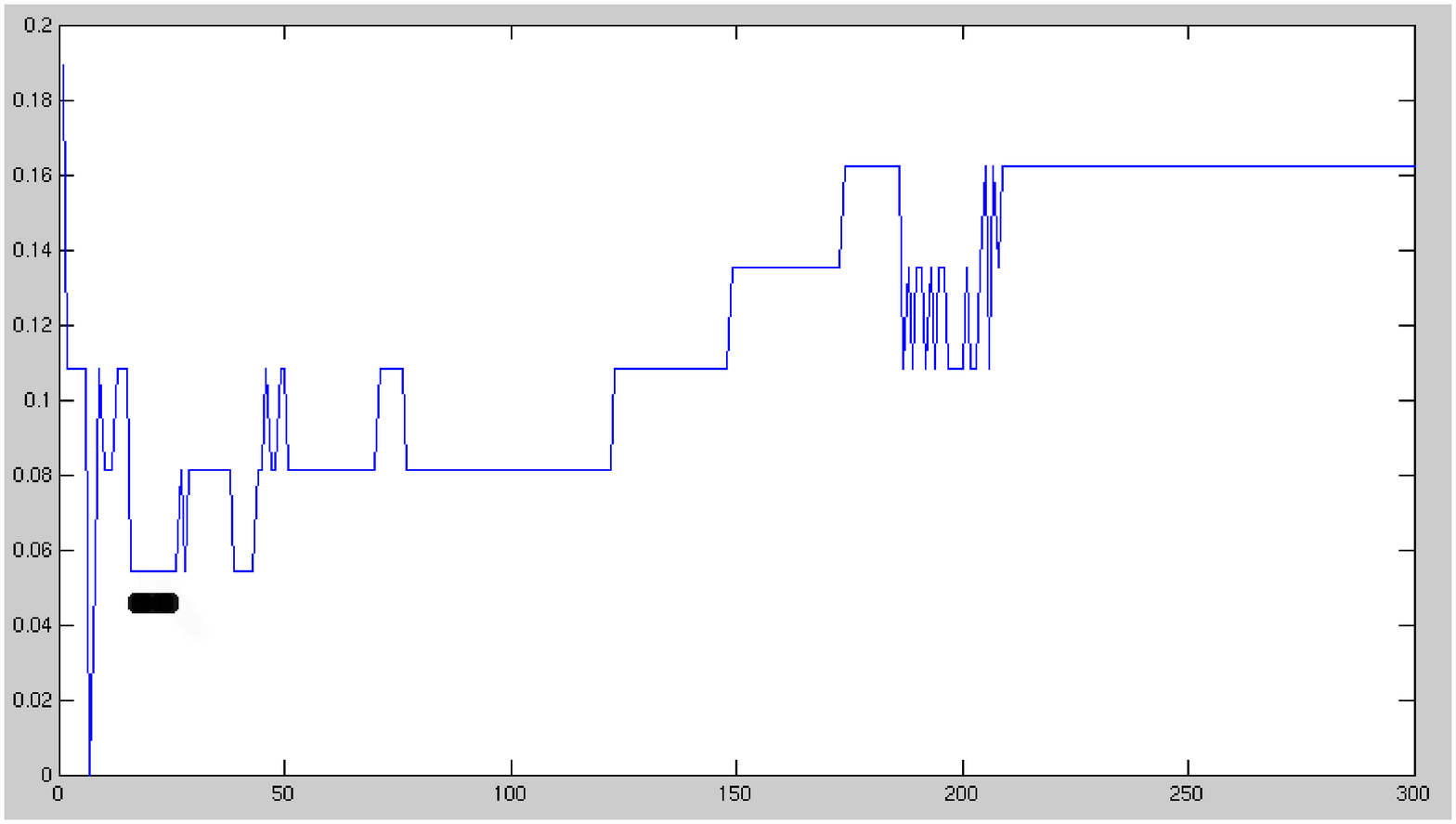}
\caption{Classification error rate (in  a two group problem: Méningiomes versus Glioblastomes of first type) as a function of the selected dimension. The dimension selected by our algorithm is marked by a black point in the Figure.}
\label{fig:tum}
\end{center}
\end{figure}

We tested different configurations summarized in the table Figure \ref{tabl3}. The classification error rate is still significant, but the reduction dimension procedure provides a reduction of the error rate (Recall that in the case of $4$ groups having equal a priori probability a rule that would guess randomly the type of tumor would have an error rate of $75 \%$). There are two reasons for this moderate performances.\\
\indent Roughly, theoretical physic predicts that a spectrum associated with a given tumor, for example a Glioblastome, is a random variable $y=(y_{q})_{q=1,\dots,p}$ that has a quite small variability. Also, we shuold be able to separate easily spectra associated with different groups. Unfortunately, in practice, the instrumentation leads to a measurement of spectra $z=(z_{q})_{q=1,\dots,p}$ having complex values and for which there exists a sequence of angles $(\psi_{q})_{q=1\dots,p}$ such that:
\[\forall q\in \{1,\dots,p\} \;\;\; y_q=\Re(e^{i\psi_q}z_q).\]
This sequence of angles is unknown. The theoretical physics of instrumentation shows that there are two real $(a,b)$ such that  
\[\forall q\in \{1,\dots,p\} \;\; \psi_q=aq+b.  \]
Methods to obtain $a$ and $b$ are not sufficiently efficient, but this represents an active field of research. We chose to ask the physicians to change the phase manually in order to have a homogeneous real part of the spectra in a particular group and we kept the real part of the spectra. The change of phase made by the physicians is not optimal and the residual variation of the phase creates a certain disparity  of observed spectra inside each group. This disparity can be seen Figure \ref{fig:contour}. The incorporation of the phase into a classification algorithm, and the use of the complex nature of the data will be the object of further studies. We note, however that these phase problems in the Fourier domain can be translated interestingly in the temporal domain.\\
\indent Finally, the learning set is still too small. We hope to see the size increase in the forthcoming years.


\section{A more geometric alternative measure of error: the learning error}\label{errapp}
\subsection{Definition and main result}
We have already defined the learning error to be 
\[
\mathcal{R}(g)=P(g(X)\neq Y \text{ et }g^*(X)=Y),
\]
which when $Y\leadsto \mathcal{U}(\{0,1\})$ equals 
\[\mathcal{R}(g)=\frac{1}{2}\left ( P_1(g(X)\neq 1 \text{ et }g^*(X)= 1)+ P_0(g(X)\neq 0 \text{ et }g^*(X)= 0)\right ).\]
 In other words, the learning error is the probability to misclassify $X$ with $g$ and to classify it correctly with $g^*$. The point that motivates the use of this error is that 
 \begin{enumerate}
 \item it leads to a simple geometric interpretation (mostly used in the two following Sections) and hence it is used in all the further theoretical development we will give;
 \item it is not sensitive to the possible indistinguishability of the distributions $P_0$ and $P_1$ and it leads to lower bounds as in Section $2$ (see remark below).
 \end{enumerate}
It follows easily from 
\[\mathcal{C}(g)-\mathcal{C}(g^*)=P(g(X)\neq Y\text{ et }g^*(X)= Y)-P(g(X)= Y\text{ et }g^*(X)\neq Y),\]
that a classification rule $g$ satisfies:
\begin{equation}\label{majerr}
\mathcal{C}(g)-\mathcal{C}(g^*)\leq \mathcal{R}(g).
\end{equation}
In the gaussian case that is studied in this article, we proved the following theorem that gives a reverse inequality of (\ref{majerr}). 

\begin{theoreme}\label{theorem-recip}
 Let $g^*$ be the optimal rule in the binary classification problem (as presented in Section $1$).
\begin{enumerate}
\item If $P_0$ and $P_1$ have the same covariance $C$ and respective means $\mu_1$ and $\mu_0$,  then, for all measurable functions $g:\R^p \rightarrow \{0,1\}$, we have:
\[\mathcal{C}(g)-\mathcal{C}(g^*) \geq\min\left \{ \frac{\sqrt{2\pi}}{2*16^2}\|C^{-1/2}m_{10}\|_{\R^p}e^{\frac{\|C^{-1/2}m_{10}\|_{\R^p}^2}{8}}\mathcal{R}(g)^{2},\frac{\mathcal{R}(g)}{8}\right \},\]
 where $m_{10}=\mu_1-\mu_0$.

\item Let $c_1>0$ and $\mathcal{P}(c_1)$ be the set of couples $(P,Q)$ of gaussian measure on $\R^p$ such that $d_1(P,Q)>c_1$.
If $(P_1,P_0)\in \mathcal{P}(c_1)$ then there exists a constant $c(c_1)>0$ (that only depends on $c_1$) such that 
\[\mathcal{C}(g)-\mathcal{C}(g^*) \geq \min\left \{c(c_1)\mathcal{R}(g)^{8},\frac{\mathcal{R}(g)}{8}\right \}.\]

\end{enumerate}
\end{theoreme} 
Before we prove this result, let us comment it. 
\paragraph{Comments.} 
Let us note that
\[\mathcal{C}(g)-\mathcal{C}(g^*)\leq \frac{1}{2}d_1(P_1,P_0).\]
Also, in the case where $d_1(P_1,P_0)$ tends to $0$, the excess risk does not measure the difference between $g$ and $g^*$ but the proximity of $P_1$ and $P_0$. The learning error is not sensitive to this scale phenomenon, as witness the following example.   
\begin{exemple}
 Let $\mu\geq 0$, $P_1=\mathcal{N}(\mu,1)$ and $P_0=\mathcal{N}(-\mu,1)$. In this case, for all $a\in \R$
\[\mathcal{R}(\1_{[a,\infty[})=\frac{1}{2}\left (P(0<\xi+\mu<a)+P(a<\xi-\mu<0)\right ),\]
where $\xi\leadsto\mathcal{N}(0,1)$ ; and $d_1(P_1,P_0)\rightarrow 0$ if and only if $\mu\rightarrow 0$ in which case
\[\mathcal{R}(\1_{[a,\infty,[})\rightarrow \frac{1}{2} P(\xi\in [0,|a|]). \]
Under these conditions, the learning error associated with $\1_{[a,\infty,[}$ tends to $0$ only if $a$ tends to $0$. In other words, when $\mu\rightarrow 0$, the learning error makes a difference between the rules  $\1_{[100,\infty,[}$ and $g^*=\1_{[0,\infty,[}$:
\[ \inf_{\mu<50}\mathcal{R}(\1_{[100,\infty[})\geq  \frac{1}{2} P(\xi\in [0,|50|])\approx \frac{1}{4}\] 
while we have
\[\mathcal{C}(\1_{[100,\infty[})-\mathcal{C}(g^*)\leq \frac{1}{2}d_1(P_1,P_0)\leq  \frac{\mu}{\sqrt{2\pi}}.\] 
\end{exemple}
\begin{remarque}\label{rem:distl1}
By definition, is the quantity of interest. The problem with it is that it can gives credit to every given procedure when $d_1(P_1,P_0)$ is sufficiently small. Also, one cannot argue that a rule is never good according to the excess risk. In the preceding example, the procedure $g(x)=\1_{[100,\infty[}(x)$ is uniformly (on say $ |\mu|\leq 50$) inconsistent according to the learning error but not according to the excess risk. 
\end{remarque}

The main consequence of this Theorem has already been used in Section $2.2$. From equation (\ref{majerr}), if $(g_n)_{n\geq 0}$ is a sequence of classification rules such that $\mathcal{R}(g_n)$ tends to zero, then $\mathcal{C}(g_n)-\mathcal{C}(g^*)$ tends to zero. Theorem \ref{theorem-recip}, implies the converse result. \\
\subsection{Proof of Theorem \ref{theorem-recip}}
\begin{proof}
Let us take 
\[K_1=\{x\in \R^p\;:\; g(x)\neq 1 \text{ et }g^*(x)=1\}\]
and 
\[K_0=\{x\in \R^p\;:\;g(x)\neq 0 \;\text{ et } g^*(x)=0\}.\]
Also, $\mathcal{R}(g)=\frac{1}{2}\left( P_1(K_1)+P_0(K_0)\right )$ and at least one of the following two inequalities is satisfied (from the pigeonhole principle): 
\[P_1(K_1)\geq \mathcal{R}(g),\;\;\;\; P_0(K_0)\geq \mathcal{R}(g).\]
Without loss of generality we will suppose that $P_1(K_1)\geq \mathcal{R}(g)$ which implies $P_1(K_1)+P_0(K_1)\geq \mathcal{R}(g)$. Note that we have 
\begin{align*}
\mathcal{C}(g)-\mathcal{C}(g^*)&=P(g\neq Y)-P(g^*\neq Y)\\
&=\frac{1}{2}\left (P_1(K_1)-P_1(K_0)\right )+\frac{1}{2}\left (P_0(K_0)-P_0(K_1)\right )\\
&(\text{ by conditioning with respect to }Y)\\
&=\frac{1}{2}\left ((P_1-P_0)(K_1)+(P_0-P_1)(K_0)\right ),
\end{align*}
and, because $g^*(X)=1$ if and only if $dP_1\geq dP_0$ (by definition of $g^*$ and from the fact that $Y\leadsto\mathcal{U}(\{0,1\})$), we get
\begin{equation}\label{duriet}
\mathcal{C}(g)-\mathcal{C}(g^*)=\frac{1}{2}\int 1_{K_1\cup K_0}|dP_1-dP_0|\geq \frac{1}{2}\int 1_{K_1}|dP_1-dP_0|.
\end{equation}
A straightforward calculation  (see for example \cite{Girard:2008wd} Proposition 1.4.2 Chapter 1 Part I) leads to 
\[\int_{\X}m(x)(dP_1-dP_0)=2\E_{P}\left [m(X)e^{f_{10}(P,X)}|\sinh\left (\frac{1}{2}\mathcal{L}_{10}(X)\right ) |\right ],\]
for all measurable $m$, where $P$ is any probability measure that dominates $P_1$ and $P_0$, $f_{10}(P,X)=\frac{1}{2}\log(\frac{dP_1}{dP}\frac{dP_0}{dP})$ and $\mathcal{L}_{10}(x)=\log(\frac{dP_1}{dP_0}(x))$. In particular
\[d_1(P_1,P_0)=2\E_{P}\left [e^{f_{10}(P,X)}|\sinh\left (\frac{1}{2}\mathcal{L}_{10}(X)\right ) |\right ],\]

Also note that whenever $K\subset\{x\in \R^p:\mathcal{L}_{10}(x)\geq 0\}$ we have
\[
P_1(K)-P_0(K)=2\E_{P}[1_{K}e^{f_{10}(P,X)}\sinh(\mathcal{L}_{10}(X)/2)],
\]
and as a consequence, (\ref{duriet}) can be rewritten
\begin{equation}\label{dfghl}
\mathcal{C}(g)-\mathcal{C}(g^*)\geq \E[1_{K_1}(X)e^{f_{10}(P,X)}\sinh(\mathcal{L}_{10}(X)/2)].
\end{equation}
It can also be shown that 
\[
P_1(K)+P_0(K)=2\E_{P}[1_{K}e^{f_{10}(P,X)}\cosh(\mathcal{L}_{10}(X)/2)],
\]
and consequently, $P_1(K_1)+P_0(K_1)\geq \mathcal{R}(g)$ is rewritten 
\begin{equation}\label{equ-conda}
2\E_{P}[1_{K_1}(X)e^{f_{10}(P,X)}\cosh(\mathcal{L}_{10}(X)/2)]\geq \mathcal{R}(g).
\end{equation}
On the other hand, $d_1(P_1,P_0)\geq c_1$ leads to:
 \begin{equation}\label{equ-condaa}
2\E_{P}[e^{f_{10}(P,X)}|\sinh(\mathcal{L}_{10}(X)/2)|]\geq c_1.
\end{equation}

In the rest of the proof, we shall combine (\ref{equ-conda}) and (\ref{equ-condaa}) in order to lower bound the right member of  (\ref{dfghl}). We remark that the left member in  (\ref{equ-conda}) and the right member of  (\ref{dfghl}) only differ by a factor two  and replacing a $\sinh$ by a  $\cosh$. For our purpose, these two functions only differ fundamentally near zero. We are going to decompose $K_1$ into two disjoint sets. Also, we will define
\[K_1^+=\{x\in K_1\;:\; \mathcal{L}_{10}(x)\geq 2\} \;\text{ et }\;K_1^-=\{x\in K_1\;:\; \mathcal{L}_{10}(x)\leq 2\} .\]
Let us also define $A$ and $B$ by:
\begin{align*}
\int_{K_1}e^{f_{10}(P,x)}\sinh(\mathcal{L}_{10}(x)/2)P(dx)=&\underbrace{\int_{K_1^+}e^{f_{10}(P,x)}\sinh(\mathcal{L}_{10}(x)/2)P(dx)}_{A}\\
&+\underbrace{\int_{K_1^-}e^{f_{10}(P,x)}\sinh(\mathcal{L}_{10}(x)/2)P(dx)}_{B}.
\end{align*}
From (\ref{equ-conda}), (and the pigeonhole principle) two cases can occur. In the first case 
\[\E_P[1_{K_1^+}(X)e^{f_{10}(P,x)}\cosh(\mathcal{L}_{10}(X)/2)]\geq \mathcal{R}(g)/4,\]
and in the second
\begin{equation}\label{domenisu}
\E_P[1_{K_1^-}(X)e^{f_{10}(P,x)}\cosh(\mathcal{L}_{10}(X)/2)]\geq \mathcal{R}(g)/4. 
\end{equation}
In the first case, because $X\in K_1^+$ implies
\[\sinh(\mathcal{L}_{10}(X)/2)\geq \frac{1}{2}\cosh(\mathcal{L}_{10}(X)/2)\;\; (\ln(6)\leq 2), \]
we have $A\geq \mathcal{R}(g)/8$ and hence the desired result ( it suffices to remark that $\mathcal{L}_{10}(x)\geq 0$ if $x\in K_1$ which implies $B\geq 0$).\\

 We shall now consider the case where (\ref{domenisu}) is satisfied. In this case, because $\cosh(x)\leq 2$ for all $|x|\leq 1$, we have 
\[
\int_{K_1^-}e^{f_{10}(P,x)}P(dx)\geq \mathcal{R}(g)/8.
\]
Also, the definition
\[d\nu=\frac{e^{f_{10}(P,x)}dP}{\int e^{f_{10}(P,x)}dP},\]
makes $\nu$ a probability measure on $\R^p$ and 
\begin{equation}\label{equ-condb}
\nu(K_1^-)\geq \mathcal{R}(g)/8.
\end{equation} 
On the other hand, (see the definition of $f_{10}$)
\[\int e^{f_{10}(P,x)}dP=\int \sqrt{dP_1dP_0}=A_2(P_1,P_0)\]
($A_2(P_1,P_0)$ is the Hellinger affinity between $P_1$ and $P_0$)
which leads to
\begin{equation}\label{defB}
B=A_2(P_1,P_0)\int_{0}^{\infty}\nu\left (X\in K_1^-\;\text{ and }|\sinh(\mathcal{L}_{10}(X)/2)|\geq t\right ) dt.
\end{equation}
 We have
\[\nu(X\in K_1^-)=\nu\left (X\in K_1^-\text{ and }|\sinh(\mathcal{L}_{10}/2)|\leq t\right )\]
\[\hspace{2cm}+\nu\left (X\in K_1^-\text{ and }|\sinh(\mathcal{L}_{10}/2)|\geq t\right ).\]
Let  $g$ be the application which associates to $t>0$ the real 
\begin{equation}\label{defgdet}
g(t)=\sup_{(P_1,P_0)\in \mathcal{P}(c_1)}\nu(|\sinh(\mathcal{L}_{10}(X)/2)|\leq t).
\end{equation} 
For every $t>0$, we have:
\begin{align*}
\nu&\left (X\in K_1^-\text{ and }|\sinh(\mathcal{L}_{10}/2)|\geq t\right )\\
&=\nu(X\in K_1^-)-\nu\left (X\in K_1^-\text{ and }|\sinh(\mathcal{L}_{10}/2)|\leq t\right )\\
\end{align*}

We then deduce from this inequality and from (\ref{defB}) that for all $\epsilon\geq 0$,   
\begin{align*}
B&\geq A_2(P_1,P_0) \int_{0}^{\epsilon}\nu\left (X\in K_1^-\;\text{ and }|\sinh(\mathcal{L}_{10}(X)/2)|\geq t\right ) dt\\
&\geq \epsilon\nu(X\in K_1^-)-A_2(P_1,P_0)\int_{0}^{\epsilon}\nu\left (X\in K_1^-\text{ and }|\sinh(\mathcal{L}_{10}/2)|\leq t\right )dt )\\
&\geq \epsilon \mathcal{R}(g)/8-\int_{0}^{\epsilon}\nu\left (X\in K_1^-\text{ and }|\sinh(\mathcal{L}_{10}/2)|\leq t\right )dt A_2(P_1,P_0)
\end{align*}
where this last inequality results from (\ref{equ-condb}).
The rest of the proof relies on the following lemma. 
\begin{Lemme}\label{lemmetemp}
\begin{enumerate}
\item The application $g$ defined by (\ref{defgdet}) leads to 
\[g(t)\leq \frac{c(c_1)}{A_{2}(P_1,P_0)}t^{1/7}\]
($c(c_1)$ is a positive constant that only depends on $c_1$). 
\item In the case where $C_1=C_0=C$, we have   
\[\nu\left (X\in K_1^-\text{ and }|\sinh(\mathcal{L}_{10}/2)|\leq t\right )\leq \frac{4t}{\sqrt{2\pi}\|C^{-1/2}m_{10}\|_{\R^p}}.\]
\end{enumerate}
 
\end{Lemme}
We prove this result at the end of the current proof. Let us note that  it is equation  (\ref{equ-condaa}) that plays a crucial role in the proof. \\
 
In the case where $C_1\neq C_2$,
\[\int_{0}^{\epsilon}\nu\left (X\in K_1^-\text{ and }|\sinh(\mathcal{L}_{10}/2)|\leq t\right )dt A_2(P_1,P_0)\leq \tilde{c}(c_1)\epsilon^{1+1/7},\]
and the choice $\epsilon=\left (\frac{\mathcal{R}(g)}{16}\tilde{c}(c_1)\right )^{7}$ leads to the desired result. In the case where $C_1=C_2$,
\[\int_{0}^{\epsilon}\nu\left (X\in K_1^-\text{ and }|\sinh(\mathcal{L}_{10}/2)|\leq t\right )dt\leq \frac{2\epsilon^2}{\sqrt{2\pi}\|C^{-1/2}m_{10}\|_{\R^p}},\]
and the choice $\epsilon=\sqrt{2\pi}\|C^{-1/2}m_{10}\|_{\R^p}\frac{\mathcal{R}(g)}{32A_2(P_1,P_0)}$ leads to the desired result. Indeed, in the case where $C_1=C_0$, classical calculation leads to  
\[A_2(P_1,P_0)=\int e^{f_{10}(P,X)}dP =e^{-\frac{\|C^{-1}(\mu_1-\mu_0)\|_{\R^p}^2}{8}}.\] 
\end{proof}
Let us now prove Lemma (\ref{lemmetemp})
\begin{proof}
Let us begin by point $2$. It is sufficient to notice that if $P_{1|0}$ is a gaussian measure with covariance $C$ and mean $s_{10}$, and if $X$ is a random variable drawn from $P_{1|0}$, then 
\[e^{f_{10}(P_{1|0},X)}=e^{-\frac{\|C^{-1}(\mu_1-\mu_0)\|_{\R^p}^2}{8}} \text{ in distribution }\mathcal{L}_{10}(X)\leadsto \mathcal{N}(0,\sigma^2),\]
where $\sigma^2=\|C^{-1}(\mu_1-\mu_0)\|_{\R^p}^2$. Also, we get
\begin{align*}
\nu(|\sinh(\mathcal{L}_{10}(X)/2)|\leq t) &=P\left (|\mathcal{N}(0,\sigma^2)|\leq 2Argsinh(t)\right )\leq \frac{4Argsinh(t)}{\sqrt{2\pi}\sigma}\\
&\leq \frac{4t}{\sqrt{2\pi}\sigma}.
\end{align*}

Let us now prove point $1$ of the Lemma. 
\begin{align*}
\nu(|\sinh(\mathcal{L}_{10}(X)/2)|\leq t)&\leq \int 1_{|\sinh(\mathcal{L}_{10}(x)/2)|\leq t}\left (\frac{dP_1}{dP_0}\right )^{1/2}dP_0/A_2(P_1,P_0).\\
&\leq \frac{P_0^{1/2}(|\mathcal{L}_{10}(X)/2|\leq t)}{A_{2}(P_1,P_0)} \\
&\text{ (from Cauchy-Schwartz and }Argsh(y)\geq y).
\end{align*}
 Finally, we conclude from point $2$ of Theorem \ref{th:formquadra}, given in Section \ref{proofth2}, which hypothesis is satisfied since:  
 \begin{align*}
 c_1&\leq d_1(P_1,P_0)\\
	&\leq 2\sqrt{K(P_0,P_1)} \\
	& (\text{from Pinsker inequality (see \cite{Tsybakov:2004fk})}),\\
 &\leq 2\|\mathcal{L}_{10}\|^{1/2}_{L_2(P_0)}\\
 & (\text{from Cauchy-Schartz inequality}).
 \end{align*}
\end{proof}

\section{A geometrical Analysis of LDA to solve Problem \ref{Pb1}}\label{proofth1}
\subsection{Introduction and first result}
Let $\X$ be a separable Banach space $\X=\R^p$, endowed with its Borel $\sigma$-field and a gaussian measure $\gamma$. Throughout the next section, we will associate to any measurable $f$ the set
\begin{equation}\label{de:rg}
V_{f}=\{x\in \X \;\; :\;\; f(x)\geq 0 \}.
\end{equation}

In this section $\X=\R^p$. Recall that $\alpha$ (defined by (\ref{alpha})) is the angle, according to the geometry of $L_2(\gamma_{C})$ between $F_{10}$ et $\hat{F}_{10}$. This quantity will play a very important role in the whole section. In order to  shorten the notation, we will replace $\mathcal{R}(\1_{\hat{V}})$ by $\mathcal{R}$ in this section and those that follow. \\ 
\indent Recall that 
\[
F_{10}=C^{-1}m_{10},\;\;\; m_{10}=\mu_1-\mu_0,\;\; s_{10}=\frac{\mu_1+\mu_0}{2},
\]
where $\mu_1$, (resp. $\mu_0$) and $C$ are the mean and (common) covariance of the distribution $P_1=\gamma_{C,\mu_1}$ (resp. $P_0=\gamma_{C,\mu_0}$) of data from group $1$ (resp. $0$). With the above defined notation (\ref{de:rg}), the optimal rule and the plug-in rule can be rewritten with 
\[V=V_{\langle F_{10},x-s_{10}\rangle_{\R^p}}\;\;\text{ and }\; \hat{V}=V_{\langle \hat{F}_{10},x-\hat{s}_{10}\rangle_{\R^p}}\]

For the purpose of this section, let us note that the learning error studied in the preceding section and introduced by equation (\ref{errapprr}) is (in the case of LDA) 
\[
\mathcal{R}=\frac{1}{2}\left ( \gamma_{C,\mu_0} \left (X \in \hat{V} \setminus V \right )+\gamma_{C,\mu_1} \left (X \in V \setminus \hat{V}\right )\right ).
\]
which implies
\begin{equation}\label{defrisque2}
\mathcal{R}=\frac{1}{2}\left ( \gamma_{C,s_{10}} \left (X \in \left (\hat{V} \setminus V -\frac{m_{10}}{2} \right )\right )+\gamma_{C,s_{10}} \left (X \in \left (V \setminus \hat{V}+\frac{m_{10}}{2}\right ) \right )\right ).
\end{equation}
The Problem now becomes to that of measuring two areas of $\R^p$ with $\gamma_{C,s_{10}}$. Standard properties of gaussian measure now leads to
\begin{equation}\label{tmpp1}
\mathcal{R}=\frac{1}{2}\gamma_p\left ((V_{\langle .,G_p\rangle_{\R^p} }\setminus V_{\langle .,G_p+e_p\rangle_{\R^p} +d_0})-\frac{G_p}{2}\right )
\end{equation}
\[\hspace{2cm}+\frac{1}{2}\gamma_{p}\left ((V_{\langle .,G_p+e_p\rangle_{\R^p} +d_0}\setminus V_{\langle .,G_p\rangle_{\R^p} })+\frac{G_p}{2}\right ),\]
where $d_0=\langle \hat{F}_{10};\hat{s}_{10}-s_{10}\rangle_{\R^p}$, 
\begin{equation}\label{defqte}
G_p=C^{1/2}F_{10}=C^{-1/2}m_{10},\;\; \hat{G}_p=C^{1/2}\hat{F}_{10} \;\text{ and }\; e_p=C^{1/2}(\hat{F}_{10}-F_{10}).
\end{equation} 
One may note that the change of geometry implies 
\begin{equation}\label{geomcha}
\|G_p\|_{\R^p}=\|F_{10}\|_{L_2(\gamma)},\;\; \|\hat{G}_p\|_{\R^p}=\|\hat{F}_{10}\|_{L_2(\gamma)},\;\;\|e_p\|_{p}=\|F_{10}-\hat{F}_{10}\|_{L_2(\gamma_C)}, 
\end{equation} 
and $\alpha$ (defined by equation (\ref{alpha})) is the angle, in the geometry of $\R^p$ between $G_p$ and $\hat{G}_p$.\\

\indent The following theorem gives lower bounds and upper bounds on the learning error $\mathcal{R}$ as functions of (among others) $\alpha$. Its proof relies on the fact that $\mathcal{R}$ is the measure by $\gamma_2$ of two "simple" areas of $\R^p$ (see Figure \ref{fig:lemprin}) and the use of four elementary properties of gaussian measure to be given later (see Figure \ref{fig:pro}).
\begin{theoreme}\label{the:1}
Let $d_0=\langle \hat{F}_{10},\hat{s}_{10}-s_{10}\rangle_{\R^p}$. The Learning error $\mathcal{R}$ as a function of $\alpha$ satisfies:
\[ \forall \alpha\in [-\pi,\pi] \;\;\; \mathcal{R}(\alpha)=\mathcal{R}(-\alpha).\]

The Learning error also satisfies the following inequality\\ 

If $\alpha\geq \frac{\pi}{2}$, then $\mathcal{R}\geq \frac{1}{2}$.\\

If $0\leq \alpha<\frac{\pi}{2}$, then we have $\mathcal{R}\leq \frac{1}{2}$ and we distinguish between four cases.
\begin{enumerate}
\item If $|d_0|\leq \frac{1}{4}|\langle F_{10} ,\hat{F}_{10} \rangle_{L_2(\gamma_C)}|$, we have: 
\begin{equation}\label{eque:th12}
e^{-\frac{\|F_{10}\|_{L_2(\gamma_C)}^2}{8}}\frac{1}{4} \left (\frac{\alpha}{2\pi}+\frac{1}{2}\gamma_1\left (\left [0;\frac{|d_0|\tan(\alpha)}{\|\Pi_{F_{10}^{\bot}}\hat{F}_{10}\|_{L_2(\gamma_C)}} \right ]\right )\right) \leq \mathcal{R},
\end{equation}
and 
\begin{equation}\label{eque:th1}
\mathcal{R}\leq e^{-\frac{\|F_{10}\|_{L_2(\gamma_C)}^2\cos(\alpha)^2}{32}}\left (\frac{\alpha}{2\pi}+\gamma_1\left (\left [0;\left (1+\tan(\alpha)\right )\frac{|d_0|\tan(\alpha)}{\|\Pi_{F_{10}^{\bot}}\hat{F}_{10}\|_{L_2(\gamma_C)}} \right ]\right )\right).
\end{equation}
\item If $\frac{1}{4}|\langle F_{10} ,\hat{F}_{10} \rangle_{L_2(\gamma_C)}| <  |d_0|\leq \frac{1}{2}|\langle F_{10} ,\hat{F}_{10} \rangle_{L_2(\gamma_C)}|$, we have: 
\begin{equation}
 e^{-\frac{\|F_{10}\|_{L_2(\gamma_C)}^2}{2}} \frac{1}{4}\left (\frac{1}{2}\gamma_1\left (\left [0;\frac{\|F_{10}\|_{L_2(\gamma_C)}}{4}\right ]\right )+\frac{\alpha}{2\pi} \right )\leq \mathcal{R}
 \end{equation}
 \begin{equation}\mathcal{R} \leq \frac{\alpha}{2\pi}+\gamma_1\left (\left [0;\left (1+\tan(\alpha)\right )\frac{|d_0|\tan(\alpha)}{\|\Pi_{F_{10}^{\bot}}\hat{F}_{10}\|_{L_2(\gamma_C)}}\right ]\right ).
\end{equation}
\item If $\frac{1}{2}|\langle F_{10} ,\hat{F}_{10} \rangle_{L_2(\gamma_C)}| < |d_0|$, we have:
\begin{equation}\label{eque:th02}
\frac{\alpha}{4\pi}+\frac{1}{4}\gamma_1\left (\left [0;\frac{\|F_{10}\|_{L_2(\gamma_C)}}{2}\right ]\right )\leq \mathcal{R},\end{equation}  
\[\mathcal{R}\leq \frac{\alpha}{2\pi}+\gamma_1\left (\left [0;(1+\tan(\alpha))\frac{|d_0|\tan(\alpha)}{\|\Pi_{F_{10}^{\bot}}\hat{F}_{10}\|_{L_2(\gamma_C)}}\right ]\right ).\]
\item If $|d_0|=0$, then we have 
\begin{equation}\label{equth1e}
 e^{-\frac{\|F_{10}\|^2_{L_2(\gamma_C)}}{8}}\frac{\alpha}{2\pi}\leq \mathcal{R}.
\end{equation}
\end{enumerate}
\end{theoreme}
\begin{proof}
\textit{Step 1: The  problem is two dimensional}
We shall prove this equality:
\begin{equation}\label{finlamt}
\mathcal{R}=\frac{1}{2}\gamma_2\left (Q_-^a-y^+\right )+\frac{1}{2}\gamma_2\left (Q_-^b-y^-\right ),
\end{equation}
where $Q_-^a$, $Q_-^b$, $y_+$ and $y_-$ will be defined below.  $Q_-^a$ and $Q_-^b$ are two areas of $\R^2$, $y_+$ and $y_-$ are two vectors of $\R^2$ and all these quantities are illustrated Figure \ref{fig:lemprin}. 
In the following we shall use the notation $\tilde{e}_p=\Pi_{G_p^{\bot}}e_p$ for the orthogonal projection of $e_p$ on the orthogonal to $G_p$ in $\R^p$. We will suppose that $\|\tilde{e}_p\|_{\R^p}\neq 0$, since the part of the result concerning $\|\tilde{e}_p\|_{\R^p}=0$ is straightforward. The calculation of $\mathcal{R}$ is intrinsically a calculus in the two dimensional space $M_p$, spanned by $G_p$ and $\tilde{e}_p$. In order to make this fact clear, note that for all $z_1\in M_p$ $z_2\in M_p^{\bot}$ we have: \[
V_{\langle .,G_p+e_p\rangle_{\R^p} +d_0}\setminus V_{\langle .,G_p\rangle_{\R^p}}+z_1+z_2= V_{\langle .,G_p+e_p\rangle_{\R^p} +d_0}\setminus V_{\langle .,G_p\rangle_{\R^p}}+z_1
\]
 and 
 \[
 V_{\langle .,G_p\rangle_{\R^p}}\setminus V_{\langle .,G_p+e_p\rangle_{\R^p} +d_0}+z_1+z_2= V_{\langle .,G_p\rangle_{\R^p}}\setminus V_{\langle .,G_p+e_p\rangle_{\R^p} +d_0}+z_1
 \] 
 (here $M_p^{\bot}$ was the orthogonal of $M_p$ in $\R^p$).
 By the tensorial property of $\gamma_p$ and equation (\ref{tmpp1}), we finally get
\begin{eqnarray}\label{reducdim2}
\mathcal{R} & = & \frac{1}{2}\gamma_2\left (M_p\cap(V_{\langle .\, ,G_p+e_p\rangle_{\R^p} +d_0}\setminus V_{\langle .\, ,G_p\rangle_{\R^p}}-\frac{G_p}{2})\right )\\
& & +\frac{1}{2}\gamma_2\left (M_p\cap( V_{\langle .\, ,G_p\rangle_{\R^p}}\setminus V_{\langle .\, ,G_p+e_p\rangle_{\R^p} +d_0}+\frac{G_p}{2})\right ).
\end{eqnarray}
 Also, in the sequel we will identify $M_p$ with $\R^2$, $D$ and $\hat{D}$ will be the straight lines of $M_p$ with equation $\langle .,G_p\rangle_{\R^p}=0$ and $\langle .,G_p+e_p\rangle_{\R^p} +d_0=0$. It can easily be shown that these lines intersect in $a_p$ given by
\begin{equation}\label{defan}
a_p=-d_0\frac{\tilde{e}_p}{\|\tilde{e}_p\|^2_{\R^p}}.
\end{equation}
Also,
\[V_{\langle .\, ,G_p\rangle_{\R^p}}=V_{\langle.\, -a_p,G_p\rangle_{\R^p}}\;\text{ et }\;V_{\langle .\, ,G_p+e_p\rangle_{\R^p} +d_0}=V_{\langle.\, -a_p,G_p+e_p\rangle_{\R^p}},\]
and with the same calculus that was used to obtain (\ref{tmpp1}), equation (\ref{reducdim2}) becomes:
\begin{eqnarray}
\mathcal{R} & =  &\frac{1}{2}\gamma_2\left (M_p\cap(V_{\langle.\, ,G_p+e_p\rangle_{\R^p}}\setminus V_{\langle.\, ,G_p\rangle_{\R^p}})-\frac{G_p}{2}+a_p\right )\\
& & +\frac{1}{2}\gamma_2\left ( M_p\cap (V_{\langle.\, ,G_p\rangle_{\R^p}}\setminus V_{\langle.\, ,G_p+e_p\rangle_{\R^p}})+\frac{G_p}{2}+a_p\right ).
\end{eqnarray}
Notice that for reasons of symmetry we can assume that $d_0\geq 0$ without loss of generality. In the sequel, we shall use the notation 
\begin{equation}\label{ginoru}
y^+=\frac{G_p}{2}-a_p \text{ et } y^-=-\frac{G_p}{2}-a_p, 
\end{equation}
the coordinates of $y^+$ in the orthonormal coordinate system obtained from the orthogonal coordinate system  $(0,\tilde{e}_p,G_p)$ will be noted $(y_h,y_v)$ and are equal $(\frac{d_0}{\|\tilde{e}_p\|_{\R^p}},\frac{\|G_p\|_{\R^p}}{2})$. We shall also note
\begin{equation}
Q_-^a=M_p\cap(V_{\langle.\, ,G_p+e_p\rangle_{\R^p}}\setminus V_{\langle.\, ,G_p\rangle_{\R^p}})\text{ et }Q_-^b= M_p\cap (V_{\langle.\, ,G_p\rangle_{\R^p}}\setminus V_{\langle.\, ,G_p+e_p\rangle_{\R^p}}).
\end{equation}
We finally derive equation (\ref{finlamt}). From Figure \ref{fig:lemprin}, we notice that replacing $\alpha$ by $-\alpha$, $\mathcal{R}$ does not change; that if  $0<\alpha\leq \pi/2$ then $\mathcal{R}\leq \frac{1}{2}$ and if $\pi \geq \alpha\geq \pi/2$ then $\mathcal{R}_p\geq 1/2$. Also, we will now suppose that $\alpha\in [0,\pi/2]$.\\

\textit{Step 2}. The rest of the proof relies on the following lemma.
\begin{figure}
    \center
    \includegraphics[width=10cm]{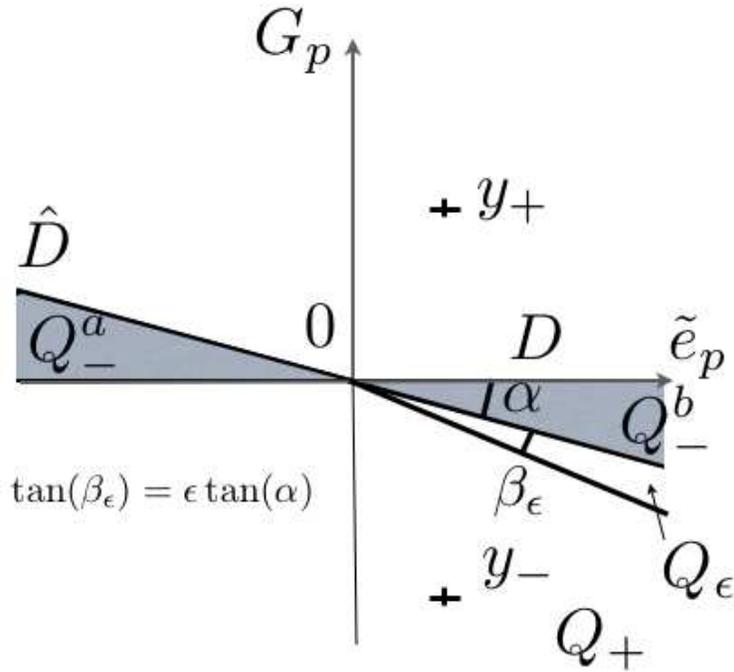}
   \caption{Figure giving the definition of  $Q_{-}^a$, $Q_-^b$, $Q_+$, and $Q_{\epsilon}$ for Lemma \ref{lemprinc}}
   \label{fig:lemprin}
\end{figure}
\begin{Lemme}\label{lemprinc}
Let, $Q_+$ and $Q_{\epsilon}$ be defined by Figure \ref{fig:lemprin} forming, with $Q_-^a$ et $Q_-^b$, a partition of $\R^2$. Let $u=\tan(\alpha)y_h$. We then have 
\begin{itemize}
\item If $y^-\in Q_-$, then
 \[\frac{1}{2}\gamma_1([0;|y_v|])+\frac{\alpha}{2\pi}+\gamma_1([0,\frac{y_v}{2}])\gamma_1\left (\left [0;\left |y_{v}/2\frac{\cos(\alpha)}{\sin(\alpha)}\right |\right]\right ) \leq  \gamma_2(Q_-^b-y^-)\]
\begin{equation}\label{l00}
 \gamma_2(Q_-^b-y^-)\leq \frac{\alpha}{2\pi}+\gamma_1([0;|u|(1+\tan(\alpha))]) ,
\end{equation}
\item If $y^-\in Q_+$, then
\[e^{-\frac{y_v^2}{2}}\frac{1}{2}\left (\frac{1}{2}\gamma_1([0;|u|])+\frac{\alpha}{2\pi}\right )\leq \gamma_2(Q_-^b-y^-)\]
\begin{equation}\label{l01}
\gamma_2(Q_-^b-y^-)\leq e^{-\frac{\epsilon^2y_v^2\cos^2(\alpha)}{2(1+\epsilon)^2}}\left (\gamma_1([0;((1+\tan(\alpha))|u|])+\frac{\alpha}{2\pi}\right ) ,
\end{equation}
\item If $y^-\in Q_{\epsilon}$, then
\[
e^{-\frac{(1+\epsilon)^2|u|^2}{2}}\frac{1}{2}\left (\frac{1}{2}\gamma_1([0;|u|])+\frac{\alpha}{2\pi}\right )\leq \gamma_2(Q_-^b-y^-)
\]
\begin{equation}\label{l20}
\gamma_2(Q_-^b-y^-)\leq \left (\gamma_1([0;(1+\tan(\alpha))|u|])+\frac{\alpha}{2\pi}\right ).
\end{equation}
\item We have concerning $\gamma_2(Q_-^a-y^+)$: 
\begin{equation}\label{l10}
 \gamma_2(Q_-^a-y^+)\leq \gamma_2(Q_-^b-y^-).
\end{equation}
\item Finally, if $y_h=0$, we have 
\begin{equation}\label{l10000}
e^{-\frac{y_v^2}{2}}\frac{\alpha}{2\pi} \leq \gamma_2(Q_-^a-y^+)= \gamma_2(Q_-^b-y^-).
\end{equation}
\end{itemize}
\end{Lemme}
This Lemma will be proven in Subsection \ref{proflemprin}, let us see how it implies Theorem \ref{the:1}. 
Fix $\epsilon=1$ for the rest of the proof (Other values of $\epsilon$ will help us in the proof of Theorem \ref{the:11}). Equation (\ref{l10}) of the lemma implies that
\[\frac{1}{2}\gamma_{2}(Q_-^b-y^-)\leq \mathcal{R}\leq \gamma_{2}(Q_-^b-y^-).\]
Recall that $(y_h,y_v)$ has been defined following equation (\ref{ginoru}) as the coordinates of $y^+$ and that $u=\tan(\alpha)y_h$.
A simple calculation leads to
\[u=|d_0|\frac{\tan(\alpha)}{\|\Pi_{F_{10}^{\bot}}\hat{F}_{10}\|_{L_2(\gamma_C)}}\text{ et }y_v^2=\frac{ \|F_{10}\|_{L_2(\gamma_C)}^2}{4}.\]
If $\frac{1}{2}|\langle G_p ,\hat{G}_p \rangle_{\R^p}|<|d_0|$, we have in the preceding Lemma $y_-\in Q_{-}$ and:
\[\frac{1}{4}\gamma_1\left (\left [0;\frac{\tan(\alpha)\|F_{10}\|_{L_2(\gamma_C)}}{2}\right ]\right )+\frac{\alpha}{4\pi}\leq \mathcal{R}\]
\[\mathcal{R}\leq \frac{\alpha}{2\pi}+\gamma_1\left (\left [0;(1+\tan(\alpha))\frac{|d_0|\tan(\alpha)}{\|\Pi_{F_{10}^{\bot}}\hat{F}_{10}\|_{L_2(\gamma_C)}}\right ]\right ).\] 
The case where $|d_0|<\frac{1}{4}|\langle G_p ,\hat{G}_p \rangle_{\R^p}|$ (which means that $2|u|<|y_v|$) is the case where $y_-\in Q_+$, and we then have:
\[
e^{-\frac{\|F_{10}\|_{L_2(\gamma_C)}^2}{8}} \frac{1}{4}\left (\frac{\alpha}{2\pi}+\frac{1}{2}\gamma_1\left (\left [0;\frac{|d_0|\tan(\alpha)}{\|\Pi_{F_{10}^{\bot}}\hat{F}_{10}\|_{L_2(\gamma_C)}}\right ] \right )\right) \leq \mathcal{R},
\]
and 
\[
\mathcal{R}\leq e^{-\frac{\|F_{10}\|_{L_2(\gamma_C)}^2\cos(\alpha)^2}{32}}\left (\frac{\alpha}{2\pi}+\gamma_1\left (\left [0;\left (1+\tan(\alpha)\right )\frac{|d_0|\tan(\alpha)}{\|\Pi_{F_{10}^{\bot}}\hat{F}_{10}\|_{L_2(\gamma_C)}}\right ] \right )\right).
\]
If  $\frac{1}{4}|\langle G_p ,\hat{G}_p \rangle_{\R^p}|<|d_0|<\frac{1}{2}|\langle G_p ,\hat{G}_p \rangle_{\R^p}|$, (which means that $2|u|>|y_v|>|u|$) we have in the preceding lemma $y_-\in Q_{\epsilon}$ ($\epsilon=1$), and since in this case $|y_v|>|u|>|y_v|/2$, we get:
\[e^{-\frac{\|F_{10}\|_{L_2(\gamma_C)}^2}{2}} \frac{1}{4}\left (\frac{1}{2}\gamma_1\left (\left [0;\frac{\|F_{10}\|_{L_2(\gamma_C)}}{4}\right ]\right ) +\frac{\alpha}{2\pi}\right ) \leq \mathcal{R}\]
and
\[\mathcal{R}\leq \frac{\alpha}{2\pi}+\gamma_1\left (\left [0;\left (1+\tan(\alpha)\right )\frac{|d_0|\tan(\alpha)}{\|\Pi_{F_{10}^{\bot}}\hat{F}_{10}\|_{L_2(\gamma_C)}}\right ]\right ).\]
This ends the proof of Theorem \ref{the:1}.
\end{proof}
\subsection{Proof of Theorem \ref{the:11}}

Theorem (\ref{the:11}) is also a consequence of the preceding Lemma. We will use the preceding lemma while tuning the value of $\epsilon$. We use without restating them the definitions given before the preceding lemma.\\
 \indent Let us assume that $\frac{2|d_0|}{|\langle F_{10} ,\hat{F}_{10} \rangle_{L_2(\gamma_C)}|}$ has an inferior limit $a<1$. Then, there exists $\epsilon>0$ such that $y^+$ and $y^-$ (defined by (\ref{ginoru})) belong to $Q_+$ (for $\|F_{10}\|_{L_2}\cos(\alpha)$ large enough), then equation (\ref{l01}) implies that
 \[
 \mathcal{R}\leq e^{-\frac{\epsilon^2\|F_{10}\|_{L_2}^2\cos^2(\alpha)}{2(1+\epsilon)^2}}\left (1+\frac{|\alpha|}{2\pi}\right ),
\]
and $\mathcal{R}$ tends to $0$ when $\|F_{10}\|_{L_2}^2\cos^2(\alpha)$ tends to infinity.\\
\indent If now $\frac{2|d_0|}{|\langle F_{10} ,\hat{F}_{10} \rangle_{L_2(\gamma_C)}|}$ tends to $a>1$, then $y^+$ or $y^-$ (given by (\ref{ginoru})) belongs to  $Q_-$  (for $\|F_{10}\|_{L_2}\cos(\alpha)$ large enough). And since in this case   equation (\ref{l00}) leads to
 \begin{equation}\label{midort}
 \mathcal{R}\geq \frac{1}{4}\left (\frac{1}{2}\gamma_1([0;\|F_{10}\|_{L_2}/2])\right .
 \end{equation}
 \[+\left .\gamma_1\left (\left [0;\frac{\|F_{10}\|_{L_2}\cos(\alpha)}{4\sin(\alpha)}\right ]\right )\gamma_1([0;\|F_{10}\|_{L_2}/4])+\frac{\alpha}{2\pi}\right ),\]
we obtain the desired result by letting $\|F_{10}\|_{L_2}$ tend to infinity. One has to observe that $\alpha$ depends on $\|F_{10}\|_{L_2}$ and that the limit values $\alpha=\pi/2$ and $\alpha=0$ require the use of different terms in inequality (\ref{midort}). This ends the proof of Theorem \ref{the:11}.

\subsection{Proof of Lemma \ref{lemprinc}}\label{proflemprin}

This proof is the central part of this section. It is mostly geometrical, and require only is the following four properties  (given by Figure \ref{fig:pro}):
\begin{itemize}
\item \underline{Property $1$}. If $A\subset \R^2$ between the two half straight lines $(0,u)$ and $(0,v)$ such that Angle$(u,v)=\alpha$, then $\gamma_2(A)=\frac{\alpha}{2\pi}$. This result follows directly from rotational invariance of the gaussian measure. Such an area will be called an angular portion of size $\alpha$ and centre $0$. 
\item \underline{Properties $2$ and $3$}. Let $y\in\R^2$, $D$ a straight line of $\R^2$, $b$ the orthogonal projection of $y$ on $D$ and $h$ the distance from $y$ to $D$. If $A\subset \R^2$ and $A$ is included in the half plan delimited by $D$ that does not contain $y$, then $\gamma_2(A-y)\leq e^{-h^{2}/2}\gamma_2(A-b)$. This is property $2$. If $A\subset\R^p$ is included in the half plan delimited by $D$ that contains $y$ then $\gamma_2(A-y)\geq e^{-h^{2}/2}\gamma_2(A-b)$.This is property $3$. 
\item \underline{Property $4$}.  
If $A=[0;d]\times [0;\infty[$ (see Figure \ref{fig:pro}) then $\gamma_2(A)=\frac{1}{2}\gamma_1([0;d])$. Such a rectangle will be called an infinite rectangle of origin $0$ and height $d$. 
\end{itemize}
\begin{figure}
    \center
    \includegraphics[width=15cm]{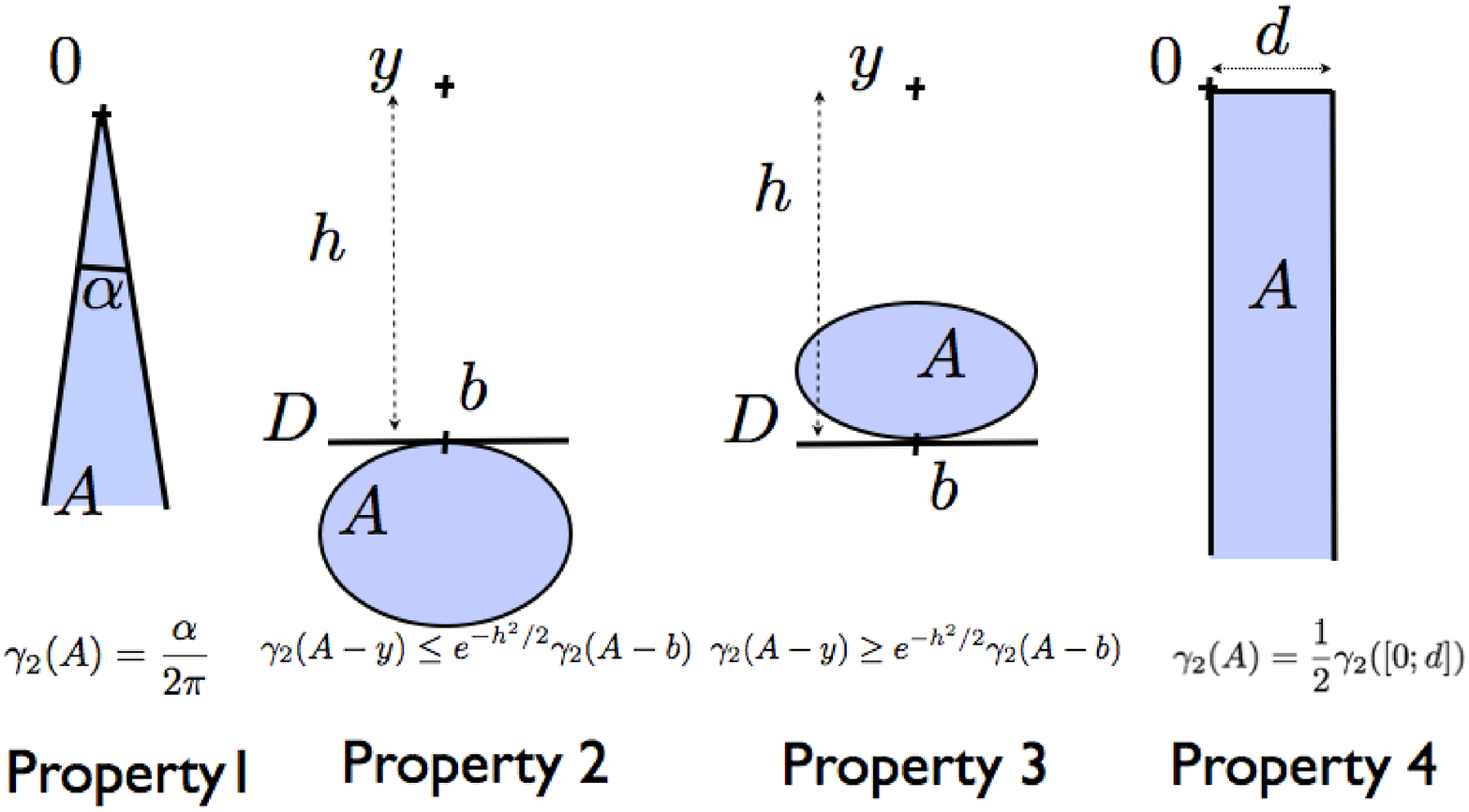}
   \caption{The four properties used in the proof }
   \label{fig:pro}
\end{figure}
 We will note $q$ and $\hat{q}$ the orthogonal projections of $y$ on $D$ and $\hat{D}$.
The properties $2$ and $3$ are well known but for the sake off completeness we recall their proof. It suffices to note that

\[\gamma_2(A-y)=\int_{x\in A}\frac{1}{2\pi}e^{-\frac{\|x-y\|_{\R^2}^2}{2}}dx=e^{-\frac{h^2}{2}}\int_{x\in A}\frac{1}{2\pi}e^{-\frac{\|x-b\|_{\R^2}^2}{2}}e^{\langle x-b,y-b \rangle_{\R^2}}dx , \]
and that $x\in A$ implies $\langle x-b,y-b \rangle_{\R^2}\leq 0$ for property $2$ and $\langle x-b,y-b \rangle_{\R^2}\geq 0$ for property $3$.\\

We are now going to distinguish between a number of cases and, in each of them, use the announced properties.
First note that the inequality concerning $y^+$ is trivial. Figure \ref{fig:pro1} and \ref{fig:lemprin} will be useful in the following.
 \begin{figure}
    \center
    \includegraphics[width=7cm]{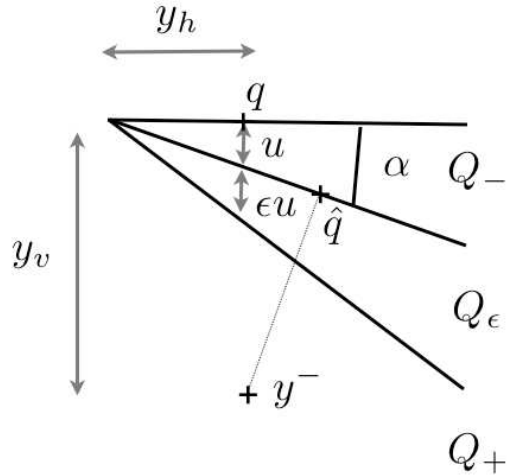}
   \caption{Figure to visualize de proof}
   \label{fig:pro1}
\end{figure}
\paragraph{Case $y^-\in Q_-^b$.} In this case $|y_v|\leq |u|$. One can include in $Q_-^b$ the disjoint union of an infinite rectangle of origin $y^-$, and height $|y_v|$ ; an angular portion of size $\alpha$ and centre $y^-$ ; and a rectangle with vertex $y^-$ height $|y_v|/2$ and length $|y_{v}/2\frac{\cos(\alpha)}{\sin(\alpha)}|$.
Using properties $4$ and $1$, we then get:
\begin{equation}\label{cas11min}
\frac{1}{2}\gamma_1([0;|y_v|])+\frac{\alpha}{2\pi}+\gamma_1([0,\frac{y_v}{2}])\gamma_1\left (\left [0;\left |y_{v}/2\frac{\cos(\alpha)}{\sin(\alpha)}\right |\right]\right )\leq \gamma_2(Q_-^b-y^-).
\end{equation}
On the other hand, $Q_-^b$ can be included in the disjoint union of an angular portioin with centre $y^-$, of two infinite rectangles with height less than or equal to $|u|\tan(\alpha)$ and of two infinite rectangle of height lower or equal to $|u|$. Also, properties $1$ and $4$ imply:
\begin{equation}
\gamma_2(Q_-^b-y^-)\leq \frac{\alpha}{2\pi}+\gamma_1([0;|u|(1+\tan(\alpha))]).
\end{equation}
\paragraph{Case $y^-\in Q_+$.}
In this case $|y_v|>(1+\epsilon)|u|$, $y^-$ is at a distance $|y_v|$ from $D$ and at a distance $(|y_v|-|u|)\cos(\alpha)\geq \frac{\epsilon}{1+\epsilon}|y_v|\cos(\alpha)$ from $\hat{D}$.  Properties $2$ and $3$ imply:
\begin{equation}\label{cas12}
e^{-\frac{y_v^2}{2}}\gamma_2(Q_-^b-q)\leq \gamma_2(Q_-^b-y^-)\leq e^{-\frac{\epsilon^2y_v^2\cos^2(\alpha)}{2(1+\epsilon)^2}}\gamma_2(Q_-^b-\hat{q}).
\end{equation}

One can include in $Q_-^b$ an angular portion of size $\alpha$ and with centre $q$ or an infinite rectangle of origin $y$ and height $|u|$. Also, properties $1$ and $4$ imply, with (\ref{cas12}) and the fact that $\max(a,b)\geq \frac{a+b}{2}$ the equation:
\[\frac{1}{2}\left ( \frac{1}{2}\gamma_1([0;|u|])+\frac{\alpha}{2\pi}\right )\leq \gamma_2(Q_-^b-q).\]
 The set $Q_-^b$ can be included in the union of an angular portion of size $\alpha$ centred  in $\hat{q}$ and of two infinite rectangles of origin $\hat{q}$ and height $|u|(1+\tan(\alpha))$. Also, properties $1$ and $4$ together with (\ref{cas12}) and $\max(a,b)\geq \frac{a+b}{2}$ imply the following equation:

\begin{equation}\label{cas13}
e^{-\frac{y_v^2}{2}}\frac{1}{2}\left ( \frac{1}{2}\gamma_1([0;|u|])+\frac{\alpha}{2\pi}\right )\leq \gamma_2(Q_-^b-y^-),\end{equation}
\[\gamma_2(Q_-^b-y^-)\leq e^{-\frac{ \epsilon^2y_v^2\cos^2(\alpha)}{2(1+\epsilon)^2}}\left ( \gamma_1([0;|u|(1+\tan(\alpha))])+\frac{\alpha}{2\pi}\right ).
\]
\paragraph{Case $y^-\in Q_{\epsilon}$.}
In this case $(1+\epsilon)|u|>|y_v|>|u|$, $y^-$ is at a distance $|y_v|\leq (1+\epsilon)|u|$ from $D$ and at a distance $(|y_v|-|u|)\cos(\alpha)\geq 0$ from $\hat{D}$.  Properties $2$ and $3$ imply
\begin{equation}\label{cas14}
e^{-\frac{(1+\epsilon)^2|u|^2}{2}}\gamma_2(Q_-^b-q)\leq \gamma_2(Q_-^b-y^-)\leq \gamma_2(Q_-^b-\hat{q}).
\end{equation}
from which we deduce the following inequality in the same way as in the preceding paragraph:

\begin{equation}\label{cas15}
e^{-\frac{(1+\epsilon)^2|u|^2}{2}}\frac{1}{2}\left ( \frac{1}{2}\gamma_1\left (\left [0;|u|\right ]\right )+\frac{\alpha}{2\pi}\right )\leq \gamma_2(Q_-^b-y^-),
\end{equation}
\[ \gamma_2(Q_-^b-y^-)\leq \left ( \gamma_1([0;|u|(1+\tan(\alpha))])+\frac{\alpha}{2\pi}\right ).\]
This ends the proof of the Lemma.

\begin{remarque}[On log-concave measures]
It is natural to ask which type of probability measure satisfies the four properties used.  Concerning property $2$, it is possible to consider measures that are not gaussian. Suppose that $\mu$ is a probability measure on $\R^p$ with positive density, $ae^{-\phi}$ with respect to the Lebesgue measure, where $\phi$ is strictly convex in the sense that their exists $c>0$ such that for all $x,y\in \R^p$
\begin{equation}\label{convex}
\phi(x)+\phi(y)-2\phi\left (\frac{x+y}{2}\right)\geq \frac{c}{2}\|x-y\|^2_{\R^p},
\end{equation}
$\phi(0)=0=\Arginf\phi$, $a$ is a positive constant and $\phi$ is radial: there exists a function $\psi$ from $\R$ to $\R$ such that $\phi(x)=\psi(\|x\|)$.
  Let $y\in\R^p$, $D$ be a hyperplane of $\R^p$, $b$ the orthogonal projection of $y$ on $D$, $h$ the distance from $y$ to $D$ and $A\subset \R^p$ included in the half space delimited by $D$ which does not contain $y$. One can show (see proposition $3.3.1$ p126 in \cite{Girard:2008wd}) that  
  \[\mu(A-y)\leq e^{-c\frac{h^2}{2}}\mu(A-b).\]
  \end{remarque}
\subsection{Proof of Theorem \ref{the:32}}\label{dthe32}

\begin{proof}
The second equation of the Theorem results directly from equation (\ref{eque:th1}) in Theorem \ref{the:1}. To show the first equation of the Theorem, we will four cases. Case number $4$ is the important one that relies on the use of Theorem \ref{the:1}. The other cases rely on verifying that the right member of the first equation of the Theorem is not too small.
\begin{enumerate} 
\item
Case where $\langle \hat{F}_{10},F_{10}\rangle_{L_2(\gamma_C)}<0$.\\
Let us note that because $\mathcal{R}$ is a probability, we have $\mathcal{R}\leq 1$. In addition,  
\[\mathcal{E}\geq \|F_{10}-\hat{F}_{10}\|_{L_2(\gamma_C)}\geq \|F_{10}\|_{L_2(\gamma_C)}.\]
which implies that  $\mathcal{R}_p\leq \frac{\mathcal{E}}{\|F_{10}\|_{L_2(\gamma_C)}}$. \\
\item Case where $\langle \hat{F}_{10},F_{10}\rangle_{L_2(\gamma_C)}>0$ and $\|\hat{F}_{10}\|_{L_2(\gamma_C)}\leq \frac{1}{2}\|F_{10}\|_{L_2(\gamma_C)}$.\\
Recall that $\mathcal{R}$ is upper bounded by $\frac{1}{2}$ when $\langle \hat{F}_{10},F_{10}\rangle_{L_2(\gamma_C)}>0$ (see Theorem \ref{the:1}, it is the case where $\alpha$ defined by (\ref{alpha}) satisfies $-\pi/2\leq \alpha\leq \pi/2$). \\
In addition, the inequality $\|\hat{F}_{10}\|_{L_2(\gamma_C)}\leq \frac{1}{2}\|F_{10}\|_{L_2(\gamma_C)}$ implies 
\[\mathcal{E}\geq \frac{1}{2} \|F_{10}\|_{L_2(\gamma_C)},\] 
and as a consequence $\mathcal{R}_p\leq \frac{1}{2}$ implies that $\mathcal{R}_p\leq \frac{\mathcal{E}}{\|F_{10}\|_{L_2(\gamma_C)}}$. 
\item Case where $\langle \hat{F}_{10},F_{10}\rangle_{L_2(\gamma_C)}>0$, $\|\hat{F}_{10}\|_{L_2(\gamma_C)}\geq \frac{1}{2}\|F_{10}\|_{L_2(\gamma_C)}$ et $\frac{\pi}{2}>\alpha>\frac{\pi}{4}$ (recall that $\alpha$ has been defined by \ref{alpha}). \\
Since $\frac{\pi}{2}>\alpha>\frac{\pi}{4}$, we have $\cos(\alpha)\leq \frac{1}{2}$ and as a consequence and with the help of (\ref{alpha}):
\[\langle \hat{F}_{10},F_{10}\rangle_{L_2(\gamma_C)}\leq \frac{\sqrt{2}}{2}\|\hat{F}_{10}\|_{L_2(\gamma_C)}\|F_{10}\|_{L_2(\gamma_C)}.\]
Under this last constraint, we have
\[\min_{\hat{F}_{10}}\|F_{10}-\hat{F_{10}}\|_{L_2(\gamma_C)}^2=\min_{\alpha}\left ((1-\alpha)^2+\alpha^2\right )\|F_{10}\|_{L_2(\gamma_C)}^2=\|F_{10}\|_{L_2(\gamma_C)}^2, \]
which again implies $\mathcal{R}_p\leq \frac{\mathcal{E}}{\|F_{10}\|_{L_2(\gamma_C)}}$. 
\item Case where $\langle \hat{F}_{10},F_{10}\rangle_{L_2(\gamma_C)}>0$, $\|\hat{F}_{10}\|_{L_2(\gamma_C)}\geq \frac{1}{2}\|F_{10}\|_{L_2(\gamma_C)}$ and $\alpha<\frac{\pi}{4}$.\\

Since $\alpha\in [0,\frac{\pi}{4}]$, the concavity of the $\sin$ function gives
\[\frac{\alpha}{\pi}\leq \frac{\sin(\alpha)}{2\sqrt{2}}.\]
In addition, the relation $\|\hat{F}_{10}\|_{L_2(\gamma_C)}\geq \frac{1}{2}\|F_{10}\|_{L_2(\gamma_C)}$ implies that
\[
\sin(\alpha)=\frac{\|\Pi_{F_{10}^{\bot}}\hat{F}_{10}\|_{L_2(\gamma_C)}}{\|\hat{F}_{10}\|_{L_2(\gamma_C)}}\leq \frac{2\|F_{10}-\hat{F}_{10}\|_{L_2(\gamma_C)}}{\|F_{10}\|_{L_2(\gamma_C)}},
\]
(the first inequality is a trigonometric formula). Finally, we obtain:
\begin{equation}\label{alppi}
\frac{\alpha}{\pi}\leq \frac{\|F_{10}-\hat{F}_{10}\|_{L_2(\gamma_C)}}{\sqrt{2}\|F_{10}\|_{L_2(\gamma_C)}}.
\end{equation}
Recall that $d_0=\langle \hat{F}_{10},\hat{s}_{10}-s_{10}\rangle_{\R^p}$. The equality defining $\alpha$ (\ref{alpha}) and the fact that $\cos(\alpha)\geq \frac{\sqrt{2}}{2}$ now imply:
\begin{align*}
\frac{|d_0|\tan(\alpha)}{\|\Pi_{F_{10}^{\bot}}\hat{F}_{10}\|_{L_2(\gamma_C)}} &\leq  \sqrt{2}|d_0|\frac{\sin(\alpha)}{\|\Pi_{F_{10}^{\bot}}\hat{F}_{10}\|_{L_2(\gamma_C)}} \text{ (since }\cos(\alpha)\geq \frac{\sqrt{2}}{2})\\
&= \frac{\sqrt{2}|d_0|}{\|\hat{F}_{10}\|_{L_2(\gamma_C)}} \text{ (from a trigonometric formula)} . \\
\end{align*}
Also, noticing that $\gamma_{1}([0;u])\leq \frac{u}{\sqrt{2\pi}}$, and that $\tan(\alpha)\leq 1$, we get:
\begin{equation}\label{alppi2}
\gamma_{1}\left (\left [0;(1+\tan(\alpha))\frac{|d_0|\tan(\alpha)}{\|\Pi_{F_{10}^{\bot}}\hat{F}_{10}\|_{L_2(\gamma_C)}}\right ] \right )\leq \gamma_{1}\left (\left [0;\frac{2\sqrt{2}|d_0|}{\|\hat{F}_{10}\|_{L_2(\gamma_C)}}\right ] \right )\end{equation} 
\[\hspace{3cm}\leq \frac{2|d_0|}{\sqrt{\pi}\|\hat{F}_{10}\|_{L_2(\gamma_C)}}.
\]
 In the cases $1$, $2$ and $3$ of Theorem \ref{the:1},    
because $\tan(\alpha)\leq 1$ ($\alpha\leq \frac{\pi}{4}$), the equations (\ref{alppi}), (\ref{alppi2}), (\ref{eque:th1}),(\ref{eque:th02}) imply:
\[\mathcal{R}\leq \frac{\mathcal{E}}{\|F_{10}\|_{L_2(\gamma_C)}}.\]
This ends the proof of Theorem \ref{the:32}.
\end{enumerate} 
\end{proof}

\section{A general scheme to solve Problem \ref{Pb1}}\label{proofth2}
\subsection{Introduction and main result}
\paragraph{Presentation of the main ideas.}
In this section, we will prove results concerning the QDA procedure. Recall that the learning error $\mathcal{R}$ (The probability to misclassify data with a given rule when the optimal rule gives a correct classiication) satisfies:
\begin{equation}\label{riqsue}
\mathcal{R}\leq \frac{1}{2}\left ( P_1(X\in V_{\hat{\mathcal{L}}_{10}^Q}\triangle V_{\mathcal{L}_{10}^Q})+ P_0(X\in V_{\hat{\mathcal{L}}_{10}^Q}\triangle V_{\mathcal{L}_{10}^Q})\right )
\end{equation}
 (If $f:\X\rightarrow \R$, $V_f$ is defined by (\ref{de:rg}) at the beginning of the preceding section).
Indeed, the event $X\in V_{\hat{\mathcal{L}}_{10}^Q}\triangle V_{\mathcal{L}_{10}^Q}$ corresponds to the case where decisions (good or erroneous) taken by the optimal rule and the plug-in rule are different. 
\begin{remarque}
In the case of procedure LDA, we had
\[
\mathcal{R}=\frac{1}{2}\left ( \gamma_{C,s_{10}} \left (X \in \hat{V} \setminus V -\frac{m_{10}}{2}\right )+\gamma_{C,s_{10}} \left (X \in V \setminus \hat{V}+\frac{m_{10}}{2}\right )\right ).
\]
From this equation, one can easily deduce that
\[2\mathcal{R}=\frac{1}{2}\left ( \gamma_{C,s_{10}} \left (X \in \hat{V} \triangle V -\frac{m_{10}}{2}\right )+\gamma_{C,s_{10}} \left (X \in V \triangle \hat{V}+\frac{m_{10}}{2}\right )\right ),\]
and as a consequence:  
\begin{equation}
2\mathcal{R}=\frac{1}{2}\left ( P_1(X\in V_{\hat{\mathcal{L}}_{10}^A}\triangle V_{\mathcal{L}_{10}^A})+ P_0(X\in V_{\hat{\mathcal{L}}_{10}^A}\triangle V_{\mathcal{L}_{10}^A})\right ). 
\end{equation}
It is less obvious that this type of relation is true in the "quadratic case. It's seems less obvious.
\end{remarque}
 In subsection \ref{emlrpoit} we will present a technique to put an upper bound on the probabilities like $P(V_{f}\triangle V_{f+\delta})$. In this type of quantity, we shall call perturbation function the measurable function $\delta$ (which can be thought as a small function) and optimal frontier function the measurable function $f$ from $\X$ to $\R$. In the case of the QDA, the results obtained are consequences of Theorem \ref{thedeux} given in the next paragraph, with frontier function $f=\mathcal{L}_{10}^Q$ and perturbation function $\delta=\hat{\mathcal{L}}_{10}^Q-\mathcal{L}_{10}^Q$.\\

\paragraph{A general result concerning quadratic perturbation of a quadratic rule.}In the sequel we need to introduce some quantities related to gaussian measure in separable Banach spaces, and $\X$ is a separable Banach Space. We refer to  \cite{Bogachev:1998fk} and its section on measurable polynomials for a rigourous treatment of the subject. The Hilbert Space of measurable affine function from $\X$ to $\R$ with finite $L_2(\gamma_{C,m})$ norm and null integral with respect to $\gamma_{C,m}$ will be denoted by ${\X}_{\gamma_{C,m}}^*$. The Hilbert space of measurable quadratic form in $L_2(\gamma_{C,m})$ with null integral with respect to $\gamma_{C,m}$ will be denoted $E_2(\gamma_{C,m})$. The space of measurable quadratic forms in $L_2(\gamma_{C,m})$ will be denoted by $\X_{2\gamma}^*$ and we have the classical gaussian chaos decomposition in $L_2(\gamma_{C,m})$:
\[\X_{2\gamma}^*=\{Cte\}\oplus {\X}_{\gamma_{C,m}}^* \oplus E_{2}(\gamma_{C,m}).\]
 In infinite dimension $H(\gamma_{C,m})$ is the reproducing kernel Hilbert space associated to $\gamma_{C,m}$, in finite dimension ($\X=\R^p$), we have (if $C$ is of full rank) $H(\gamma_{C,m})=\R^p$. Recall that to each Hilbert-Schmidt operator $A$ on $H(\gamma_{C,m})$, one can associate the measurable element of $E_{2}(\gamma_{C,m})$ and that each element of $E_{2}(\gamma_{C,m})$ is associated to a unique Hilbert-Schmidt operator on $H(\gamma_{C,m})$. In finite dimension, if $C$ is of full rank: 
\begin{align*}
q_A^{\gamma_{C,m}}(x)&=q_{C^{-1/2}AC^{-1/2}}(x-m)-\int_{\X} q_{C^{-1/2}AC^{-1/2}}(x-m)\gamma_{C,m}(dx)\\
&(\text{ recall that } q_{A}(x)=\langle Ax,x\rangle_{\R^p} )\\
&= \langle AC^{-1/2}(x-m),C^{-1/2}(x-m) \rangle_{\R^p}-\sum_{i=1}^p\lambda_i,
\end{align*}
where $(\lambda_i)_{i=1,\dots,p}$ is the vector of the eigenvalues of $A$.
\begin{theoreme}\label{thedeux}
  Let ${\X}$ be a separable Banach space, $\gamma_{C,m}$ be a gaussian measure on ${\X}$ with mean $m$ and covariance $C$. Let $A$ and $D$ be $2$ symmetric Hilbert-Schmidt operators on $H(\gamma_{C,m})$, $F,d\in {\X}_{\gamma_{C,m}}^*$, and $c,d_0\in \R$. Let
  \[f(x)=c+ F(x) +q_A^{\gamma_{C,m}}(x)\;\;\text{and}\;\; \delta(x)=d_0+d(x)+q_D^{\gamma_{C,m}}(x)\] 
  be the function defining $V_f$ and $V_{f+\delta}$
 (If $g:\X\rightarrow \R$, $V_g$ is defined by equation (\ref{de:rg})). Finally, let $r,R\in \R$  be such that  $R>r>0$. 
  \begin{enumerate}
 \item  Assume that $r\leq \|f\|_{L_2(\gamma_{C,m})}$. Then, for all $q\in ]0,1[$, there exists $c_1(r,q)>0$ (that only depends on $r$ and $R$) such that
   \begin{equation}
   \gamma_{C,m}(V_f\triangle V_{f+\delta})\leq c_1(r,q) \|\delta\|_{L_2(\gamma_{C,m})}^{q/3} .
     \end{equation} 
  \item If $|\E_{L_2(\gamma_{C,m})}[f]|>r$ and $\|f\|_{L_2(\gamma_{C,m})}$, then, for all $q\in ]0,1[$, there exists 
   $c_2(r,q)>0$ (that only depends on $r$ and $R$) such that
   \begin{equation}
   \gamma_{C,m}(V_f\triangle V_{f+\delta})\leq c_2(r,q) \|\delta\|_{L_2(\gamma_{C,m})}^{2q/7} .
     \end{equation} 

  \end{enumerate}

    \end{theoreme} 

The two following subsections are devoted to the proof of this theorem.  Subsection \ref{emlrpoit} presents a general methodology to obtain this type of result, and in Section \ref{sduytzr}, we apply this methodology to obtain Theorem \ref{thedeux}.
\subsection{Decomposition of the domain}\label{emlrpoit}
We will give an upper bound to the probability that $X\in V_{f}\Delta V_{f+\delta}$. In the cases we have in mind, this set is essentially composed of elements for which $\delta$ takes large values or $f$ is near zero. Also, we shall bound the  measure of areas on which
\begin{enumerate}
\item  the perturbation is large (with large deviation inequality),
\item  $|f|$ is small (with an inequality such as $P(|f(X)|\leq \epsilon)\leq g(\epsilon)$).
\end{enumerate}
 Lemma \ref{fonda} that follows is based on the two following assumptions.\\
\begin{enumerate}
\item \underline{Assumption $A_1$}. It exists $c_0,c_1>0$, $h_{\delta}:\R^+\rightarrow \R^+$ non decreasing such that $h_{\delta}(0)=0$ , $\lim_{s\rightarrow\infty} h_{\delta}(s)=\infty$ and
\begin{equation}\label{GD}
\forall s>0, \;\; P\left ( |\delta(X)-\E[\delta(X)]| \geq c_0 h_{\delta}(s)\right )\leq c_1 e^{-\frac{s^2}{2}}.
\end{equation}

\item \underline{Assumption $A_2$}.
It exists $\beta>0$ and $c_2>0$ such that 

\begin{equation}\label{SD}
\forall \epsilon>0,\;\; P\left ( |f(X)|\leq \epsilon\right )\leq c_2 \epsilon^{\beta}.
\end{equation}

\end{enumerate}
\begin{remarque}
The function $h_{\delta}$ of Assumption $A_1$ will help us in measuring the effect of a perturbation $\delta$.
\end{remarque}
\begin{Lemme}\label{fonda}
Under Assumption $A_1$ (\ref{GD}) and $A_2$ (\ref{SD}), for all $q\in]0;1[$ we have:    
\begin{align*}
P(X\in V_f\Delta V_{f+\delta})\leq & c_1^{1-q}c_2|\E_{P}[\delta(X)]|^{q\beta}\\
& +\sqrt{\frac{2\pi}{1-q}}\frac{c_2c_1^{1-q}}{2}\E\left [\left (c_{0}h_{\delta}\left (\frac{|\xi|}{\sqrt{1-q}}+1\right )+|\E_{P}[\delta(X)]|\right )^{q\beta}\right ],
\end{align*}
where $\xi$ is a centred real gaussian random variable with variance $1$.
\end{Lemme}
\begin{proof} 
Recall that $V_{f}=\{ x : f(x)\geq 0\}$.
\[P(X \in V_f\Delta V_{f+\delta})=\]
\[P\left (-(\delta(X)-\E[\delta(X)])-\E[\delta(X)]\leq f(X)\leq 0 \right .\]
\[\hspace{1cm}\left .\text{ or } 0\leq f(X)\leq (\delta(X)-\E[\delta(X)])+ \E[\delta(X)]\right ),\]
also,
\[P(X \in V_f\Delta V_{f+\delta}) \leq P(U),\]
\[\text{ where } \;\; U=\left \{|f(X)|\leq |\delta(X)-\E[\delta(X)]|+|\E[\delta(X)]|\right \}.\]
Define $B_j=\{c_0h_{\delta}(j)\leq |\delta(X)-\E[\delta(X)]|< c_0h_{\delta}(j+1)\}$ for $j\in \N$. This family of events permits us to recover all possible events. \\

We observe that
\[P(U)= \sum_{j\geq 0}P(U\cap B_j),\]

and then using the  Holder inequality, ( $p+q=1$) we get:
\[P(U)\leq  \sum_{j\geq 0}P(U\cap B_j)^{q}P(B_j)^{p}.\]  
 It follows that
\[
P(X \in V_f\Delta V_{f+\delta}) 
\]
\begin{align*}&\leq \sum_{j}P\left (|f(X)|\leq |\E[\delta(X)]|+c_0h_{\delta}(j+1)  \right )^{q}P\left (|\delta(X)-\E[\delta(X)]| \geq c_0h_{\delta}(j)\right )^{1-q}\\
 &\leq c_2c_1^{1-q}\sum_{j\geq 0}\left (|\E[\delta(X)]|+c_0h_{\delta}(j+1)\right)^{q\beta}e^{-\frac{(1-q)j^2}{2}},\\
&(\text{ from assumption A1 and A2 })\\
\end{align*}
\[\leq c_2c_1^{1-q}\left (|\E[\delta(X)]|^{q\beta_{0}} \right . \]
\[\hspace{1cm}\left . +\sqrt{\frac{2\pi}{1-q}}\int_{0}^{\infty}\left  (h_{\delta}(x+1)+|\E[\delta(X)]|\right )^{q\beta} \sqrt{\frac{1-q}{2\pi}}e^{-\frac{(1-q)x^2}{2}}dx\right )
\]
which implies the desired result.
\end{proof}
\begin{Lemme}\label{lclaim1}
Let $\delta_1,\dots,\delta_k$ be $k$ perturbations satisfying assumption $A_1$ defined by equation (\ref{GD}) with the error functions $h_{\delta_1},\dots,h_{\delta_k}$. Then, if $h_{\delta}=\sum_{i=1}^kh_{\delta_i}$, there exists $c_0(k),c_1(k)>0$ such that 
\begin{equation}
\forall s>0\;\; P\left (|\delta-\E(\delta)|\geq c_0 h_{\delta}(s)\right )\leq c_1e^{-\frac{s^{2}}{2}}.
\end{equation}
\end{Lemme}
\begin{proof}
Recall that for all $i$, $h_{\delta_i}\geq 0$. Let us fix $s>0$. The proof relies on the pigeonhole principle. Indeed, if $\sum_{i=1}^k|\delta_i-\E[\delta_i]|\geq k\sum_{i=1}^kc_{0i}h_{\delta_i}(s)$ then there exists $i_{0}\in \{1,\dots,k\}$ such that $|\delta_{i_0}-\E[\delta_{i_0}]|\geq \sum_{i=1}^kc_{0i}h_{\delta_i}(s)$. If we fix $c_{0}=k\max c_{0i}$, we then have
\begin{align*}
P\left (\left |\sum_{i=1}^k\delta_i-\E[\delta_i]\right |\geq \right .&\left . c_0 \sum_{i=1}^kh_{\delta_i}(s)\right )\leq P\left (\sum_{i=1}^k|\delta_i-\E[\delta_i]|\geq k\sum_{i=1}^kc_{0i}h_{\delta_i}(s)\right )\\
&(\text{ from the triangle inequality and the fact that } \\
& c_{0}\sum_{i=1}^kh_{\delta_i}(s)\geq k\sum_{i=1}^kc_{0i}h_{\delta_i}(s) \text{ )}\\
& \leq P\left (\exists i_0\in \{1,\dots,k\} \; : \; |\delta_{i_0}-\E[\delta_{i_0}]|\geq \sum_{i=1}^kc_{0i}h_{\delta_i}(s)\right )\\
&\text{ (pigeon hole principle) }\\
&\leq \sum_{i=1}^kP\left (|\delta_i-\E[\delta_i]|\geq c_{0i}h_{\delta_{i}}(s)\right )\\
&\text{(subadditivity of probability)}\\
&\leq \sum_{i=1}^k c_{1i}e^{-\frac{s^{2}}{2}}\\
& (h_{\delta_i} \text{ satisfies assumption } A_1),
\end{align*}
which ends the proof.
\end{proof}
The results that allow us to verify assumption A2 are presented in Section \ref{pourhypA2}. We now recall some standard large deviation results that allow us to verify assumption A1. 
\subsection{Large deviation}
In the case where $\delta$ is linear or Lipschits, the following classical result (see for example \cite{Bogachev:1998fk} (p174)) allows us to check assumption $A_1$. 
\begin{theoreme}\label{thlipis}
Let $\gamma=\gamma_{C}$ be a gaussian measure of covariance $C$ on ${\X}$ a separable Banach Space, $H=H(\gamma)$ be the associated reproducing kernel Hilbert Space,  $\delta:{\X}\rightarrow \R$ a function such that there exists $N(\delta)>0$ with 
\begin{equation}
|\delta(x+h)-\delta(x)|\leq N(\delta) |h|_{H(\gamma)} \;\;\forall h\in H(\gamma) \;\;\gamma-ps.
\end{equation}
Then
\begin{equation}
\forall s>0\;\; \gamma\left (x\in {\X}\;\; :|\delta(x)-\int\delta(x) d\gamma|>s\right )\leq 2 e^{-\frac{s^2}{2N(\delta)^2}}
\end{equation}
\end{theoreme}
In the case where $\delta$ is quadratic, the following result from Massart and Laurent \cite{Laurent:2000fk} (Lemma 1 p1325 ) will help us to check assumption $A_1$.
\begin{theoreme}\label{proplaurent}
If $D=Diag(d_1,\dots,d_p)$ and $q_D(x)=\langle Dx,x\rangle_{\R^p}$, then
\begin{equation}\label{equ1}
\gamma_{p}\left (x\in \R^p \;:\; q_D(x)-\int_{\R^p}q_D(x)\gamma_{p}(dx)\geq \frac{s}{2} \|q_D\|_{L_2(\gamma_{p})}+\sup_{i}|d_i|s^2 \right )\leq e^{-\frac{s^2}{2}}
\end{equation} 
\begin{equation}\label{equ2}
\gamma_{p}\left (x\in \R^p \;\; :\;  q_D(x)-\int_{\R^p}q_D(x)\gamma_{p}(dx)\leq -\frac{s}{2} \|q_D\|_{L_2(\gamma_{p})}\right )\leq e^{-\frac{s^2}{2}}
\end{equation} 
 \end{theoreme}
 As a consequence, assumption $A_1$ is satisfied with $h_{\delta}(s)=\frac{s}{2}\|q_D\|_{L_2(\gamma_{p})}+s^2\sup_{i}|d_i|)\leq \|q_D\|_{L_2(\gamma_{p})}(\frac{s}{2}+s^2)$.\\
The use we will make of these results is entirely contained in the following corollary.
 \begin{corollaire}
Let $\X$ be a separable Banach space, $\gamma$ a gaussian measure on $\X$ and $\delta\in E_2(\gamma)$. Then $\delta$ satisfies assumption $A1$ with $h_{\delta}(s)=\|\delta-\E_{\gamma}[\delta]\|_{L_2(\gamma)}(s+s^2)$.
 \end{corollaire}

 \begin{proof}
  It suffices to check the result for $\X=\R^p$ and to use a standard approximation argument. Recall that
 in $L_2(\gamma)$, we have $\X_{2,\gamma}^*=\{cte\}\oplus \X_{\gamma}^*\oplus E_2(\gamma)$. Also, there exists a unique triplet $\delta_0=\E_{\gamma}[\delta]\in \{cte\}$, $\delta_1\in \X_{\gamma}^*$ and $\delta_2\in E_2(\gamma)$ such that $\delta=\delta_0+\delta_1+\delta_2$.  
From the preceding corollary, assumption $A_1$ is satisfied for perturbation $\delta_2$, measure $P=\gamma$ and $h_{\delta_2}(s)=\|\delta_2\|_{L_2(\gamma)}(s+s^2)$. Because $\delta_1\in \X_{\gamma}^*$, $\delta_1$ is affine. Also, by Theorem \ref{thlipis}, the assumption $A_1$ is satisfied for perturbation $\delta_1$ with $h_{\delta_1}(s)=s\|\delta_1\|_{L_2(\gamma)}$. We can then conclude using Lemma \ref{lclaim1} and the fact that
\begin{align*}
\|\delta_2\|_{L_2(\gamma)}(s+s^2)+s\|\delta_1\|_{L_2(\gamma)}&\leq (\|\delta_1\|_{L_2(\gamma)}+\|\delta_2\|_{L_2(\gamma)})(s+s^2)\\
&\leq  \sqrt{2}(s+s^2)\|\delta-\delta_0\|_{L_2(\gamma)}.
\end{align*} 
\end{proof}
We now have all elements to demonstrate Theorem \ref{thedeux}. 
\subsection{Proof of Theorem \ref{thedeux}}\label{sduytzr}
As announced, we shall apply Theorem \ref{fonda}. From Theorem \ref{th:formquadra} Assumption $A_2$ is satisfied with $\beta=1/3$ in the case $1$ of our Theorem and for $\beta=2/7$ in the case $2$ of our Theorem. In both cases the constant $c_2$ depends on $r$ only. In both cases, from the preceding corollary, assumption $A_2$ is satisfied with the function  $h_{\delta}(s)=(s+s^2)\|\delta-\delta_0\|_{L_2(\gamma)}$. Also, if we apply Lemma \ref{fonda}, for all $q\in ]0,1[$, there exists a constant $C(r,q)>0$ such  that 
\[\gamma( V_f\Delta V_{f+\delta})\leq C(r,q)\left (|\E_{\gamma}(\delta)|+\|\delta-\E[\delta]\|_{L_2(\gamma)}\right )^{q\beta},\]
and a constant $C'(r,q)>0$ such that 
\[\gamma( V_f\Delta V_{f+\delta})\leq C'(r,q)\|\delta\|_{L_2(\gamma)}^{q\beta},\]
This ends the proof of the Theorem.

\subsection{Small crown probability}\label{pourhypA2}
In this subsection $\X_{2}^*$ is the set of real random variables that can be written $c+\sum_{i\geq 1}\beta_i (\xi_i^2-1)+\alpha_i\xi_i$ with  $c\in \R$, $\beta=(\beta_i)_i\in l_2(\N)$, $\alpha=(\alpha_i)_i\in l^2(N)$ $(\xi_i)_{i\in \N}$ is a sequence of independent identically distributed gaussian random variables with mean $0$ and variance $1$.
Let $q\in \X_{2}^*$ given by 
\[q=c+\sum_{i\geq 0}\alpha_i\xi_i+\sum_{i}\beta_i(\xi_i^2-1).\]
we will note
\begin{equation}\label{deftruc}
n_1(q)=\max_i |\alpha_i|\;\; \; n_2(q)=\max_i |\beta_i|, \;\; \sigma(q)=\left (\sum_{i\geq 0}2\beta_i^2+\alpha_i^2\right )^{1/2}.
\end{equation}

\begin{theoreme}\label{th:formquadra}

\begin{enumerate}
\item There exists $C(c_0)>0$ such that 
\[\sup \left \{P(|q|\leq \epsilon) \;:\; q\in \X_{2}^* \;:\;|\E[q]|\geq c_0\;\right \}\leq C(c_0)\epsilon^{2/7}.\]
\item There exists $C'(c_0)>0$ such that
\[\sup \left \{P(|q|\leq \epsilon) \;:\; q\in \X_{2}^* \;:\;E[q^2]\geq c_0\;\right \}\leq C'(c_0)\epsilon^{1/3}.\]
\item Let $q\in \X_{2}^*$,  for all $\epsilon\geq 0$,  
\[ P(|q|\leq \epsilon) \leq \sqrt{\frac{1}{\pi}\frac{\epsilon}{n_2(q)}}.\]
\end{enumerate}
\end{theoreme}
\begin{remarque}
This result may seem surprising, and we did not show it is optimal. If $n_2(q)=\max_{i}|\beta_i|>c_0$, the bound of point $3$ is optimal in the sense that if $\beta=(1,0,\dots)$, $c=1$ and $\alpha=0$ we get $P(|q|\leq \epsilon)=P(|\xi^2|\leq \epsilon)\sim C\epsilon^{1/2}$ (for a constant $C$ which can be calculated explicitly). In addition, when $\|\beta\|_{l^2}\rightarrow 0$ the behaviour of $P(|q|\leq \epsilon)$ tends to be the same as $P(|\|\alpha\|_{l^2} \mathcal{N}(0,1)-c|\leq \epsilon)\sim C'(c_0)\epsilon$. Also, it may be conjectured that points $1$ and $2$ of the Theorem can be improved (in order to obtain an exponent $1/2$ instead of $2/7$ and $1/3$) but we believe this is unlikely. The difficult cases to study (and point $3$ of the following proof demonstrate this) are those with $\|\beta\|_{\infty}\rightarrow 0$ but $\|\beta\|_{l^2}$ does not tend to zero. 
\end{remarque}
\begin{proof}

We shall proceed in four steps.\\
\indent \textit{Step 1}. We claim that if $|\E[q]|>\epsilon$ then 
\begin{equation}\label{eq:etp1}
P(|q|\leq \epsilon)\leq \frac{\sigma^2(q)}{(|\E[q]|-\epsilon)^2}.
\end{equation}

Notice that $|q-\E[q]|\geq ||q|-|\E[q]||$ and if $|q|<\epsilon<|\E[q]|$ then $||q|-|\E[q]||=|\E[q]|-|q|$ and
\[|q|\geq |\E[q]|-|q-\E[q]|.\]
Also
\[
P(|q|\leq \epsilon)\leq P(|\E[q]|-|q-\E[q]|\leq \epsilon) = P(1\leq \frac{|q-\E[q]|}{|\E[q]|-\epsilon})
\]
which implies (\ref{eq:etp1}) by the Markov inequality.

\indent \textit{Step 2}. We will assume without loss of generality that for all $i\in \N$ $\alpha_i\geq 0$. This is what we will do. In the following, $\alpha_{i_0}=\max_i\alpha_i$, $j_0\in arg\max|\beta_{j}|$ and $\sign(x)$  is the function that returns the sign of the real $x$. We claim that 
\begin{equation}\label{equ:etap11}
P(|q|\leq \epsilon)\leq \sqrt{\frac{1}{\pi}\frac{\epsilon}{n_2(q)}}.
\end{equation}

Let 
\[Z=\sum_{i\neq j_0}\alpha_{i}\xi_{i}+\beta_{i}(\xi_{i}^2-1). \]
To obtain the desired inequality, note that for all $\alpha_{j_0}\geq 0$, $\beta_{j_0}\neq 0$
\begin{align*}
P&\left (|Z+\alpha_{j_0}\xi+\beta_{j_0}(\xi^2-1)|\leq \epsilon\right )=P\left (|\sign(\beta_{j_0})Z+\alpha_{j_0}\xi+|\beta_{j_0}|(\xi^2-1)|\leq \epsilon\right )\\
&=P\left (|\frac{\sign(\beta_{j_0})Z}{|\beta_{j_0}|}+(\xi+\frac{\alpha_{j_0}}{2|\beta_{j_0}|})^2-1-\frac{\alpha_{j_0}^2}{4\beta_{j_0}^2})|\leq \frac{\epsilon}{|\beta_{j_0}|}\right )\\
&=P\left (\xi\in\left [f_{\alpha_{j_0},\beta_{j_0}}(-\epsilon)-\frac{\alpha_{j_0}}{2|\beta_{j_0}|};f_{\alpha_{j_0},\beta_{j_0}}(\epsilon)-\frac{\alpha_{j_0}}{2|\beta_{j_0}|}\right ]\right ).
\end{align*} 
where
\[f_{\alpha,\beta}(\epsilon)=\sqrt{(1+\frac{\alpha^2}{4\beta^2}-\frac{\sign(\beta)Z-\epsilon}{|\beta|})_+},\]
and $(x)_+=x\1_{x\geq 0}$.
The inequality (\ref{equ:etap11}) results from the choice $\alpha=\alpha_{j_0}$ and $\beta=\beta_{j_0}$

and from the fact that if $u\in \R$, $\sqrt{(u+\frac{\epsilon}{|\beta_{j_0}|})_+}-\sqrt{(u-\frac{\epsilon}{|\beta_{j_0}|})_+}\leq \sqrt{\frac{2\epsilon}{n_2(q)}}$. \\

 \textit{Step 3} We claim that  
 \begin{equation}\label{froyutr}
 P(|q|\leq \epsilon)\leq 208 \frac{ n_2(q)}{\sigma(q)}+ \frac{2\epsilon}{\sigma(q)}e^{-\frac{(|\E[q]|-\epsilon)^2}{\sigma^2(q)}}.
 \end{equation}
 We prove the following lemma (which is a central limit theorem) at the end of the proof. 
\begin{Lemme}\label{lemme:teplo}
 Let $X_i=\beta_i(\xi_i^2-1)+\alpha_i\xi_i$, $\xi$ be a gaussian centered random variable with variance $1$ and $\sigma(q)$ given by (\ref{deftruc}). We obtain:
  \[
 \sup_{\epsilon \geq 0} \left |P \left (|\E_{\gamma}[q]+\sum_{i\geq 0}X_i|\leq \epsilon \right )-P\left ( |\xi+\frac{\E_{\gamma}[q]}{\sigma(q)}| \leq \frac{\epsilon}{\sigma(q)}\right )\right | \leq 104\frac{\max(|\beta_i|)}{\sigma(q)}.
 \]
 \end{Lemme}
 
Also, because $|\E[q]|>\epsilon$
\[P\left (|\xi+\frac{\E[q]}{\sigma(q)}|\leq \frac{\epsilon}{\sigma(q)}\right )\leq \frac{2\epsilon}{\sigma(q)}e^{-\frac{(|\E[q]|-\epsilon)^2}{\sigma^2(q)}}, \]
we have inequality (\ref{froyutr}).\\
\indent \textit{Step 4}.
As announced we will distinguish several disjoint cases to demonstrate points $1$ and $2$ of the theorem. We begin with point $1$. \begin{enumerate}
\item In the case where $\sigma(q)<\epsilon^{1/7}$, it is the inequality from step $1$ (\ref{eq:etp1}) that leads to the desired conclusion.
\item In the case where $n_2(q)\geq \epsilon^{3/7}$, it is the inequality from step $2$ (\ref{equ:etap11}) that leads to the desired conclusion.
\item In the case where $n_2(q)<\epsilon^{3/7}$ and $\sigma(q)>\epsilon^{1/7}$, it is the inequality from step $3$ (\ref{froyutr}) that leads to the desired conclusion.
\end{enumerate}

We conclude with point $2$. 
\begin{enumerate}
\item In the case where $n_2(q)\geq \epsilon^{1/3}$, it is the inequality from step $2$ (\ref{equ:etap11}) that leads to the desired conclusion.
\item In the case where $n_2(q)<\epsilon^{1/3}$ it is the inequality from step $3$ (\ref{froyutr}) that leads to the desired conclusion.
\end{enumerate}
\end{proof}
We now give the proof of theorem \ref{lemme:teplo}. 
\begin{proof}
This proof is decomposed into two steps. In the first step, we calculate 
\begin{equation}\label{etapodi}
\forall \alpha,\beta\in \R,\;\; \phi_{\alpha,\beta}(t)=\E\left [e^{it(\xi\alpha+\beta(\xi^2-1))}\right ],
\end{equation}
and in the second one we deduce that for all $|t|<\frac{\sigma}{6\max_j|\beta_j|}=a$ 
\begin{equation}\label{etoiru}
|\prod_{j\geq 0}\phi_{\alpha_j,\beta_j}(t/\sigma)-e^{-t^2/2}|\leq \frac{4\max_j|\beta_j|}{\sigma}\frac{|t|^{3}}{2}e^{-t^2/6},
\end{equation}
which implies the desired result from the Essen inequality (see for example \cite{Shorack:2000yp} p358)
\begin{align*}
\sup_{u\in \R}&\left |P\left (\frac{1}{\sigma}\sum_{j\geq 0}\alpha_j\xi_j+\beta_j(\xi_j^2-1)\geq u\right )-\Phi(u)\right |\\
& \leq  \int_{-a}^{a}\left |\frac{\prod_{i\geq 0}\phi_{\alpha,\beta}(t/\sigma)-e^{-t^2/2}}{t}\right |dt+\frac{24}{a\sqrt{2\pi}}\\
&\leq \frac{4\max_j|\beta_j|}{\sigma}\int_{\R}\frac{t^2}{2}e^{-\frac{t^2}{6}}dt+\frac{\max_j|\beta_j|72\sqrt{2}}{\sigma\sqrt{\pi}}\\
&=\frac{\max_j|\beta_j|}{\sigma}\left (72\sqrt{\frac{2}{\pi}}+32\right )\leq 104\frac{\max_j|\beta_j|}{\sigma},
\end{align*}
where $\Phi$ is the cumulative distribution function of a standardised gaussian real random variable.
 \\
\textit{Step 1.} 
Let $\Omega_{\beta}=\{z\in \mathbb{C}\;\; 2\Im(z)\beta>-1\}$ and  $\psi_{\alpha,\beta}(z)$ be given by
\[\forall \alpha,\beta\in \R,\;\; z\in\omega_{\beta}\;\; \psi_{\alpha,\beta}(z)=\frac{e^{-\beta i z}}{(1-2\beta i z)^{1/2}}e^{-1/2\frac{\alpha^2z^2}{(1-2\beta i z)}} .\]
The function $\psi_{\alpha,\beta}$ is analytic on $\Omega_{\beta}$. The function $\phi_{\alpha,\beta}(t)$ defined by (\ref{etapodi}) can be continued into an analytic function on the domain  $\Omega_{\beta}$ and because

\[\frac{x^2}{2}+y(\alpha x+\beta (x^2-1))=\frac{1}{2}(1+2\beta y)(x+\frac{\alpha y}{1+2\beta y})^2-\frac{\alpha^2 y^2}{2(1+2\beta y)}\]
we observe that
\[\forall y>-\frac{1}{2\beta}\;\;\; \psi_{\alpha,\beta}(iy)=\phi_{\alpha,\beta}(iy).\]
Also, we can deduce that $\phi_{\alpha,\beta}(z)$ and $\psi_{\alpha,\beta}(z)$ are equal on $\Omega_{\beta}$ and in particular on $\R$ which gives
\[
\forall \alpha,\beta\in \R,\;\; t\in \R\;\; \phi_{\alpha,\beta}(t)=\frac{e^{-\beta i t}}{(1-2\beta i t)^{1/2}}e^{-1/2\frac{\alpha^2t^2}{(1-2\beta i t)}}.
\]
\textit{Step 2.} Proof of (\ref{etoiru}).
The preceding equation gives
\[|\prod_{i\geq 0}\phi_{\alpha,\beta}(t/\sigma)-e^{-t^2/2}|=e^{-\frac{t^2}{2}}|e^{z}-1|\leq e^{-\frac{t^2}{2}} |z|e^{z},\]
where 
\[u=\frac{t}{\sigma}\;\text{ et } \;z=\frac{t^{2}}{2}+\sum_{j\geq 0}\left \{-1/2\frac{\alpha_j^2u^2}{(1-2\beta_j i u)}+\frac{1}{2}(-2\beta_jui-\log(1-2\beta_jui))\right \},\]
and hence 
\begin{equation}\label{toutcru}
z=\sum_{j\geq 0}\left \{\left (\frac{u^{2}\alpha_j^2}{2}-\frac{1}{2}\frac{\alpha_j^2u^2}{(1-2\beta_j i u)}\right ) +\left (\frac{u^{2}2\beta_j^2}{2}-\frac{1}{2}(2\beta_jui+\log(1-2\beta_jui))\right )\right \}.
\end{equation}
In addition, if $|t|<\frac{\sigma}{6\max_i|\beta_i|}$, then for all  $j\in \N$  $|2u\beta_j|<\frac{1}{3}$  and we have (cf Taylor expansion  (1) p352 in \cite{Shorack:2000yp} )
\[|\log(1-2\beta_jui)+2\beta_jui-\frac{4\beta_j^2u^{2}}{2}|\leq \frac{8|u\beta_j|^3}{3}\left |\frac{1}{1-|2u\beta_j|}\right |\leq 4|u\beta_j|^2\max_{j}|\beta_j|. \]
We also have
 \[|\frac{u^{2}\alpha_j^2}{2}-\frac{1}{2}\frac{\alpha_j^2u^2}{(1-2\beta_j i u)}|\leq \frac{1}{2}\alpha_j^2|u|^ 3 \frac{2|\beta_j|}{1+4\beta_j^2 u^2}\leq \alpha_j^2|u|^3\max_{j}|\beta_j|.\]
As a consequence, if $|t|<\frac{\sigma}{6\max_i|\beta_i|}$, then (\ref{toutcru}) implies:
 \[|z|\leq 2\sigma^2|u|^3\max_{j}|\beta_j|=\frac{2\max_{j}|\beta_j|}{\sigma}|t|^3, \]
 and
 \[e^{-\left( \frac{t^2}{2}-|z|\right )}\leq e^{-\frac{t^2}{2}(1-\frac{2}{3})}=e^{-\frac{t^2}{6}}.\]
\end{proof}

\section*{Acknowledgements}
This work has been done with support from La Region Rhones-Alpes. 

\bibliography{../../../biblio/biblio}
\bibliographystyle{plain}
\end{document}